\documentclass[a4paper]{amsart}
\usepackage{color}
\usepackage[latin1]{inputenc}
\usepackage[T1]{fontenc}
\usepackage{amsfonts}
\usepackage{amssymb}
\usepackage{amsmath}
\usepackage{amsthm}
\usepackage{stmaryrd}
\usepackage{enumerate}
\usepackage{multirow}
\usepackage{graphicx}
\usepackage{graphicx,type1cm,eso-pic,color}
\usepackage{pstricks} 
\usepackage{float}
\newtheorem{theorem}{Theorem}[section]
\newtheorem{lemma}[theorem]{Lemma}
\newtheorem{proposition}[theorem]{Proposition}
\newtheorem{corollary}[theorem]{Corollary}

\newtheorem{assumption}[theorem]{Assumption}
\newtheorem{remark}[theorem]{Remark}
\begin{document}
\setlength\arraycolsep{2pt}
\title[Estimation of the derivative of a regression]{On a Projection Estimator of the Regression Function Derivative}
\author{Fabienne COMTE*}
\address{*Universit\'e de Paris, CNRS, MAP5 UMR 8145, F-75006 Paris, France}
\email{fabienne.comte@parisdescartes.fr}
\author{Nicolas MARIE$^{\dag}$}
\address{$^{\dag}$Laboratoire Modal'X, Universit\'e Paris Nanterre, Nanterre, France}
\email{nmarie@parisnanterre.fr}
\date{}
\maketitle
\noindent
%


\begin{abstract}
In this paper, we study the estimation of the derivative of a regression function in a standard univariate regression model. The estimators are defined either by derivating nonparametric least-squares estimators of the regression function or by estimating the projection of the derivative. 
We prove two simple risk bounds allowing to compare our estimators. More elaborate bounds under a stability assumption are then provided.  Bases and spaces on which we can illustrate our assumptions and first results are both of compact or non compact type, and we discuss the rates reached by our estimators. They turn out to be optimal in the compact case. Lastly, we propose a model selection procedure and prove the associated risk bound. To consider bases with a non compact support makes the problem difficult.
\end{abstract}



%
\section{Introduction}

In this paper, we consider the random design regression model
\begin{equation}\label{model}
Y_i = b(X_i) +\varepsilon_i
\textrm{ $;$ }i\in\{1,\dots,n\}
\textrm{, }n\geqslant 1,
\end{equation}
where $b(.)$ is the unknown continuously differentiable regression function, $X_1,\dots,X_n$ are independent and identically distributed (i.i.d.) random variables with density $f$ with respect to Lebesgue's measure, and the errors $\varepsilon_1,\dots,\varepsilon_n$ are i.i.d, unobserved, centered with variance $\sigma^2$, and independent of the $X_i$'s. The observations are $(X_i, Y_i)_{1\leqslant i\leqslant n}$, and we assume that $b$ is regular enough to admit a derivative. We are interested in nonparametric estimation of {\it the derivative} $b'$ of $b$, on a compact or a non-compact support.

\subsection{Motivation and bibliographical elements}
The question of nonparametric estimation of derivatives is not new and is studied in different contexts, such as density estimation or white noise model (see Efromovich \cite{EFROMOVICH98}), and not only in regression. Indeed, there can be a lot of reasons for estimating not only a function but also its derivative, which may be of  intrinsic interest as measure of slope for instance.
Recently, Bercu {\it et al.} \cite{BCD19} studied this question in the concrete application setting of sea shores water quality. Precisely, they propose an estimator defined as the derivative of the well-known Nadaraya-Watson estimator. Dai {\it et al.} \cite{DTG16} also mention applications to the modeling of human growth data (Ramsay and Silverman \cite{RS02}) or to Raman spectra of bulk materials (Charnigo {\it et al.} \cite{Char11}).
\\
\\
Derivatives of a rate optimal estimate of the regression function are proved to be rate optimal estimates of the corresponding derivatives, see Stone \cite{STONE80,STONE82}, who establishes optimal rates for local polynomial weighted estimators on a compact set. See also a discussion on the topic in Rice and Rosenblatt \cite{RR83}, for a fixed design model.\\
Nonparametric estimation of the regression function derivative has been studied following different methods, relying on kernels, local polynomial regression, regression by smoothing splines, or difference quotients. We emphasize that the strategy for fixed design context, where $X_i$ are replaced by $x_i = i/n$, relies on dedicated methods. Indeed, differences $Y_i - Y_{i-1}$ bring information on $b'$, which is not the case for random design on non compact support. Kernel estimation of the regression function and its derivative is for instance studied by Gasser and M\"{u}ller \cite{GasserMuller}, in the fixed design case. In local polynomial regression, the derivative can be estimated by the coefficient of the derivative of the local polynomial regression fitted at given point $x$, as summarized in Tsybakov \cite{TSYBAKOV09}, Chapter 1, see also Huang and Chan~\cite{HuangChan}. Stone \cite{STONE82} showed that derivative  estimation with splines can achieve the optimal ${\mathbb L}^2$-rate  of convergence (proved in Stone~\cite{STONE80}) under mild assumptions. Further asymptotic properties are obtained by Zhou and Wolfe \cite{ZhouWolfe}, mainly in  the  fixed  design  setting on compact support: they rely on splines estimators, arguing that they avoid boundary problems of kernel estimators. Note that extensions to functional regressors have been conducted (see Hall {\it et al.} \cite{HMY}).\\
The  smoothing parameter selection problem remained unanswered in the first papers. For kernel strategies, the bandwidth choice for the derivative estimator (based on a factor rule) is discussed in Fan and Gijbels \cite{FanGij}, but not studied from theoretical point of view. Liu and Brabanter \cite{LiuBra} propose a methodology which generalizes the difference quotient based estimator of the first order derivative to the random design setting, when $X$ follows a uniform distribution on $[0,1]$. They also discuss bandwidth selection in their setting. Lastly, we mention that an adaptive method in wavelet bases is studied in Chesneau \cite{Chesneau14}, but it involves an estimate of $f$. As a consequence, the rate of estimation depends on the regularity of this function, which we want to avoid.

\subsection{Contributions of the paper}
In the present work, we consider a projection method and propose an estimator as a finite development in an orthonormal ${\mathbb L}^2$ basis with $m$ coefficients. We start from the least-square estimator studied by Baraud for the fixed design model in \cite{BARAUD00} and the random design model in \cite{BARAUD02}. These works  consider compactly supported bases and assume that the density of the $X_i$ is lower bounded on the interval of estimation. The lower bound on the density is involved in the upper bound on the risk. These results have been extended to non compactly supported bases by Comte and Genon-Catalot~\cite{CGC20}; then, the assumption that the density is lower bounded can not be done, and the problem has to be handled differently. In some sense, regression function estimation in this setting has some characteristics of inverse problems. 

Here, we show that two strategies can be considered to deduce from the least square estimator of $b$, an estimator of $b'(.)$, and these strategies do not coincide in general. We prove non asymptotic bounds on the integrated ${\mathbb L}^2-$risk of the estimators, for both strategies. The fact that our results are non asymptotic and global (and not pointwise), make them different from the literature mentioned previously. To our knowledge, these are the first results allowing for non compactly supported bases in the definition of the estimators. In the case of a trigonometric basis and compact support estimation, we recover the optimal rates given in Stone \cite{STONE80} under weak assumptions. We also obtain specific rates in the non compact Hermite basis setting. Therefore, our results contain previous ones, and extensions. 
Last but not least, we propose a model selection strategy relying on a Goldenshluger and Lepski \cite{GL2011} method and prove a risk bound for the final estimator: this result holds for sub-gaussian (bounded or gaussian) errors and implies that the estimator automatically reaches the optimal rate on regularity spaces, without requiring the knowledge of the regularity index of $b$. We discuss our assumptions, which remain rather weak.
\\
\\
The plan of the paper is the following. We define our notation and estimators in Section \ref{section_estimators_assumptions}. In Section \ref{Risk_bounds}, we present our assumptions and prove two simple risk bounds allowing to compare our estimators. More elaborate bounds under a stability assumption (see Cohen {\it et al.} \cite{CDL13, CDL19}) are also provided.  Bases and spaces on which we can illustrate our assumptions and first results are described in Section \ref{section_bases}. They are of compact (trigonometric basis) or non compact (HErmite basis) type, and we discuss the rates reached by our estimators. They are the optimal ones in the compact case. Section \ref{GLAdapt} is dedicated to the adaptive procedure: we prove a risk bound and deduce corollaries about adaptive rates. The possibility of non compact support makes the problem difficult, and even if the estimator seems to follow a standard Goldenshluger and Lepski \cite{GL2011} scheme, the proofs are delicate, due to an additional bias term. A numerical study shows that the collection of estimators contains relevant proposals and that the data driven  estimator works in a satisfactory way, especially compared to the derivative of a Nadaraya-Watson estimator. 
%


%
\section{Definition of the estimators}\label{section_estimators_assumptions}
Let $\mathcal B = (\varphi_j)_{j\in\mathbb N\backslash\{0\}}$ be a Hilbert basis of $\mathbb L^2(I,dx)$ with $I\subseteqq \mathbb R$ an interval. For the sake of readability, for every $j\in\mathbb N\backslash\{0\}$, the function $x\in\mathbb R\mapsto\varphi_j(x)\mathbf 1_I(x)$ is also denoted by $\varphi_j$. The following mean squares estimator   of $b_I = b\mathbf 1_I$, which is studied in Baraud \cite{BARAUD02} and in Comte and Genon-Catalot \cite{CGC20}, is defined by 
\begin{displaymath}
\widehat b_m(x) :=
\sum_{j = 1}^{m}[\widehat\theta_{m}^{1}]_j\varphi_j(x)
\textrm{ $;$ }
x\in I,
\end{displaymath}
where $m\in\mathbb N\backslash\{0\}$,
\begin{displaymath}
\widehat\theta_{m}^{1} =
\widehat\theta_{m}^{1} (\mathbf X,\mathbf Y) :=
\frac{1}{n}\widehat\Psi_{m}^{-1}\widehat\Phi_{m}^{*}\mathbf Y,
\end{displaymath}
$M^*$ denotes the transpose of $M$, $\mathbf Y := (Y_1,\dots,Y_n)^*$, $\mathbf X := (X_1,\dots,X_n)^*$, $\widehat\Phi_m := (\varphi_j(X_i))_{1\leqslant i\leqslant n, 1\leqslant j\leqslant m}$ and
\begin{displaymath}
\widehat\Psi_m :=
\frac{1}{n}\widehat\Phi_{m}^{*}\widehat\Phi_m =
(\langle\varphi_j,\varphi_k\rangle_n)_{1\leqslant j,k\leqslant m }
\end{displaymath}
with
\begin{displaymath}
\langle\varphi,\psi\rangle_n :=
\frac{1}{n}\sum_{i = 1}^{n}\varphi(X_i)\psi(X_i)
\end{displaymath}
for every $\varphi,\psi :\mathbb R\rightarrow\mathbb R$. The map $(\varphi,\psi)\mapsto\langle\varphi,\psi\rangle_n$ is the empirical scalar product, and the associated norm is denoted by $\|.\|_n$ in the sequel. The theoretical analogue on $\mathbb L^2(\mathbb R,f(x)dx)$ is
\begin{displaymath}
(\varphi,\psi)\longmapsto\langle\varphi,\psi\rangle_f
:=\int_I\varphi(z)\psi(z)f(z)dz,
\end{displaymath}
and the associated norm is denoted by $\|.\|_f$. Notice that ${\mathbb E}( \langle\varphi,\psi\rangle_n ) = \langle\varphi,\psi\rangle_f$. The reader can refer to Baraud \cite{BARAUD00,BARAUD02}, Cohen et al. \cite{CDL13,CDL19}, and Comte and Genon-Catalot \cite{CGC20} for risk bounds on $\widehat b_m$ and an adaptive estimator.
\\
\\
\noindent {\bf Strategy 1.} On the one hand, a natural estimator of $b_I'$ is
\begin{equation}\label{first_estimator}
\widehat b_{m}^{\prime,1}(x) :=
\sum_{j = 1}^{m}[\widehat\theta_{m}^{1}]_j\varphi_j'(x)
\end{equation}
with $m\in\mathbb N\backslash\{0\}$. Obviously,
\begin{displaymath}
(\widehat b_{m}^{\prime,1}(X_1),\dots,\widehat b_{m}^{\prime,1}(X_n))^* =
\widehat\Phi_m' \widehat\theta_{m}^{1} =
\frac{1}{n}
\widehat\Phi_m'\widehat\Psi_{m}^{-1}\widehat\Phi_{m}^{*}\mathbf Y
\end{displaymath}
with $\widehat\Phi_m' := (\varphi_j'(X_i))_{i,j}$. This requires to choose a regular basis. Note that, contrary to what may occur for the density estimator,  this way is simpler than derivating the Nadaraya-Watson kernel based estimator as done in Bercu {\it et al.} \cite{BCD19}. Indeed, the latter involves the derivative of a quotient of two functions.
\\
\\
\noindent {\bf Strategy 2.} On the other hand, when $(b\varphi_j)(\inf(I)) = (b\varphi_j)(\sup(I))$ for every $j\in\{1,\dots,m\}$, $ \langle b',\varphi_j\rangle =-\langle b,\varphi'_j\rangle$ and the orthogonal projection $(b')_m$ of $b'$ on $\mathcal S_m :=\textrm{span}\{\varphi_1,\dots,\varphi_m\}$ in $\mathbb L^2(I,dx)$ satisfies
\begin{displaymath}
(b')_m(x) =
-\sum_{j = 1}^{m}\langle b,\varphi_j'\rangle\varphi_j(x).
\end{displaymath}
Several of the basis we have in mind are such that the derivative of $\varphi_j$ can be expressed as a finite linear combination of the other $\varphi_k$'s. Thus, if there exist (known) coefficients $d_{j,k}$ such that $\varphi_j'=\sum_{k=1}^{m} d_{j,k}\varphi_k$, then 
$$\langle b,\varphi_j'\rangle = \sum_{k=0}^m d_{j,k}\langle b,\varphi_k\rangle.$$
A simple plug-in strategy leads thus to propose an estimate of $\langle b,\varphi_j'\rangle$ by replacing $\langle b,\varphi_k\rangle$ in the above formula by $[\widehat\theta_m^1]_k$.
 
In other words, if there exists $\Delta_{m,m + p}\in\mathcal M_{m,m + p}(\mathbb R)$ such that $\widehat\Phi_m' =\widehat\Phi_{m + p}\Delta_{m,m + p}^{*}$, one can consider a projection estimator of the derivative instead of derivating the projection estimator of $b$:
\begin{equation}\label{second_estimator}
\widehat b_{m}^{\prime,2}(x) :=
\sum_{j = 1}^{m}[\widehat{\theta}_{m}^{2}]_j\varphi_j(x)
\end{equation}
with
\begin{displaymath}
\widehat\theta_{m}^{2} =
-\frac{1}{n}\Delta_{m,m + p}\widehat\Psi_{m + p}^{-1}\widehat\Phi_{m + p}^{*}\mathbf Y.
\end{displaymath}
Obviously,
\begin{displaymath}
(\widehat b_{m}^{\prime,2}(X_1),\dots,\widehat b_{m}^{\prime,2}(X_n))^* =
-\frac{1}{n}
\widehat\Phi_{m + p}
\Delta_{m,m + p}
\widehat\Psi_{m + p}^{-1}\widehat\Phi_{m + p}^{*}\mathbf Y.
\end{displaymath}
We shall see in this paper that the two strategies are different and we will provide risk bounds that allow to compare the two methods.

\section{Risk bounds}\label{Risk_bounds}

%
\subsection{Notations and useful elementary properties:}\label{section_notations}
\begin{itemize}
 \item The operator norm of a matrix $M$ is defined by $\|M\|_{\textrm{op}}^{2} :=\lambda_{\max}(MM^*)$, where we recall that $M^*$ is the transpose of $M$ and $\lambda_{\max}(MM^*)$ is the largest eigenvalue of the square matrix $MM^*$, which are nonnegative. Note that for a square, symmetric and nonnegative matrix $A$, $\|A\|_{\textrm{op}} = \lambda_{\max}(A)$. Note also that if $A$ and $B$ are two matrices such that $AB$ and $BA$ are well defined, then $\lambda_{\max}(AB) =\lambda_{\max}(BA)$. Finally, note that if $A$ and $B$ are two square, symmetric and nonnegative matrices, then $\textrm{Tr}(AB)\leqslant\|A\|_{{\rm op}}\textrm{Tr}(B) =\lambda_{\max}(A)\textrm{Tr}(B)$, where $\textrm{Tr}(M)$ denotes the trace of a (square) matrix $M$.
 \item The Frobenius norm of a matrix $M$ is defined by
 \begin{displaymath}
 \|M\|_{F}^{2} :=\textrm{Tr}(MM^*) =\textrm{Tr}(M^*M).
 \end{displaymath}
 \item The natural scalar product on $\mathbb L^2(I,f(x)dx)$, also called $f$-weighted scalar product, is denoted by $\langle .,.\rangle_f$, and the associated norm by $\|.\|_f$.
 \item For every $\psi\in\mathbb L^2(I,dx)$, its orthogonal projection on $\mathcal S_m =\{\varphi_1,\dots,\varphi_m\}$ in $\mathbb L^2(I,dx)$ is denoted by $\psi_m$.
\end{itemize}
%


%

%


%
\subsection{Preliminary rough risk bounds on $\widehat b_{m}^{\prime,1}$ and $\widehat b_{m}^{\prime,2}$}
In the sequel, we assume that $b'$ exists and is square integrable on $I$, and that the density function $f$ fulfills the following assumption.
%


%
\begin{assumption}\label{assumption_f}
The density function $f$ is bounded on $I$.
\end{assumption}
\noindent
First, we provide the following rough but general risk bound on $\widehat b_{m}^{\prime,1}$.
%


%
\begin{proposition}\label{rough_risk_bound_estimator_1}
Under Assumption \ref{assumption_f},
\begin{eqnarray*} 
 \mathbb E\left[\|\widehat b_{m}^{\prime,1} - b'\|_{n}^{2}\right]
 & \leqslant &
 3\|f\|_{\infty}\inf_{t\in\mathcal S_m}\|t - b'\|^2 +
 3\mathbb E\left[\|\widehat\Phi'_m(\widehat\Phi_m^*\widehat\Phi_m)^{-1}\widehat\Phi_m^*\|_{{\rm op}}^{2}
 \|b - b_m\|_{n}^{2}\right] +
 3\|b_m' - (b')_m\|_{f}^{2}\\
 & & \hspace{3cm} +
 \frac{\sigma^2}{n}
 \mathbb E\left[{\rm Tr}
 \left((\widehat\Phi_m^*\widehat \Phi_m)^{-1}(\widehat\Phi_m')^*\widehat\Phi_m'\right)\right],
\end{eqnarray*}
where $b_m$ is the ${\mathbb L}^2(I,dx)$-orthogonal projection of $b$ on ${\mathcal S}_m$, and $b_m'$ is its derivative, while $(b')_m$ is the ${\mathbb L}^2(I,dx)$-orthogonal projection of $b'$ on ${\mathcal S}_m$.
\end{proposition}
\noindent
Let us comment the four terms in the previous bound:
\begin{enumerate}
 \item The first term is the bias term we could expect. It can be evaluated on regularity spaces. Without Assumption \ref{assumption_f}, this terms can be replaced by $\inf_{t\in\mathcal S_m}\|t - b'\|_{f}^{2}$.
 \item The second term involves the bias related to $b$, which would be negligible compared to the previous one; but it is multiplied by a coefficient which has an order depending on $m$ and will at least compensate the improvement.
 \item The third term can be evaluated in the different bases: the procedure makes sense if the derivative of the projection and the projection of the derivative are close, for fixed $m$. Under Assumption \ref{assumption_f}, it is less than $3\|f\|_{\infty}\|b_m' - (b')_m\|^2$, null in trigonometric spaces with odd dimensions, and of order less or equal than the first term in Laguerre or Hermite bases (see Proposition \ref{additional_term_risk_bound_estimator_1}).
 \item The last term is the variance term, and it is established in Proposition \ref{increase_variance} that it increases with $m$ as expected.
\end{enumerate}
%


%
\begin{proposition}\label{increase_variance}
The map $m\mapsto {\mathbb E}\left[{\rm Tr}\left((\widehat\Phi_m^*\widehat \Phi_m)^{-1}(\widehat\Phi_m')^*\widehat\Phi_m'\right)\right]$ is increasing.
\end{proposition}

To sum up, we will have to make a compromise between decreasing bias term (1) and increasing variance term (4), with the specific difficulty related to nuisance terms (2) and (3).

Now, we turn to the estimator $\widehat b_{m}^{\prime,2}$ and assume that there exists $p\in\mathbb N$ such that $\varphi_1,\dots,\varphi_m$ fulfill the following assumption:
%


%
\begin{assumption}[$m,p$]\label{assumption_basis_derivatives}
For every $j\in\{1,\dots,m\}$, $\varphi_j'\in\mathcal S_{j + p}$.
\end{assumption}
\noindent
Note that Assumption \ref{assumption_basis_derivatives}($m,p$) implies that there exists $\Delta_{m,m + p}\in\mathcal M_{m,m + p}(\mathbb R)$ such that
\begin{equation}\label{defDelta}
\widehat\Phi_m' =\widehat\Phi_{m + p}\Delta_{m,m + p}^{*}.
\end{equation}
Trigonometric, Laguerre, Hermite and Legendre bases satisfy Assumption \ref{assumption_basis_derivatives}($m,p$), see Section \ref{section_bases}. More precisely, we have $p = 0$ for the Laguerre and Legendre bases and $\Delta_{m,m}$ is a lower triangular square matrix. We have $p = 1$ for Hermite and trigonometric bases with $\Delta_{m,m + 1}(j,k) = 0$ for $k\geqslant j + p$. For the trigonometric basis with an odd dimension, we can keep a square link $\Delta_{m,m}$ with a null first line followed by diagonal $2\times 2$ blocks of type
\begin{displaymath}
\begin{pmatrix}
 0 & -2\pi j\\
 2\pi j & 0
\end{pmatrix}.
\end{displaymath}
Assume that $b$ and $\varphi_1,\dots,\varphi_m$ fulfill also the following assumption.
%


%
\begin{assumption}[$m$]\label{assumption_bounds_I}
For every $j\in\{1,\dots,m\}$,
\begin{displaymath}
b(\underline{{\tt a}})\varphi_j(\underline{{\tt a}}) = b(\overline{{\tt a}})\varphi_j(\overline{{\tt a}}),
\end{displaymath}
where $\underline{{\tt a}} :=\inf(I)$ and $\overline{{\tt a}} :=\sup(I)$.
\end{assumption}
\noindent
Note that, for instance, Assumption \ref{assumption_bounds_I}($m$) holds for every $m\in\mathbb N$ when $b(\underline{{\tt a}}) = b(\overline{{\tt a}}) = 0$. Under this additional condition, by the integration by parts formula,
\begin{equation}\label{IPP}
\langle b',\varphi_j\rangle = -\langle b,\varphi_j'\rangle
\textrm{ $;$ }
\forall j\in\{1,\dots,m\}.
\end{equation}
So,
\begin{equation}\label{projection_derivative}
(b')_m(\mathbf X) =
-\sum_{j = 1}^{m}
\langle b,\varphi_j'\rangle\varphi_j(\mathbf X) =
-\widehat\Phi_m\Delta_{m,m+p}\left(\langle b,\varphi_j\rangle\right)_{1\leqslant j\leqslant m + p},
\end{equation}
which legitimates the definition (\ref{second_estimator}) of the alternative estimator $\widehat b_{m}^{\prime,2}$ of $b'$. Let us establish a risk bound for this estimator.
%


%
\begin{proposition}\label{rough_risk_bound_estimator_2}
Under Assumptions \ref{assumption_f}, \ref{assumption_basis_derivatives}($m,p$) and \ref{assumption_bounds_I}($m$),
\begin{eqnarray*} 
 \mathbb E\left[\|\widehat b_{m}^{\prime,2} - b'\|_{n}^{2}\right]
 & \leqslant &
 2\|f\|_{\infty}\inf_{t\in\mathcal S_m}\|t - b'\|^2
 + 2\mathbb E\left[
 \|\widehat\Phi_m\Delta_{m,m + p}(\widehat\Phi_{m + p}^{*}\widehat\Phi_{m + p})^{-1}
 \widehat\Phi_{m + p}^{*}\|_{{\rm op}}^2
 \|b - b_{m + p}\|_{n}^{2}\right]\\
 & &
 \hspace{3cm} +\frac{\sigma^2}{n}\mathbb E\left[{\rm Tr}\left((\widehat\Phi_{m + p}^{*}
 \widehat\Phi_{m + p})^{-1}\Delta_{m,m + p}^{*}
 \widehat\Phi_{m}^{*}\widehat\Phi_{m}\Delta_{m,m + p}\right)\right].
\end{eqnarray*}
\end{proposition}
\noindent Now we have elements to compare the two estimators.

\noindent {\bf Comparison of the two estimators.} Note that for $p = 0$, this bound is almost the same as in Proposition \ref{rough_risk_bound_estimator_1}, except that the undesirable term $\|b_m' - (b')_m\|_{f}^{2}$ no longer appears. The counterpart is that the result of Proposition \ref{rough_risk_bound_estimator_2} requires the additional Assumptions \ref{assumption_basis_derivatives} and \ref{assumption_bounds_I}. Thanks to Proposition \ref{additional_term_risk_bound_estimator_1} (see Section \ref{section_bases} for details):
\begin{itemize}
 \item In the specific case of the trigonometric basis, the additional term $\|b_m' - (b')_m\|_{f}^{2}$ in the bound of Proposition \ref{rough_risk_bound_estimator_1} is null, and the first estimator requires less assumptions, so the first strategy is better.
 \item In the case of the Hermite basis, Assumption \ref{assumption_basis_derivatives}($m,p$), $p=1$ and Assumption \ref{assumption_bounds_I}($m$) are automatically fulfilled. However, it is difficult to determine which strategy is better.
 \item In the Laguerre basis, Assumption \ref{assumption_bounds_I}($m$) is satisfied for all $m$ if $b(0)=0$. If this holds, it follows from Proposition \ref{additional_term_risk_bound_estimator_1} (iii) that both strategies give the same rate.
 \item In the case of the Legendre basis, the additional term $\|b_m' - (b')_m\|_{f}^{2}$ is likely to be large, so the second strategy should be preferred.
\end{itemize}
%


%
\subsection{Elaborate risk bounds on $\widehat b_{m}^{\prime,1}$ and $\widehat b_{m}^{\prime,2}$}
First of all, under Assumption \ref{assumption_f}, let us consider
\begin{displaymath}
\Psi_m :=\mathbb E(\widehat\Psi_m) = (\langle\varphi_j,\varphi_k\rangle_f)_{j,k}.
\end{displaymath}
Assume also that it fulfills the following assumption called "stability assumption" by Cohen {\it et al.} \cite{CDL13}.
%


%
\begin{assumption}[\textit m]\label{assumption_Psi}
The matrix $\Psi_m$ satisfies
\begin{displaymath}
\mathfrak L(m) :=
\sup_{x\in I}\sum_{j = 1}^{m}\varphi_j(x)^2 <\infty
\quad
\textrm{and}\quad
\mathfrak L(m)(\|\Psi_{m}^{-1}\|_{{\rm op}}\vee 1)\leqslant\frac{\mathfrak c}{2}\cdot\frac{n}{\log(n)},
\end{displaymath}
where $\mathfrak c = (3\log(3/2) - 1)/9$.
\end{assumption}
\noindent
Since the $\varphi_j$'s do not depend on $m$, the $\mathcal S_m$'s are nested spaces. Thus,  since $m\mapsto\mathfrak L(m)$ and $m\mapsto\|\Psi_{m}^{-1}\|_{\textrm{op}}$ are increasing, if there exists $m_0\in\mathbb N\backslash\{0\}$ such that Assumption \ref{assumption_Psi}($m_0$) is fulfilled, then Assumption \ref{assumption_Psi}($m$) is fulfilled for every $m\leqslant m_0$.
\\
\\
Now, consider the truncated estimators
\begin{displaymath}
\widetilde b_{m}^{\prime,1} :=
\widetilde b_{m}^{\prime,1}\mathbf 1_{\Lambda_{m + p}}
\textrm{ and }
\widetilde b_{m}^{\prime,2} :=
\widetilde b_{m}^{\prime,2}\mathbf 1_{\Lambda_{m + p}},
\end{displaymath}
where
\begin{displaymath}
\Lambda_m :=
\left\{\mathfrak L(m)(\|\widehat\Psi_{m}^{-1}\|_{\textrm{op}}\vee 1)
\leqslant
\mathfrak c\frac{n}{\log(n)}\right\}.
\end{displaymath}
Then, let us establish elaborate risk bounds on $\widetilde b_{m}^{\prime,1}$ and $\widetilde b_{m}^{\prime,2}$.
%


%
\begin{proposition}\label{elaborate_risk_bound_estimator_1}
Under Assumptions \ref{assumption_f}, \ref{assumption_basis_derivatives}($m,p$) and \ref{assumption_Psi}($m + p$), if $\mathbb E[b'(X_1)^4] <\infty$ and $\mathbb E(Y_{1}^{4}) <\infty$, then
\begin{eqnarray*} 
 \mathbb E\left[\|\widetilde b_{m}^{\prime,1} - b'\|_{n}^{2}\right]
 & \leqslant &
 3\|f\|_{\infty}\inf_{t\in\mathcal S_m}\|t - b'\|^2 +
 9\|\Delta_{m,m + p}^{f,1}\|_{{\rm op}}^{2}\|b - b_m\|_{f}^{2} +
 3\|b_m' - (b')_m\|_{f}^{2} +
 \frac{2\sigma^2}{n}\|\Delta_{m,m + p}^{f,1}\|_{F}^{2}\\
 & &
 \hspace{4cm}
 +\left[\frac{2\mathfrak cn}{\log(n)}\|\Delta_{m,m + p}\|_{{\rm op}}^{2}
 \mathbb E(Y_{1}^{4})^{1/2} +
 3\mathbb E(b'(X_1)^{4})^{1/2}\right]\frac{\mathfrak c_{\ref{Omega}}^{1/2}}{n^4}
\end{eqnarray*}
with $\Delta_{m,m + p}^{f,1} :=\Psi_{m + p}^{1/2}\Delta_{m,m + p}^{*}\Psi_{m}^{-1/2}$ and $\Delta_{m,m+p}$ is defined in (\ref{defDelta}).
\end{proposition}
%


%
\begin{proposition}\label{elaborate_risk_bound_estimator_2}
Under Assumptions \ref{assumption_f}, \ref{assumption_basis_derivatives}($m,p$), \ref{assumption_bounds_I}($m$) and \ref{assumption_Psi}($m + p$), if $\mathbb E[b'(X_1)^4] <\infty$ and $\mathbb E(Y_{1}^{4}) <\infty$, then
\begin{eqnarray*} 
 \mathbb E\left[\|\widetilde b_{m}^{\prime,2} - b'\|_{n}^{2}\right]
 & \leqslant &
 2\|f\|_{\infty}\inf_{t\in\mathcal S_m}\|t - b'\|^2 +
 6\|\Delta_{m,m + p}^{f,2}\|_{{\rm op}}^{2}\|b - b_{m + p}\|_{f}^{2} +
 \frac{2\sigma^2}{n}\|\Delta_{m,m + p}^{f,2}\|_{F}^{2}\\
 & &
 \hspace{4cm}
 +\left[\frac{2\mathfrak cn}{\log(n)}\|\Delta_{m,m + p}\|_{{\rm op}}
 \mathbb E(Y_{1}^{4})^{1/2} +
 3\mathbb E(b'(X_1)^{4})^{1/2}\right]\frac{\mathfrak c_{\ref{Omega}}^{1/2}}{n^4}
\end{eqnarray*}
with $\Delta_{m,m + p}^{f,2} :=\Psi_{m + p}^{-1/2}\Delta_{m,m + p}^{*}\Psi_{m}^{1/2}$ and $\Delta_{m,m+p}$ is defined in (\ref{defDelta}).
\end{proposition}
\noindent
The coefficients involved in the bounds given in Propositions \ref{elaborate_risk_bound_estimator_1} and   \ref{elaborate_risk_bound_estimator_2} are the theoretical ones instead of the empirical in Propositions \ref{rough_risk_bound_estimator_1} and \ref{rough_risk_bound_estimator_2}. They will allow us to evaluate rates of convergence for the estimator, provided that the basis is specified. This is the point of the next section.
\\
\\
Let us conclude this section with the following proposition which allows to control the risk in norm $\|.\|_f$ of $\widetilde b_{m}^{\prime,1}$ (resp. $\widetilde b_{m}^{\prime,2}$) via its risk in empirical norm, already controlled several ways in Propositions \ref{rough_risk_bound_estimator_1} and \ref{elaborate_risk_bound_estimator_1} (resp. Propositions \ref{rough_risk_bound_estimator_2} and \ref{elaborate_risk_bound_estimator_2}).
%


%
\begin{proposition}\label{bounds_f_norm}
Under Assumptions \ref{assumption_f}, \ref{assumption_basis_derivatives}($m,p$) and \ref{assumption_Psi}($m$), if $\|\Delta_{m,m + p}\|_{\rm op}^{2}\leqslant\mathfrak m_{\Delta}n^2$ with $\mathfrak m_{\Delta} > 0$ not depending on $m$ and $n$, then
\begin{displaymath}
\mathbb E(\|\widetilde b_{m}^{\prime,1} - b'\|_{f}^{2})
\leqslant 5\|f\|_{\infty}\inf_{t\in\mathcal S_m}\|t - b'\|^2 +
4\mathbb E(\|\widetilde b_{m}^{\prime,1} - b'\|_{n}^{2}) +
\frac{\mathfrak c_{\ref{bounds_f_norm}}}{n}
\end{displaymath}
and, if in addition Assumption \ref{assumption_bounds_I}($m$) is satisfied, then
\begin{displaymath}
\mathbb E(\|\widetilde b_{m}^{\prime,2} - b'\|_{f}^{2})
\leqslant 5\|f\|_{\infty}\inf_{t\in\mathcal S_m}\|t - b'\|^2 +
4\mathbb E(\|\widetilde b_{m}^{\prime,2} - b'\|_{n}^{2}) +
\frac{\mathfrak c_{\ref{bounds_f_norm}}}{n},
\end{displaymath}
where $\mathfrak c_{\ref{bounds_f_norm}} > 0$ is a constant not depending on $m$ and $n$.
\end{proposition}
\noindent
The condition $\|\Delta_{m,m + p}\|_{\rm op}^{2}\leqslant\mathfrak m_{\Delta}n^2$ is satisfied by the trigonometric basis and Hermite's basis (see Section \ref{section_bases}).
%


%
\section{Bases examples and explicit risk bounds}\label{section_bases}
In this section, we describe more precisely several examples of bases. Then we evaluate, for each, the order of the term $\|b_m' - (b')_m\|^2,$ which represents the main difference between the risk bounds of the two estimates $\widetilde b_{m}^{\prime,1}$ and $\widetilde b_{m}^{\prime,2}$. Lastly, we give explicit orders for all the terms involved in the bound of Proposition \ref{elaborate_risk_bound_estimator_2} in order to obtain from our nonasymptotic risk bound asymptotic rates of convergence.

%
\subsection{Examples of bases}
First of all, let us provide four usual bases which can be considered because the $\varphi_j$'s are differentiable:
\begin{itemize}
 \item\textbf{The trigonometric basis:} Defined on $I = [0,1]$ by $t_1(x) := 1$, $t_{2j}(x) :=\sqrt 2\cos(2\pi jx)$ and $t_{2j + 1}(x) :=\sqrt 2\sin(2\pi jx)$ for $j = 1,\dots, p$ with $m = 2p + 1$. Thus, $\mathfrak L(m) = m$ for $(\varphi_j)_{1\leqslant j\leqslant m}= (t_j)_{1\leqslant j\leqslant m}$.
 \item\textbf{The Laguerre basis:} Defined on $I =\mathbb R_+$, via Laguerre's polynomials $L_j$, $j\geqslant 0$, by
 \begin{displaymath}
 \ell_j(x) :=
 \sqrt 2L_j(2x)e^{-x}
 \quad\textrm{with}\quad
 L_j(x) :=
 \sum_{k = 0}^{j}\dbinom{j}{k}(-1)^k\frac{x^k}{k!}.
 \end{displaymath}
 It satisfies $ \langle \ell_j, \ell_k\rangle  =\delta_{k,j}$ (see Abramowitz and Stegun \cite{AS64}, 22.2.13), where $\delta_{k,j}$ is the Kronecker symbol. Then, $(\ell_j)_{j\geqslant 0}$ is an orthonormal family of $\mathbb L^2(\mathbb R_+)$ such that $\ell_j(0) =\sqrt{2}$ and 
 \begin{displaymath}
 \|\ell_j\|_{\infty} =\sup_{x\in\mathbb R_+}|\ell_j(x)| =\sqrt 2.
 \end{displaymath}
 Thus, $\mathfrak L(m)=2m$ for $(\varphi_j)_{1\leqslant j\leqslant m}= (\ell_{j-1})_{1\leqslant j\leqslant m}$.
 The $\ell_j'$'s satisfy the following recursive formula (see Lemma 8.1 in Comte and Genon-Catalot \cite{CGC18}):
 \begin{equation}\label{recursive_derivative_Laguerre}
  \ell_0' = -\ell_0\quad{\rm and}\quad
  \ell_j' = -\ell_j - 2\sum_{k = 0}^{j -1}\ell_k
  \textrm{ for }j\geqslant 1.
 \end{equation}
 \item\textbf{The Hermite basis:} Defined on $I =\mathbb R$, via Hermite's polynomials $H_j$, $j\geqslant 0$, by
 \begin{displaymath}
 h_j(x) :=
 \mathfrak c_h(j)H_j(x)e^{-x^2/2}
 \end{displaymath}
 with
 \begin{displaymath}
 H_j(x) := (-1)^je^{x^2}
 \frac{d^j}{dx^j}(e^{-x^2})
 \quad\textrm{and}\quad
 \mathfrak c_h(j) = (2^jj!\sqrt{\pi})^{-1/2}.
 \end{displaymath}
 The family $(H_j)_{j\geqslant 0}$ is orthogonal for the $e^{-x^2}$-weighted scalar product and as
 \begin{displaymath}
 \int_{{\mathbb R}}
 H_j(x)H_k(x)e^{-x^2}dx =\mathfrak c_h^2(j)\delta_{j,k},
 \end{displaymath}
 we get $\langle h_j,h_k\rangle=\delta_{j,k}$,  (see Abramowitz and Stegun \cite{AS64}, 22.2.14). Moreover,
 \begin{displaymath}
 \|h_j\|_{\infty} =
 \sup_{x\in\mathbb R}|h_j(x)|
 \leqslant\phi_0
 \end{displaymath}
 with $\phi_0 =\pi^{-1/4}$ (see Abramowitz and Stegun \cite{AS64}, 22.14.17 and Indritz \cite{INDRITZ61}). Thus, $\mathfrak L(m)\leqslant\pi^{-1/2}m$, but it is proved in \cite{CLS} that there exists $K > 0$ such that
 \begin{displaymath}
 \sup_{x\in\mathbb R}
 \sum_{j = 0}^{m - 1}h_j(x)^2
 \leqslant K\sqrt{m}.
 \end{displaymath}
 Therefore, we set $\mathfrak L(m) = K\sqrt{m}$ for $(\varphi_j)_{1\leqslant j\leqslant m}= (h_{j-1})_{1\leqslant j\leqslant m}$. The $h_j'$'s also satisfy a recursive formula (see Comte and Genon-Catalot \cite{CGC18}, Equation (52) in Section 8.2):
 \begin{equation}\label{recursive_derivative_Hermite}
 h_0' = -\frac{1}{\sqrt 2}h_1
 \quad{\rm and}\quad
 h_j' =\frac{1}{\sqrt 2}(\sqrt jh_{j - 1} -\sqrt{j + 1}h_{j + 1})
 \textrm{ for }j\geqslant 1.
 \end{equation}
 \item\textbf{The Legendre basis:} Defined on $I = [-1,1]$, via Legendre polynomials $G_j$, $j\geqslant 0$, by
 \begin{displaymath}
 g_j(x) :=\sqrt{\frac{2j + 1}2}G_j(x)
 \quad\textrm{with}\quad
 G_j(x) :=\frac{1}{2^jj!}\cdot\frac{d^j}{dx^j}[(x^2 - 1)^j].
 \end{displaymath}
 As
 \begin{displaymath}
 \int_{-1}^{1}G_j(x)G_k(x)dx =
 \frac{2}{2j + 1}\delta_{j,k},
 \end{displaymath}
 the family $(g_j)_{j\geqslant 0}$ is an orthonormal family of $\mathbb L^2([-1,1])$. For example, $g_0(x) = 1/\sqrt 2$, $g_1(x) =\sqrt{3/2}x$, $g_2(x) = 1/2\sqrt{5/2}(3x^2 - 1)$, etc. Note that they are easy to compute numerically thanks to the recursive formula
$ g_{j}(x) =\frac{1}{j}
 [(2j - 1)xg_{j - 1}(x) - (j - 1)g_{j - 2}(x)]$, $j\geqslant 1$, 
 (see Formula 2.6.2 in \cite{EFROMOVICH99}). Moreover,
 \begin{displaymath}
 \|g_j\|_{\infty}
 \leqslant\sqrt{\frac{2j + 1}{2}}
 \textrm{, which gives }
 \sum_{j = 0}^{m - 1}g_j(x)^2
 \leqslant
 \frac{1}{2}
 \sum_{j = 0}^{m - 1}(2j + 1) =
 \frac{m^2}{2}
 \end{displaymath}
 and $\mathfrak L(m) = m^2/2$ (see also Cohen {\it et al.} \cite{CDL13}) for $(\varphi_j)_{1\leqslant j\leqslant m} = (g_{j - 1})_{1\leqslant j\leqslant m}$. The $g_j'$'s also satisfy a recursive formula (see Formula (22) p.10 in Lagrange \cite{Lagrange}):
 \begin{displaymath}
 \frac{d}{dx}g_{j + 1} (x) =
 \sqrt{2j + 3}\sum_{k = 0}^{[j/2]}\sqrt{2(j - 2k) + 1}
 g_{j - 2k}(x),
 \end{displaymath}
 which can be written
 \begin{equation}\label{derivLegendre}
 g_{2p + 1}'(x) =
 \sqrt{4p + 3}
 \sum_{k = 0}^{p}
 \sqrt{4k + 1}
 g_{2k}(x),\quad
 g_{2p + 2}'(x) 
 =\sqrt{4p + 5}
 \sum_{k = 0}^{p}
 \sqrt{4k + 3}
 g_{2k + 1}(x).
 \end{equation}
\end{itemize}
Under Assumption \ref{assumption_bounds_I}($m$), thanks to Equality (\ref{IPP}) and to the recursive formulas available for each basis described above, we are able to compare the derivative $b_m'$ of $b_m$ to the derivative of the projection $(b')_m$ of $b'$ as follows:
%


%
\begin{proposition}\label{additional_term_risk_bound_estimator_1}
Under Assumption \ref{assumption_bounds_I}($m$):
\begin{itemize}
 \item[{\rm (i)}] If $I = [0,1]$ and $\varphi_j = t_j$ (the trigonometric basis with an odd $m$), then $\|b_m' - (b')_m\|^2 = 0$.
 \item[{\rm (ii)}] If $I =\mathbb R$ and $\varphi_j = h_{j-1}$ (the Hermite basis), then
 \begin{displaymath}
 \|b_m' - (b')_m\|^2 =
 \frac{m}{2}
 (\langle b,h_{m - 1}\rangle^2 +\langle b,h_m\rangle^2).
 \end{displaymath}
 \item[{\rm (iii)}] If $I =\mathbb R_+$ and $\varphi_j =\ell_{j-1}$ (the Laguerre basis), then
 \begin{displaymath}
 \|b_m' - (b')_m\|^2 =
 4m\left(\sum_{k = 0}^{m - 1}
 \langle b,\ell_k\rangle\right)^2.
 \end{displaymath}
If in addition $b(0) = 0$, then
 \begin{displaymath}
 \|b_m' - (b')_m\|^2 = 4m\left(\sum_{k\geqslant m}\langle b,\ell_k\rangle\right)^2.
 \end{displaymath}
 \item[{\rm (iv)}]
 If $I = [-1,1]$, $\varphi_j = g_{j-1}$ (the Legendre basis) and $m = 2p$, then
 \begin{eqnarray*}
  \|b_m' - (b')_{m}\|^2
  & = &
  3\left(\sum_{k = 0}^{p - 1}
  \sqrt{4k + 3}
  \langle b,g_{2k + 1}\rangle \right)^2 +
  (4p - 1)
  \left(\sum_{k = 0}^{p - 1}
  \sqrt{4k + 1}
  \langle b,g_{2k}\rangle\right)^2\\
  & &
  +\sum_{j = 0}^{p - 1}
  \left(\sqrt{4j + 3}
  \sum_{k = j}^{p - 1}\sqrt{4k + 3}
  \langle b,g_{2k + 1}\rangle +
  \sqrt{4j + 1}
  \sum_{k = 0}^{j}
  \sqrt{4k + 3}
  \langle b,g_{2k + 1}
  \rangle\right)^2\\
  & & + 
  \sum_{j = 0}^{p - 2}
  \left(\sqrt{4j + 5}
  \sum_{k = j + 1}^{p - 1}
  \sqrt{4k + 1}
  \langle b,g_{2k}\rangle +
  \sqrt{4j + 2}
  \sum_{k = 0}^{j}
  \sqrt{4k + 1}
  \langle b,g_{2k}\rangle\right)^2.
 \end{eqnarray*}
 This implies that there exists a deterministic constant $\mathfrak c_{\ref{additional_term_risk_bound_estimator_1}} > 0$, not depending on $m$ and $n$, such that $\|b_m' - (b')_m\|^2\leqslant\mathfrak c_{\ref{additional_term_risk_bound_estimator_1}}m^4$.
\end{itemize}
\end{proposition}
\noindent
The cases are ordered from the simplest (the trigonometric one) to the most complicated (Legendre case for an even $m$). Proposition \ref{additional_term_risk_bound_estimator_1} shows that the term $  \|b_m' - (b')_{m}\|^2$ which appears in the risk bound of $\widetilde b_{m}^{\prime,1}$ importantly depends on the basis. Clearly, the first two bases are more convenient for this problem and we will focus on them in the sequel (for rates and simulation experiments). 
%


%
\subsection{Explicit risk bound for the trigonometric basis}\label{trigo-order}
As the trigonometric basis has compact support, say $I$, we estimate in fact $b := b\mathbf 1_I$ and we can assume that $f(x)\geqslant f_0 > 0$ for every $x\in I$. Moreover, we assume that $f$ is bounded (Assumption \ref{assumption_f}). We set $I = [0,1]$ for simplicity and assume that $b(0) = b(1)$ (which ensures Assumption \ref{assumption_bounds_I}($m$) for all $m$). Then, by considering models with an odd $m$, Assumption \ref{assumption_basis_derivatives}($m,p$) is fulfilled for all $m$ with $p = 0$. Moreover, we know from \cite{CGC20} that $\|\Psi_{m}^{-1}\|_{{\rm op}}\leqslant 1/f_0$. Then, we get 
\begin{displaymath}
\mathfrak L(m) = m,\quad
\|\Delta_{m,m}\|_{{\rm op}}^{2}
\leqslant\pi^2m^2,\quad
\|\Delta_{m,m}^{f,1}\|_{{\rm op}}^{2} =
\|\Delta_{m,m}^{f,2}\|_{{\rm op}}^{2}\leqslant
\frac{\|f\|_{\infty}}{f_0}m^2,
\end{displaymath}
and
\begin{displaymath} 
\|\Delta_{m,m}^{f,1}\|_{F}^{2} =\|\Delta_{m,m}^{f,2}\|_{F}^{2}
\leqslant\frac{1}{f_0}m^3.
\end{displaymath}
The last bound comes from the following inequalities
\begin{eqnarray}
 \nonumber
 \|\Delta_{m,m}^{f,2}\|_{F}^{2}
 & = &
 {\rm Tr}[\Psi_{m}^{-1}\Delta_{m,m}^{*}
 \Psi_m\Delta_{m,m}]
 \leqslant
 \|\Psi_{m}^{-1}\|_{{\rm op}}{\rm Tr}[\Delta_{m,m}^{*}\Psi_m\Delta_{m,m}]
 =\|\Psi_{m}^{-1}\|_{{\rm op}}\mathbb E\left[{\rm Tr}(\Delta_{m,m}^{*}
 \widehat\Psi_m\Delta_{m,m})\right]\\
 \label{trick}
 & = &
 \frac{1}{n}\|\Psi_{m}^{-1}\|_{\rm op}
 \mathbb E\left[{\rm Tr}(\widehat\Phi_m'(\widehat \Phi_m')^*)\right]
 =\frac{1}{n}\|\Psi_{m}^{-1}\|_{\rm op}
 \mathbb E\left[\sum_{i = 1}^{n}\sum_{j = 0}^{m - 1}
 \varphi_j'(X_i)^2\right]\\
 \nonumber
 & \leqslant &
 \frac{m^2}{n}\|\Psi_{m}^{-1}\|_{\rm op}
 \mathbb E\left[\sum_{i = 1}^{n}
 \sum_{j = 0}^{m - 1}\varphi_{j}^{2}(X_i)\right]
 = m^3\|\Psi_{m}^{-1}\|_{\rm op}
 \leqslant\frac{m^3}{f_0},
\end{eqnarray}
using that for $\varphi_j = t_j$, $\varphi_j' =\pm 2\pi j\varphi_{j\pm 1}$. So, the risk bound on $\widetilde b_{m}^{\prime,2}$ established at Proposition \ref{elaborate_risk_bound_estimator_2} becomes
\begin{equation}\label{CP_trigo} 
\mathbb E\left[
\|\widetilde b_{m}^{\prime,2} - b'\|_{n}^{2}\right]
\leqslant
2\|f\|_{\infty}\left(\inf_{t\in\mathcal S_m}\|t - b'\|^2 +
\frac{6}{f_0}m^2\|b - b_m\|^2\right) +\frac{2\sigma^2}{nf_0}m^3
+\frac{\mathfrak c_1}{n}
\end{equation}
with $\mathfrak c_1 > 0$ and odd $m$. Since $p = 0$ and $b_m' = (b')_m$ for the trigonometric basis, the risk bound on $\widetilde b_{m}^{\prime,1}$ established at Proposition \ref{elaborate_risk_bound_estimator_1} is the same up to a multiplicative constant.
\\
\\
Now, let us evaluate the rate of convergence of the estimator for $b$ in some regularity space and well chosen $m$. Let $\beta$ be a positive integer, $L > 0$ and define
\begin{eqnarray*}
 W^{\rm per}(\beta, L) & := &
 \{g\in C^{\beta}([0,1];\mathbb R) : g^{(\beta-1)}\textrm{ is absolutely continuous},\\
 & &
 \hspace{2cm}
 \int_{0}^{1}g^{(\beta)}(x)^2dx
 \leqslant L^2\textrm{ and }g^{(j)}(0) = g^{(j)}(1)\textrm{, }\forall j = 0,\dots,\beta - 1\}.
\end{eqnarray*}
We obtain the following result:
%


%
\begin{corollary}\label{rate_trigonometric}
Consider the estimators $\widehat b_{m}^{\prime,i}$, $i=1, 2$ computed in the trigonometric basis on $I = [0,1]$ with $0 < f_0\leqslant f(x)\leqslant\|f\|_{\infty} <\infty$, $\mathbb E[b'(X_1)^4] <\infty$ and $\mathbb E(Y_{1}^{4}) <\infty$. If $b\in W^{{\rm per}}(\beta, L)$ with $\beta > 1$, $b(0) = b(1)$ and $m_{\rm opt} = n^{1/(2\beta + 1)}$, then
\begin{displaymath}
\mathbb E\left[
\|\widetilde b_{m_{\rm opt}}^{\prime,i} - b'\|_{n}^{2}\right]
\leqslant
\mathfrak c_i(L,\beta,\|f\|_{\infty},f_0,\sigma^2)n^{-2(\beta - 1)/(2\beta + 1)}, \quad \mbox{ for} \, i=1, \, 2.
\end{displaymath}
\end{corollary}
%


%
\begin{proof}
By Proposition 1.14 of \cite{TSYBAKOV09}, a function $f\in W^{\rm per}(\beta, L)$ admits a development
\begin{displaymath}
f =\sum_{j = 0}^{\infty}\theta_j\varphi_j
\quad\textrm{such that}\quad
\sum_{j\geqslant 0} \theta_{j}^{2}\tau_{j}^{2}
\leqslant C(L),
\end{displaymath}
where $\tau_j = j^{\beta}$ for even $j$, $\tau_j = (j - 1)^{\beta}$ for odd $j$, and $C(L) := L^2\pi^{-2\beta}$. Moreover, if $b$ belongs to a the Sobolev ellipsoid $W^{\rm per}(\beta,L)$ with $\beta >1$, then $b'\in W^{\rm per}(\beta - 1,2\pi L)$. So,
\begin{displaymath}
\|b - b_m\|^2\leqslant\mathfrak c(L,\beta)m^{-2\beta}
\quad {\rm and}\quad
\|b' - (b')_m\|^2\leqslant\mathfrak c(L,\beta)m^{-2(\beta - 1)}.
\end{displaymath}
Therefore, plugging $m = m_{\rm opt} = n^{1/(2\beta + 1)}$ in (\ref{CP_trigo}) gives the result of Corollary \ref{rate_trigonometric}. Indeed, Propositions \ref{elaborate_risk_bound_estimator_1} and  \ref{elaborate_risk_bound_estimator_2} apply because the required conditions are automatically satisfied by the trigonometric basis.
\end{proof}
\noindent
Note that we obtain the optimal rate for estimating the derivative of a regression function (see Stone \cite{STONE82}). It coincides also with the rate of estimation for the derivative of a density (see Tsybakov \cite{TSYBAKOV09}, Efromovich \cite{EFROMOVICH98,EFROMOVICH99}, recently Lepski \cite{LEPSKI18} on general Nikolski's spaces, or Comte {\it et al.} \cite{CDS20}).
%


%
\subsection{Explicit risk bound for Hermite basis}\label{Explicit_Hermite}
Consider $s,D > 0$ and the Sobolev-Hermite ball of regularity $s$
\begin{equation}\label{Sob-Her}
W_H^{s}(D) =
\left\{\theta\in\mathbb L^2(\mathbb R) :
\sum_{k\geqslant 0}k^sa_{k}^{2}(\theta)\leqslant D\right\}, 
\end{equation}
where $a_{k}^{2}(\theta) =\langle\theta,h_k\rangle$. In the Hermite case, the following bounds hold:
\begin{displaymath}
\mathfrak L(m) = K\sqrt m,\quad
\|\Delta_{m,m + 1}\|_{\rm op}^{2}
\leqslant 2m,\quad
\|\Delta_{m,m + 1}^{f,1}\|_{\rm op}^{2} =
\|\Delta_{m,m + 1}^{f,2}\|_{\rm op}^{2}\leqslant
2\|f\|_{\infty}\|\Psi_{m + 1}^{-1}\|_{\rm op}m
\end{displaymath}
and
\begin{displaymath} 
\|\Delta_{m,m + 1}^{f,1}\|_{F}^{2} =
\|\Delta_{m,m + 1}^{f,2}\|_{F}^{2}\leqslant 
2K\|\Psi_{m + 1}^{-1}\|_{\rm op}(m + 1)^{3/2}.
\end{displaymath}
The last bound is obtained by following the line of the trigonometric case above, up to (\ref{trick}), using next formula (\ref{recursive_derivative_Hermite}) for the derivative of the basis functions.
\\
\\
In this context, it is proved in \cite{CGC20} that $\|\Psi_{m}^{-1}\|_{{\rm op}}$ is increasing with $m$. Therefore, we can state the following result.
%


%
\begin{corollary}\label{rate_Hermite}
Consider the estimators $\widehat b_{m}^{\prime,i}$, $i=1,2$ computed in the Hermite basis on $I = \mathbb R$ under Assumptions \ref{assumption_f} and \ref{assumption_Psi}$(m + 1)$. Assume  that $b'$ is square integrable, $\mathbb E[b'(X_1)^4] <\infty$, $\mathbb E(Y_{1}^{4}) <\infty$ and that  $b\in W_{H}^{s}(D)$. If $\|\Psi_{m}^{-1}\|_{\rm op}\lesssim m^{\gamma}$ for all $m$ and $s > 1 +\gamma$, then by choosing $m_{\rm opt} = n^{1/(s + 1/2)}$ yields
\begin{displaymath}
\mathbb E\left[
\|\widetilde b_{m_{\rm opt}}^{\prime,i} - b'\|_{n}^{2}\right]
\leqslant
\mathfrak c(D, s,\|f\|_{\infty},\sigma^2)n^{-2(s - 1 -\gamma)/(2s + 1)} \quad  \mbox{ for }\; i=1, 2.
\end{displaymath}
\end{corollary}
\noindent
\begin{remark}\label{rmk} The rate is deteriorated compared to $n^{-2(s - 1)/(2s + 1)}$, which is the optimal rate of estimation for the derivative of a density in a similar non compact setting (see bounds (15) and (16) in \cite{CDS20}). However, we are in the framework of an inverse problem, due both to the derivative aspect and to the non compact support feature of the basis. If we compare the rate with the one found in \cite{CGC20} for the estimation of $b$ in the same context, $n^{-s/(s+1)}$, we would expect $n^{-(s-1)/(s+1)}$  (which is larger). The deterioration is unavoidable as soon as  the term $\|\Psi_{m}^{-1}\|_{\rm op}$ appears as multiplicative factor in the variance and the additional bias term. So the order obtained in Proposition \ref{rate_Hermite} shows consistency but we do not know if it is optimal.\\ The main question is about the bounds $\|\Delta_{m,m + 1}^{f,1}\|_{\rm op}^{2}$ and $\|\Delta_{m,m + 1}^{f,1}\|_{F}^{2}$: the matrices in the norms involve both a matrix of type $\Psi_m$ and a matrix of type $\Psi_m^{-1}$ and if they could be associated, the factor $\|\Psi_{m}^{-1}\|_{\rm op}$ would not appear in the risk bound. The order of the additional bias term would be $m^{-(s-1)}$ and the variance would be of order $m^2/n$. This seems to be true numerically.
 \end{remark}
%


The  behavior of $\Psi_m$ is crucial for understanding our procedure. We want here to mention that in \cite{CGC20}, it is proved that, for all $m$, the matrix  $\Psi_m$ computed in the Hermite basis is invertible and  there exists a constant $c^\star$ such that, 
\begin{equation}\label{minoutile} \|\Psi_m^{-1}\|_{{\rm op}}^2\geqslant c^{\star} m. 
\end{equation}
So, in the Hermite case, Inequality (\ref{minoutile}) clearly implies that $\|\Psi_m^{-1}\|_{{\rm op}}$ cannot be uniformly bounded in $m$ contrary to the case of compactly supported bases.
Moreover, if we assume that 
 $f(x)\geqslant c/(1+x^2)^k$ for $x\in {\mathbb R}$ and $k\geqslant 1$, then for $m$ large enough,
 $\|\Psi_m^{-1}\|_{{\rm op}}\leqslant C m^k$. 
Numerical experiments  seem to indicate that the order $m^k$ is sharp.\\

\begin{proof}[Proof of Corollary \ref{rate_Hermite}.]
The following Lemma (Lemma 2.2 in Comte {\it et al.} \cite{CDS20}) gives a relationship between the regularity of $\theta\in W_{H}^{s}(D)$ and the regularity of its derivative.
%


%
\begin{lemma}\label{lem:TWH}
Consider $s\geqslant 1$ and $D > 0$. If $\theta\in W_{H}^{s}(D)$ admits a square integrable derivative, then there exists a constant $D' = C(D) > D$ such that $\theta'\in W_{H}^{s - 1}(D')$.
\end{lemma}
\noindent
By Lemma \ref{lem:TWH}, if $b\in W_H^{s}(D)$, then $\|b - b_m\|^2\leqslant Dm^{-s}$, $\|b' - (b')_m\|^2\leqslant C(D)m^{-s + 1}$, and the risk bound on $\widetilde b_{m}^{\prime,2}$ established at Proposition \ref{elaborate_risk_bound_estimator_2} becomes
\begin{eqnarray}
 \nonumber
 \mathbb E\left[
 \|\widetilde b_{m}^{\prime,2} - b'\|_{n}^{2}\right]
 & \leqslant &
 2\|f\|_{\infty}\left(\inf_{t\in\mathcal S_m}\|t - b'\|^2 +
 6\|\Psi_{m+1}^{-1}\|_{\rm op}m\|b - b_{m + 1}\|^2\right)\\
 \nonumber
 & & \hspace{5cm} +
 \frac{4K\sigma^2}{n}\|\Psi_{m + 1}^{-1}\|_{\rm op}(m + 1)^{3/2}
 +\frac{\mathfrak c_1}{n}\\
 \label{CP_Hermite}
 & \leqslant &
 C(D)\|f\|_{\infty}[m^{-(s-1)} +
 \|\Psi_{m+1}^{-1}\|_{{\rm op}}m(m + 1)^{-s}] +
 \frac{4K\sigma^2}{n}\|\Psi_{m+1}^{-1}\|_{\rm op}(m + 1)^{3/2} +\frac{\mathfrak c_1}{n}
\end{eqnarray}
with $\mathfrak c_1 > 0$. Thus, if $\|\Psi_{m}^{-1}\|_{\rm op} = O(m^{\gamma})$, for $s > \gamma + 1$, the estimator is consistent, and to plug the choice $m = m_{\rm opt} = n^{1/(s + 1/2)}$ in (\ref{CP_Hermite}) gives the result of Corollary \ref{rate_Hermite} for $i=2$.

The risk  $\mathbb E\left[
 \|\widetilde b_{m}^{\prime,1} - b'\|_{n}^{2}\right]$ involves an additional term $\|f\|_\infty \|b_m' - (b')_m\|^2$. From Proposition \ref{additional_term_risk_bound_estimator_1}, $(ii)$, we have $$\|b_m' - (b')_m\|^2=(m/2)(a_m^2(b)+a^2_{m-1}(b))\leqslant (1/2)(\sum_{k\geqslant m-1} a_k^2(b) + \sum_{k\geqslant m-1}ka_k^2(b))\lesssim m^{-(s-1)}$$ under our assumptions, by writing that $$\sum_{k\geqslant m-1}ka_k^2(b)=\sum_{k\geqslant m-1} k^s a_k^2(b) \times k^{-s+1}\leqslant (m-1)^{-s+1} \sum_{k\geqslant 0} k^sa_k^2(b) \leqslant D(m-1)^{-(s-1)}\leqslant D 2^{s-1} m^{-(s-1)}.$$ This gives the result of Corollary \ref{rate_Hermite} for $i=1$.
\end{proof}
%


%
\subsection{Explicit risk bound for Legendre basis}
By Proposition 2.6.1 in \cite{EFROMOVICH99} (see also \cite{DeVoreLorentz}, Section 7.6), it is known that if $b\in C^r([-1,1];\mathbb R)$ ($r\geqslant 1$) and if there exists $\alpha\in (0,1]$ such that
\begin{displaymath}
|b^{(r)}(t) - b^{(r)}(s)|\leqslant Q|t - s|^{\alpha}\textrm{ $;$ }\forall s,t\in [-1,1],
\end{displaymath}
then there exists $c > 0$ such that
\begin{displaymath}
\|b - b_m\|^2\leqslant cm^{-2(r +\alpha)}
\quad {\rm and}\quad
\|b' - (b')_m\|^2\leqslant cm^{-2(r - 1+\alpha)}.
\end{displaymath}
The space of regularity $\beta = r +\alpha$ considered above will be called H\"older space and denoted by $\mathcal H(\beta,Q)$.
\\
\\
By Proposition \ref{additional_term_risk_bound_estimator_1}, we can see that the first estimator may not be consistent as $\|(b_m)' - (b')_m\|^2$ may not tend to zero. However, Formula (\ref{derivLegendre}) shows that the Legendre basis satisfies Assumption \ref{assumption_basis_derivatives}($m,p$) with $p = 0$ and triangular matrix $\Delta_{m,m}$ with null diagonal. As the basis is compactly supported, we can proceed as in the case of the trigonometric basis, assuming $I=[-1,1]$, $0< f_0 < f(x) <\|f\|_{\infty} <\infty$ for every $x\in I$, and $b(-1) = b(1) = 0$. Then,
\begin{displaymath}
\mathfrak L(m) = \frac{m^2}{2},\quad
\|\Delta_{m,m}\|_{\rm op}^{2}
\leqslant cm^4,\quad
\|\Delta_{m,m}^{f,1}\|_{\rm op}^{2} =
\|\Delta_{m,m}^{f,2}\|_{\rm op}^{2}\leqslant
c\frac{\|f\|_{\infty}}{f_0}m^4,
\end{displaymath}
%
%
\begin{displaymath} 
\|\Delta_{m,m}^{f,1}\|_{F}^{2} =\|\Delta_{m,m}^{f,2}\|_{F}^{2}
\leqslant\frac{c}{f_0}m^5.
\end{displaymath}
As a consequence, for $b\in\mathcal H(\beta,Q)$ with $\beta > 2$ and $b(-1) = b(1) = 0$, Proposition \ref{elaborate_risk_bound_estimator_2} implies that if $m_{\rm opt} = n^{1/(2\beta +1)}$, then
\begin{displaymath}
\mathbb E\left[
\|\widetilde b_{m}^{\prime,2} - b'\|_{n}^{2}\right]
\leqslant
\mathfrak c(Q,\beta,\|f\|_{\infty},f_0,\sigma^2)n^{-2(\beta - 2)/(2\beta + 1)}.
\end{displaymath}
We mention this rate, but it is sub-optimal in the compact support case, specifically in comparison with  the trigonometric basis.
%


%
\section{A Goldenshluger-Lepski type adaptive estimator}\label{GLAdapt}
The choice of the adequate $m$ is crucial to reach the best order for the quadratic risk. However, this choice depends on unknown quantities, such as the order of regularity of the unknown function. This is why it is important to propose a way to select this dimension from the data. The problem is difficult, especially if we intend to bound the risk of the associated adaptive estimator. Penalty based model selection often rely on a contrast minimization, which seems not possible here. This is why we propose a  Goldenshluger-Lepski type strategy, described in \cite{GL2011} for kernel estimators, and extended to dimension selection in Chagny \cite{CHAGNY2013}.
\\
More precisely, consider the random collection
\begin{displaymath}
\widehat{\mathcal M}_n :=
\left\{m\in\{1,\dots,n\} :
\mathfrak{L}(m + p)(\|\widehat\Psi_{m + p}^{-1}\|_{\rm op}^{2}\vee 1)\leqslant\mathfrak d\frac{n}{\log(n)}\right\}
\end{displaymath}
where $\mathfrak d > 0$ is a constant depending on $\|f\|_\infty$ (see the proof of Theorem \ref{bound_GL_estimator}), and the random penalty
\begin{displaymath}
\widehat V(m) :=
\frac{\sigma^2m}{n}\|(\widehat\Phi_{m}^{*}\widehat\Phi_m)^{-1}(\widehat\Phi_m')^*\widehat\Phi_m'\|_{\rm op}.
\end{displaymath}
This section deals with the adaptive estimator 
\begin{equation}\label{GLEstim} \widehat b' :=\widehat b_{\widehat m}^{\prime,1},\end{equation}  where
\begin{displaymath}
\widehat m =
\arg\min_{m\in\widehat{\mathcal M}_n}
\left\{A(m) +\kappa_1\widehat V(m)\right\}
\end{displaymath}
with
\begin{displaymath}
A(m) :=
\sup_{m'\in\widehat{\mathcal M}_n}\left\{\|\widehat b_{m\wedge m'}^{\prime,1} -\widehat b_{m'}^{\prime,1}\|_{n}^{2} -\kappa_0\widehat V(m')\right\}_+
\quad{\rm and}\quad
\kappa_0\leqslant\kappa_1.
\end{displaymath}
Consider
\begin{displaymath}
\mathcal M_n :=
\left\{m\in\{1,\dots,n\} :
\mathfrak{L}(m + p)(\|\Psi_{m + p}^{-1}\|_{\rm op}^{2}\vee 1)\leqslant
\frac{\mathfrak d}{4}\cdot\frac{n}{\log(n)}\right\},
\end{displaymath}
the theoretical counterpart of $\widehat{\mathcal M}_n$, and $\mathcal M_n^+$ with the same definition as $\mathcal M_n$ but with $\mathfrak d/4$ replaced by $4\mathfrak d$. The maximal element of $\mathcal M_{n}^{+}$ is denoted by $M_{n}^{+}$. Finally, let
\begin{displaymath}
V(m) :=\frac{\sigma^2m}{n}
\|\Delta_{m,m + p}^{f,1}\|_{\rm op}^{2}
\end{displaymath}
be the theoretical version of $\widehat V(m)$.
%


%
\begin{lemma}\label{increase_penalty}
The map $m\mapsto\widehat V(m)$ is increasing.
\end{lemma}

A general Theorem is stated and proved in Section \ref{sec-thm}, the intested reader is refered to this section where comments are also provided.
We emphasize that the procedure is general and does not depend on the basis. Moreover, the general results states an automatic squared-bias variance compromise. This is important in regard of Remark \ref{rmk} above and Proposition \ref{rate_Hermite_GL} below, in case the rates would not be the best possible ones.

To avoid technicalities, we state two Propositions resulting from this theorem when considering the two main bases we previously described. 
More precisely, for the trigonometric case, we get the following result.
%


%
\begin{proposition}\label{rate_trigonometric_GL}
Consider the estimator $\widehat b'$ computed in the trigonometric basis on $I = [0,1]$ with $0 < f_0\leqslant f(x)\leqslant\|f\|_{\infty} <\infty$ under Assumption \ref{assumption_Psi}$(m)$. Moreover, assume that there exists $\kappa > 0$ such that $\mathbb E(\exp(\kappa\varepsilon_{1}^{2})) <\infty$ and that $b'$ is square-integrable on $I$. If $b\in W^{\rm per}(\beta, L)$ with $\beta >1$, then
\begin{displaymath}
\mathbb E(\|\widehat b' - b'\|_{n}^{2})
\leqslant\mathfrak c(f_0,\|f\|_{\infty},L)n^{-2(\beta - 1)/(2\beta + 1)}.
\end{displaymath}
\end{proposition}
\noindent
Therefore, our data driven estimator automatically reaches the optimal rate, up to a multiplicative constant, in the compactly supported setting associated to the trigonometric basis. For the Hermite case, we obtain the following bound.
%


%
\begin{proposition}\label{rate_Hermite_GL}
Consider the estimator $\widehat b'$ computed in the Hermite basis on $I = {\mathbb R}$ under Assumptions \ref{assumption_f} and \ref{assumption_Psi}$(m+1)$. Assume also that $b'$ is square integrable on $I$, $\mathbb E[b'(X_1)^4] <\infty$ and that there exists $\kappa > 0$ such that $\mathbb E(\exp(\kappa\varepsilon_{1}^{2})) <\infty$. If $\|\Psi_m^{-1}\|_{\rm op}\lesssim m^\gamma$ for every $m\in \{1,\dots,n\}$, and if $b\in W_{H}^{s}(D)$ with $s > 2\gamma + 9/4$, then
\begin{displaymath}
\mathbb E(\|\widehat b' - b'\|_{n}^{2})
\leqslant
\mathfrak c(D,s,\|f\|_{\infty},\sigma^2)n^{-2(s - 1 -\gamma)/(2s + 1)}.
\end{displaymath}
\end{proposition}
\noindent
As a consequence, the Hermite estimator also reaches automatically the best rate we could expect, in the difficult context of non compact setting, but under stronger conditions. Note again that the general Theorem states that the data driven estimator performs the bias-variance compromise, whatever the effective orders of the terms are; this is why the estimator in the Hermite basis can numerically perform even better than in the trigonometric basis, see the next section. 
%


%
\section{A numerical insight on the method}
\begin{table}[h!]
\begin{tabular}{cc||cc|cc|cc}
 & & \multicolumn{2}{c|}{$n=250$} & \multicolumn{2}{c|}{$n=1000$} & 
 \multicolumn{2}{c}{$n=4000$} \\
 & &  Herm & Trigo & Herm & Trigo & Herm & Trigo\\ \hline
 $b_1$ &  100 MSE &  $0.91_{(0.50)}$ & $0.82_{(0.46)}$ & $0.23_{(0.13)}$ & $0.23_{(0.13)}$ & $0.06_{(0.03)}$ & $0.07_{(0.03)}$\\
 & dim & $10.7_{(1.30)}$ & $4.94_{(0.93)}$ & $13.1_{(1.46)}$ & $6.13_{(1.30)}$ & $15.2_{(2.05)}$ & $7.40_{(1.45)} $\\
 $b'_1$ & 100 MSE & $11.1_{(6.27)}$ & $8.75_{(5.11)}$ & $2.97_{(1.79)}$ & $3.16_{(1.92)}$ & $0.85_{(0.47)}$ &  $1.24_{(0.62)}$\\
 & dim & $10.8_{(1.01)}$ & $4.76_{(0.76)}$ & $12.9_{(1.37)}$ & $5.99_{(1.12)}$ & $15.2_{(2.01)}$ & $7.30_{(1.34)}$\\ \hline
 $b_2$ & 100 MSE & $0.18_ {(0.17)}$ & $0.47_{(0.34)}$ & $0.05_{(0.04)}$ & $0.12_{(0.08)}$ & $0.01_{(0.01)}$ & $ 0.03_{(0.02)}$\\
 & dim & $2.15_{(0.47)}$ & $3.08_{(0.91)}$ & $2.15_{(0.47)}$ & $3.08_{(0.91)}$ & $2.18_{(0.58)}$ &  $3.46_{(0.91)}$\\
 $b'_2$  & 100 MSE & $0.18_{(0.19)}$ & $1.77_{(1.07)}$ & $0.05_{(0.05)}$ & $0.58_{(0.36)}$ & $0.01_{(0.01)}$ & $ 0.16_{0.12)}$\\
 & dim & $2.26_{(0.54)}$ & $2.50_{(0.61)}$ & $2.26_{(0.54)}$ & $3.01_{(0.60)}$ & $2.22_{(0.55)}$ & $3.28_ {(0.56)}$\\ \hline 
 $b_3$ & MSE & $1.14_{(0.91)}$ & $0.64_{(0.55)}$ & $0.25_{(0.13)}$ & $0.19_{(0.10)}$ & $0.06_{(0.03)}$ & $0.06_{(0.03)}$\\ 
 & dim & $ 11.3_{(1.15)}$ & $3.69_{(1.16)}$ & $14.1_{(1.35)}$ & $5.12_{(1.77)}$ & $16.7_{(1.44)}$ & $7.18_{(2.07)}$\\
 $b'_3$ & 100 MSE & $18.8_{24.7)}$ & $4.96_{(9.84)}$ & $3.89_ {(2.69)}$ & $2.31_{(2.33)}$ & $1.08_ {(0.71)}$ & $1.21_{(0.73)}$\\
 & dim   & $11.4_{(1.15)}$ & $3.53_{(1.01)}$ & $14.2_{(1.30)}$ & $4.88_{(1.43)}$ & $16.7_{(1.42)}$ & $6.62_{(1.66)}$\\ \hline
 $b_4$ & 100 MSE & $0.68_{(0.37)}$ & $0.82_{(0.44)}$ & $0.21_{(0.11)}$ & $0.24_{(0.12)}$ & $0.06_{(0.03)}$ & $0.07_{(0.03)}$\\
 & dim & $8.71_{(2.21)}$ & $5.20_{(0.92)}$ & $11.6_{(2.21)}$ & $6.30_{(1.32)}$ & $14.9_{(2.43)}$ & $ 7.42_{(1.41)}$\\
 $b'_4$ & 100 MSE & $7.60_{(3.48)}$ & $10.8_{(5.48)}$ & $2.65_{(1.18)}$ & $3.40_{(1.57)}$ & $ 0.93_{(0.43)}$ & $1.17_{(0.55)}$\\
 & dim & $9.17_{(2.02)}$ & $ 5.11_{(0.75)}$ & $11.8_{(1.91)}$ & $6.17_{(1.11)}$ & $15.2_{(1.99)}$ & $7.28_{(0.94)}$\\
\multicolumn{8}{c}{}\\
\end{tabular}
\caption{"MSE": MSE of the oracle (for $b$ and $b'$, defined by (\ref{b1b4})) multiplied by 100 with standard deviations (Std) multiplied by 100 in small parenthesis. "dim": mean of the oracle dimensions with Std in small parenthesis. Columns "Herm" correspond to the Hermite basis, columns "Trigo" to the half trigonometric basis. 400 repetitions}\label{tab1}
\end{table}

\begin{table}[h!]
\begin{tabular}{cc||ccc|ccc|ccc}
 & & \multicolumn{3}{c|}{$n=250$} & \multicolumn{3}{c|}{$n=1000$} & 
 \multicolumn{3}{c}{$n=4000$} \\
 & &  Herm & Trigo &  NWO & Herm & Trigo & NWO &  Herm & Trigo & NWO \\ \hline
 $b'_1$ &  MSE & 78.7  & 13.6 & 335 &  22.0 & 7.13 & 128 & 1.73 & 3.11 & 46.8 \\
& std & 480 & 10 & 12.5 & 37 & 28 &  34 &  2 &  6 &  11 \\
& dim & 12.5 & 9.1 & 0.13 & 6.0 & 10.3 & 0.10 &  19.2 & 13.7 & 0.08\\ \hline
$b'_2$ &  MSE & 0.27 & 4.31 &  3.98 & 0.06 & 1.07 &  1.87 &  0.03 & 0.27 & 0.89 \\
& std & 0.6 & 5.1 & 1.9 & 0.09 & 0.5 & 1 & 0.05 & 0.3 & 0.3 \\
& dim & 2.04 & 4.5 &  0.32 & 2.02 & 5.1 & 0.26 & 2.5 & 5.9 &  0.20 \\ \hline
$b'_3$ &  MSE & 20.4 & 17.3 &  62.1 & 5.87 & 9.75 & 23.9 & 1.71 & 11.8 & 9.03 
\\
& std & 22 & 12 & 25 & 6 & 34 & 10 & 0.9 & 39 & 3 \\
& dim & 12.6 & 5.7 & 0.24 & 15.9 & 9.2 &   0.18 & 19.2 & 16.8 &  0.14 \\ \hline
$b'_4$ & MSE &  22.3 & 37.7 & 36.1 & 7.61 & 6.07 & 15.1 & 3.35 & 2.88 &  6.60 \\
& std & 28 & 14 & 14 & 16 & 4 &  5 & 14 & 1 & 1 \\
& dim & 12.2 & 6.83 &  0.19 & 15.8 & 10.0 & 0.15 & 19.0 & 11.2 & 0.12\\
\multicolumn{11}{c}{}\\
\end{tabular}
\caption{"MSE": MSE multiplied by 100 for the estimation of $b'$, obtained by GL method and defined by (\ref{GLEstim})) with Hermite basis (columns "Herm"), trigonometric basis (columns "Trigo") and for the derivative of NW estimator with oracle bandwidth (columns "NWO"), with their standard deviations multiplied by 100 ("std"). "dim": mean of the selected dimensions or oracle bandwidths. 400 repetitions and 3 sample sizes 250, 1000, 4000.}\label{tab2}
\end{table}

We consider the four simple functions
\begin{equation}\label{b1b4}
b_1(x) = 2\sin(\pi x),\quad
b_2(x) =0.5 x \exp(-x^2/2),\quad
b_3(x) = x^2,\quad
b_4(x) = 4x/(1 + x^2),
\end{equation}
and we generate $Y_i = b(X_i) +\varepsilon_i$, $i = 1,\dots,n$, for i.i.d. $X_i\sim\mathcal N(0,1)$, independent of the i.i.d $\varepsilon_i\sim\mathcal N(0,\sigma^2)$, with $\sigma = 0.25$ and $b = b_j$, $j = 1,\dots,4$. For each sample, we compute the least squares estimator of $b$, together with its derivative, in the Hermite and in the trigonometric bases. We use what we call the "half" trigonometric basis, relying on functions $x\mapsto\sqrt 2\sin(\pi jx)$ and $x\mapsto\sqrt 2\cos(\pi jx)$ on $[0,1]$, rescaled to the interval $[a,b]$. For each function $b$, we considered $K = 400$ repetitions, and samples of sizes $n = 250, 1000$ and $4000$. \\

Due to the theoretical difficulty of the question, in a model which looked rather simple at first sight, we first wondered if the strategy consisting in derivating the least squares regression estimator was relevant, and if numerical investigations could bring information about a good estimation strategy. This is why we first look at oracles: we compute all estimators of the collection and use the knowledge of the true function to compute the error associated to all of them in order to select the best one (the resulting "estimator" is called "oracle") in term of its $\mathbb L^2$-distance to the true. We also look at the associated dimensions.

We compute the $\mathbb L^2$-distance between each oracle estimator of $b$ and the true $b$, and each oracle estimator of $b'$ and the true $b'$, on an interval with bounds corresponding to the 3\% and 97\% quantiles of the $X_i$'s, and finally take the average on 400 independent paths generated. Moreover, we average the selected dimensions for each sample. In other words, we retain the dimension and error corresponding in each case to the smallest error, and compute means and standard deviations. The results are reported in Table \ref{tab1}.\\

\begin{figure}[h!]
\begin{tabular}{c}
\includegraphics[width=15cm,height=3.5cm]{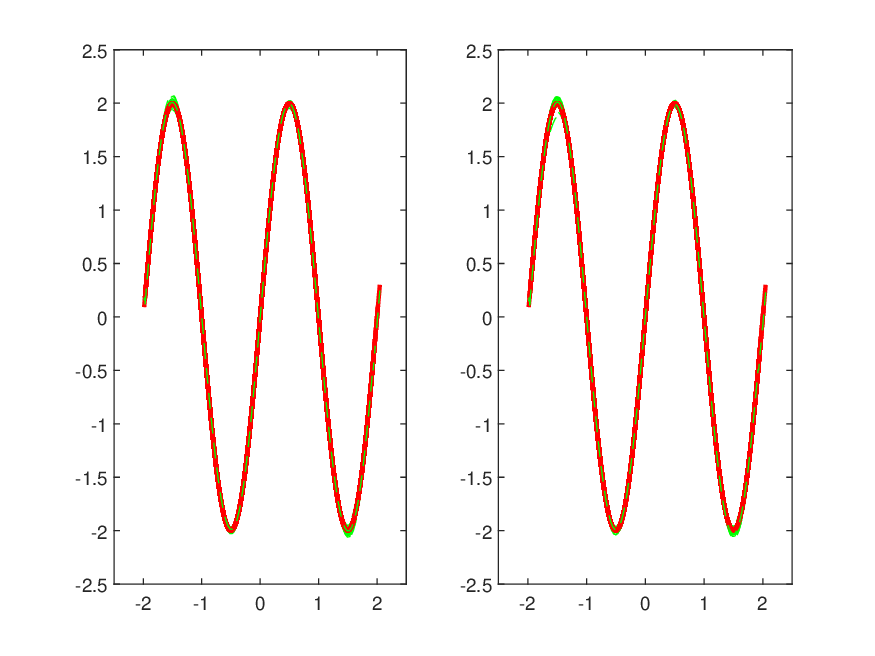}\\
\includegraphics[width=15cm,height=3.5cm]{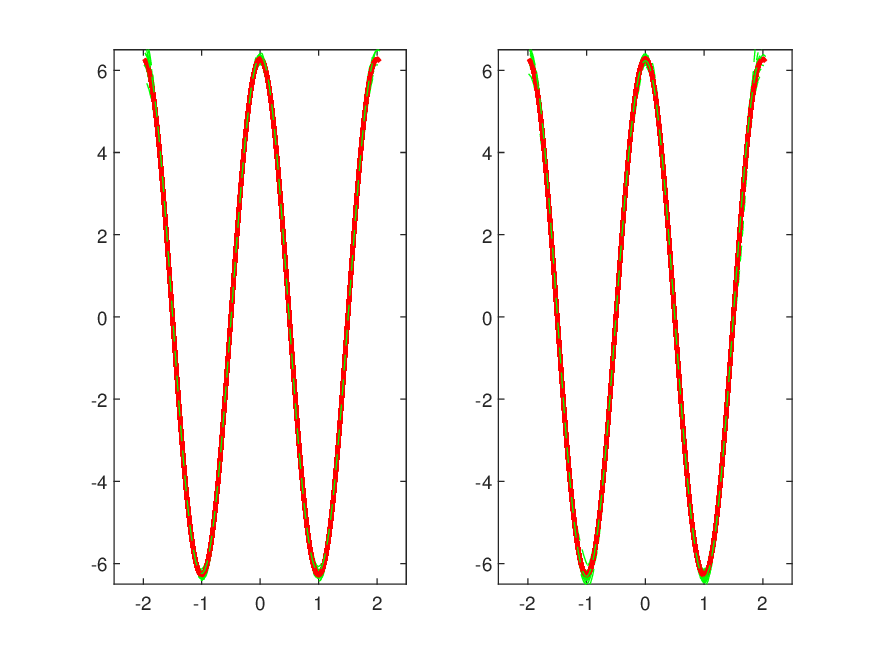}\\
\end{tabular}
\caption{40 estimated functions (dotted green) compared to the true (bold red), $n = 1000$. First line $b_1$ (see (\ref{b1b4})) by penalisation, 100 MSE = 0.26 and 0.29, mean selected dimensions: 12.2 and 11.2. Second line $b'_1$ with GL method, 100 MSE = 4.71 and 6.47, mean selected dimensions: 16.1 and 10.5. Left Hermite basis, right trigonometric basis.}\label{fig1}
\end{figure}

\begin{figure}[h!]
\begin{tabular}{c}
\includegraphics[width=15cm,height=3.5cm]{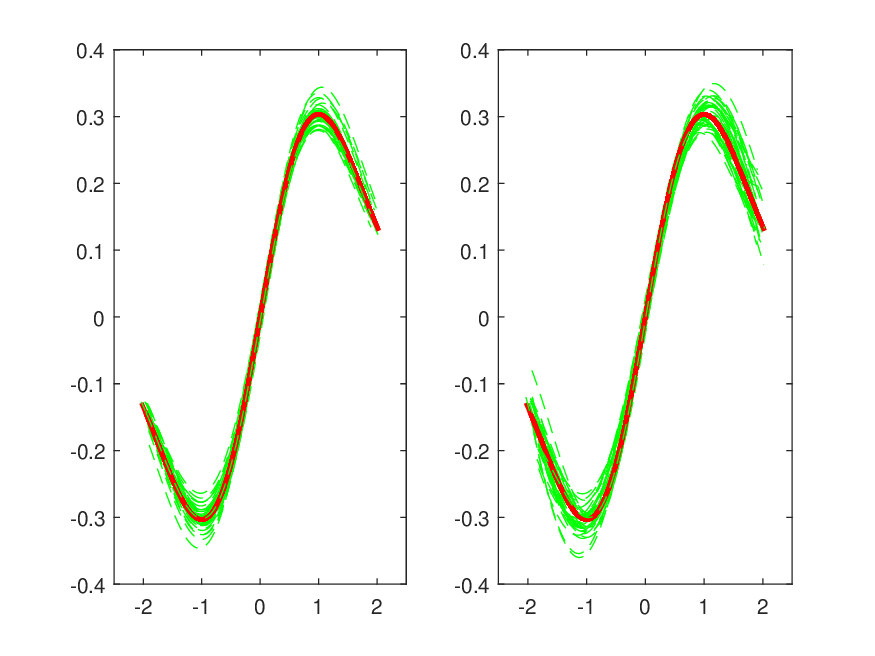}\\
\includegraphics[width=15cm,height=3.5cm]{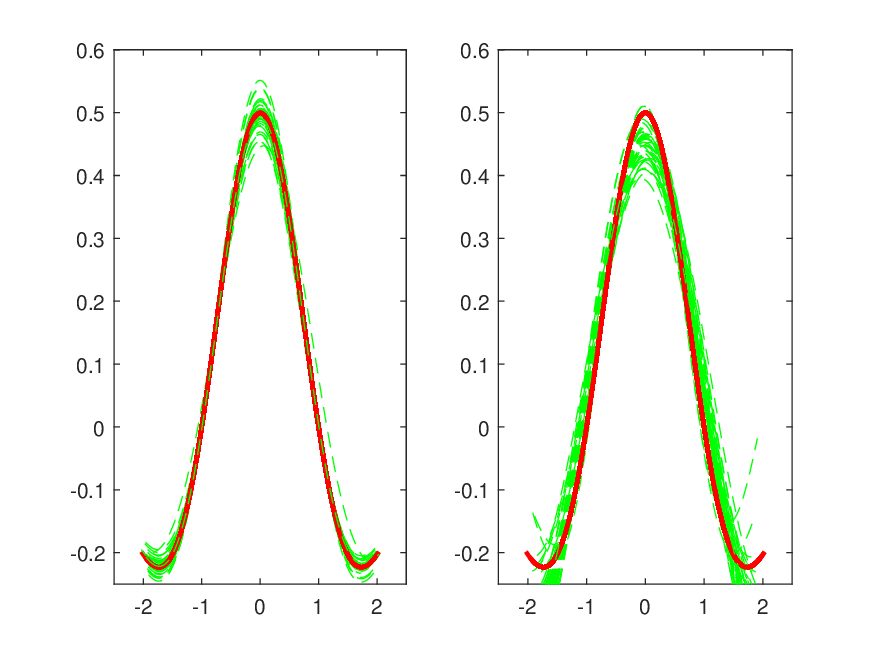}\\
\end{tabular}
\caption{40 estimated functions (dotted green) compared to the true (bold red), $n = 1000$. First line $b_2$ (see (\ref{b1b4})) by penalisation, 100 MSE = 0.07 and 0.18, mean selected dimensions: 2.1 and 5.3. Second line $b'_2$ with GL method, 100 MSE = 1.08 and 1.12, mean selected dimensions: 2.05 and 5.05. Left Hermite basis, right trigonometric basis.}\label{fig2}
\end{figure}

Table \ref{tab1} shows that the MSE decreases when $n$ increases, in all cases, and whether $b$ or $b'$ is estimated.  We can notice that function $b_1$ is chosen to be easy for the trigonometric basis, but the Hermite basis performs very well in this case too. On the contrary, the function $b_2$ is supposed to be easy for the Hermite basis, and it is, with small selected dimensions, but the trigonometric basis has a much worse performance. For the two other functions, the two bases perform similarly, with decreasing error when increasing $n$ and simultaneous increase of the selected dimensions. This is expected from the theoretical formula giving the asymptotic optimal choice of $m$ as a power of $n$, at least when the function under estimation does not admit a finite decomposition in the basis (like $b_1$ for the trigonometric basis or $b_2$ for the HErmite basis).\\
What is  puzzling in these results is the comparison of oracle dimensions for $b$ and $b'$: in each case, they are almost the same. This suggests to keep the selected model obtained for estimation of $b$ by classical penalisation, and use this for $b'$ as well. This is coherent with the fact that the order of the optimal dimension are the same for $b$ and $b'$ when using the trigonometric basis.
\begin{figure}[h!]
\begin{tabular}{c}
\includegraphics[width=15cm,height=3.5cm]{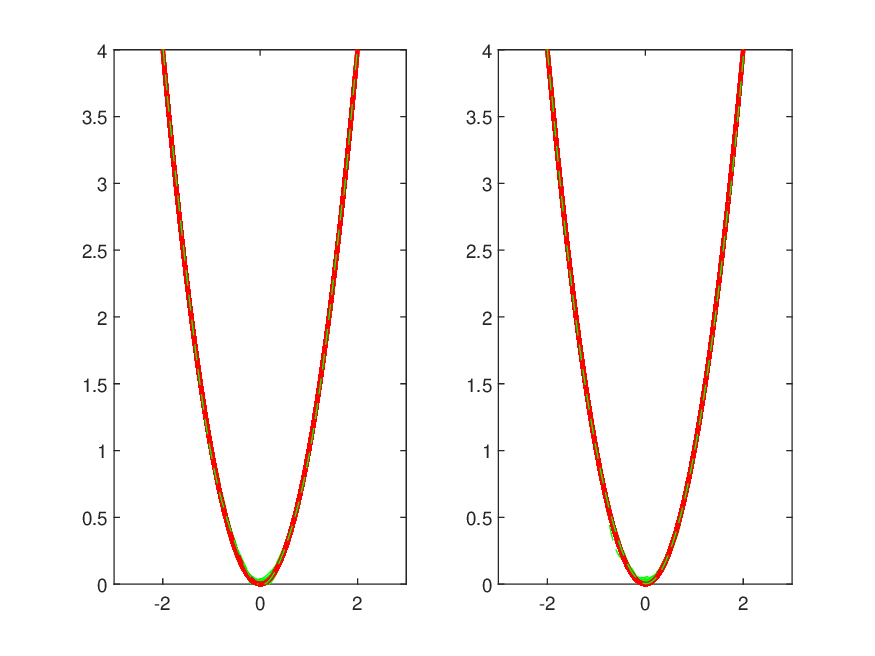}\\
\includegraphics[width=15cm,height=3.5cm]{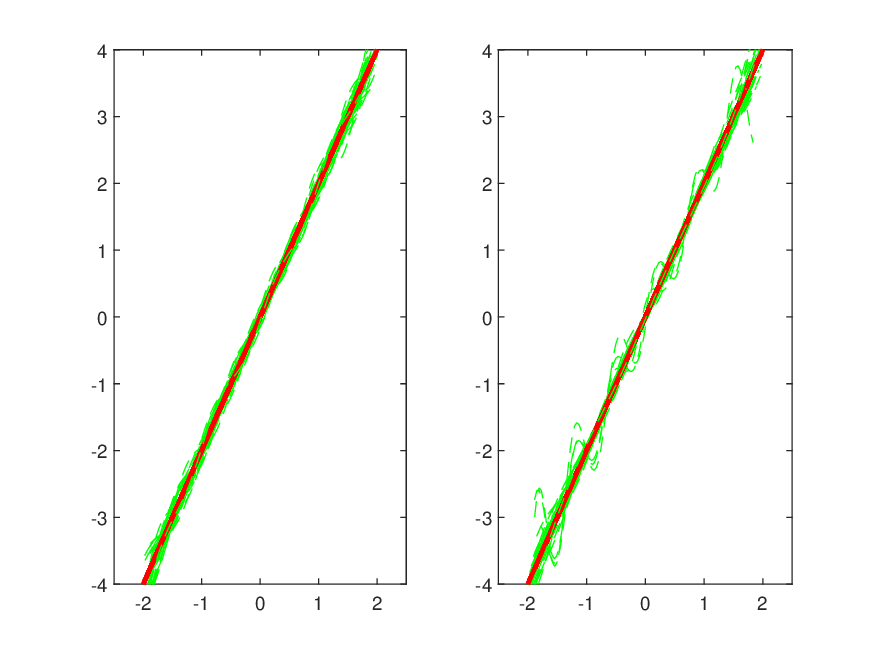}\\
\end{tabular}
\caption{40 estimated functions (dotted green) compared to the true (bold red), $n = 1000$. First line $b_3$ (see (\ref{b1b4})) by penalisation, 100 MSE = 0.30 and 0.30, mean selected dimensions: 1.05 and 9.5. Second line $b'_3$ with GL method, 100 MSE = 5.04 and 6.89, mean selected dimensions: 15.97 and 8.65. Left Hermite basis, right trigonometric basis.}\label{fig3}
\end{figure}

We also implement the Goldenshluger-Lespki method with $\kappa_1=2\kappa_0$ and $\kappa_0=1$ for both Hermite and Trigonometric bases. We compare the performance or our estimator to the derivative of the Nadaraya-Watson estimator (with Gaussian kernel) with recursive computation in the spirit of Bercu~{\it el al.}~\cite{BCD19}. However, we took a fixed oracle bandwidth because their proposal of recursively varying bandwith $h_k=k^{-\alpha}$ with $\alpha=0.3$ does not work in our case and the method of selection of $\alpha$ is not given in their paper. So, we give the results for the best possible choice. The results are given in Table \ref{tab2}, and confirm that our method performs well. Obviously, the selected dimension are larger than the ones pointed by oracles in Table \ref{tab1}, and it is possible that other couples $(\kappa_0, \kappa_1)$ may be better. But it is now documented that the Goldenshluger and Lepski method is difficult to calibrate. The kernel estimator is generally less efficient in spite of its ideal bandwidth choice, even if its error gets very comparable to the other estimators when $n$ increases. The orders associated with the MSE given are more concretely illustrated in Figures \ref{fig1} to \ref{fig4}, and we can see that the estimations are very satisfactory. The estimators of the regression function $b$ is obtained by penalisation as in Comte and Genon-Catalot~\cite{CGC20}. The Hermite basis performs globally very well, even to estimate a straight line as in Figure \ref{fig3}, which seems much more difficult for the trigonometric basis. Lastly, Figure \ref{fig4} shows that there are a lot of side-effects for the estimation of $b'_4$, but it is probably due to "heavy tail" effects since it does not occur for $b'_2$ which has faster decrease, see Figure \ref{fig2}

\begin{figure}[h!]
\begin{tabular}{c}
\includegraphics[width=15cm,height=3.5cm]{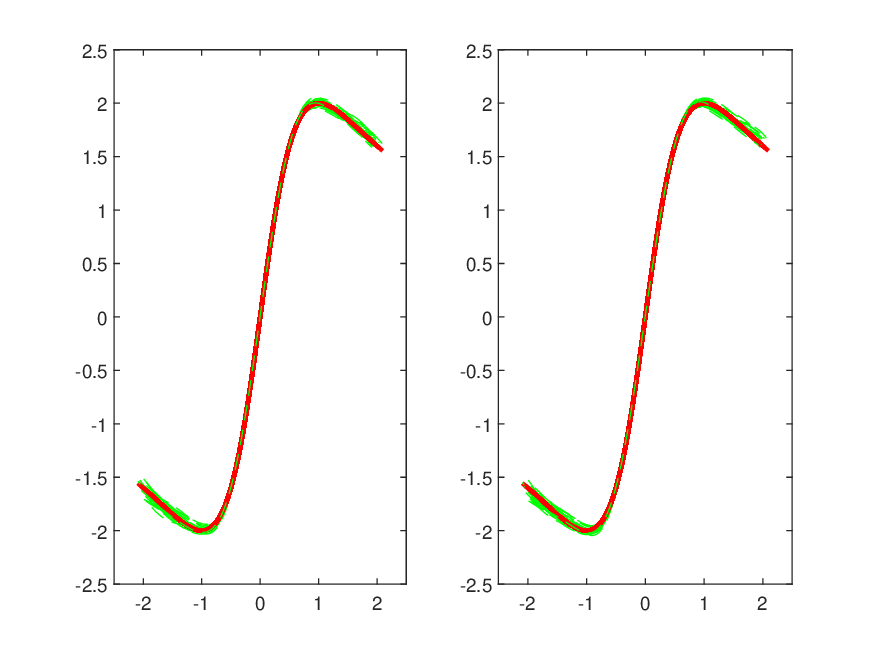}\\
\includegraphics[width=15cm,height=3.5cm]{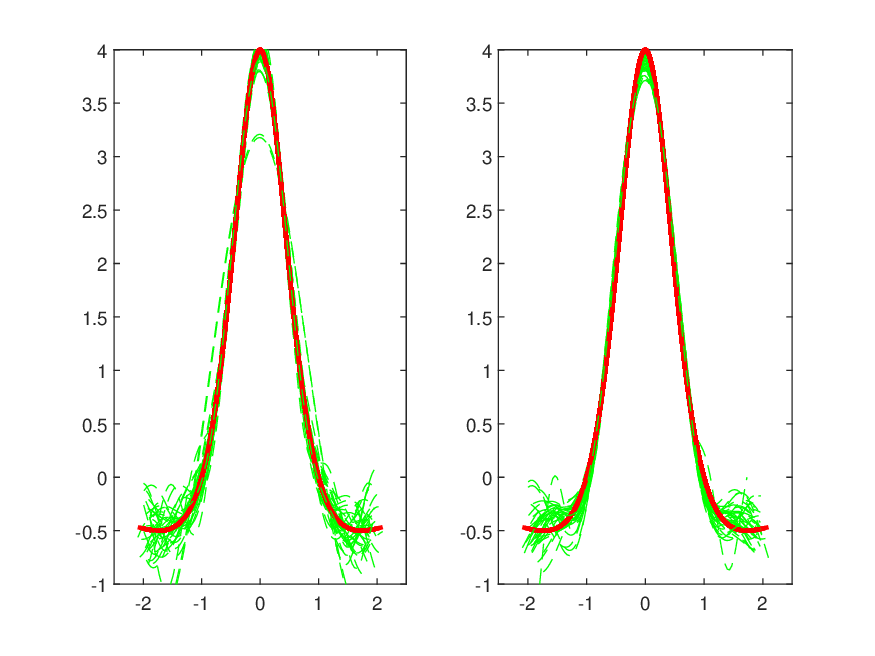}\\
\end{tabular}
\caption{40 estimated functions (dotted green) compared to the true (bold red), $n = 1000$. First line $b_4$  (see (\ref{b1b4}) by penalisation, 100 MSE = 0.29 and 0.33, mean selected dimensions: 9.85 and 10.5. Second line $b'_4$ with GL method, 100 MSE = 11.8 and 6.28 mean selected dimensions: 16.1 and 10.6. Left Hermite basis, right trigonometric basis.}\label{fig4}
\end{figure}
%


%
\section{Concluding Remarks}
In  this paper, we have defined two projection estimators of the derivative of $b$, based on observations $(X_i, Y_i)_{1\leqslant i\leqslant n}$ drawn from Model (\ref{model}). Under weak assumptions, we prove two simple risk bounds allowing to understand the differences between the two strategies. More elaborate bounds under a stability condition introduced by Cohen {\it et al.}\cite{CDL13} are also given. These results are  illustrated in the context of trigonometric, Legendre, Laguerre and Hermite bases, the first two ones being compactly supported, but not the last ones. Optimal rates are recovered with our method in the context of the trigonometric basis, but our setting is more general, which is a novelty. Lastly, we propose a model selection procedure and prove a general risk bound for the adaptive estimator. It automatically reaches the optimal rate in the trigonometric case. These last results are also new and not straightforward.\\
The method we propose is implemented and the few numerical experiments conducted shows that our estimator works well, in particular compared the derivative of a Nadaray-Watson kernel estimators; but more comparisons would be useful to confirm these practical results.  Our investigation for simple examples shows that  the collection of estimators always contains a good one. It also suggests that keeping for the estimation of $b'$ the dimension selected for $b$ may be a safe simple strategy. Several extensions of this work may be of obvious interest: explanatory variables with higher dimensions may be studied in the spirit of Dussap~\cite{Duss22}, as well as higher order of derivatives, possibly only in the compactly supported case to begin with. Extensions to dependent contexts (the case of autoregressive models or the case of diffusion models) are also to be considered. As our proofs rely on results conditionally to the $X_i$'s, thanks to their independence with the noise, dependency should imply theoretical difficulties.
%


%
\section{Proofs}
All the properties on matrix norms used in proofs are reminded in the Subsubsection \ref{section_notations} at the end of the introduction.

Moreover, we denote by $\mathbb E_{\mathbf X}$ the conditional expectation given ${\mathbf X}=(X_1, \dots, , X_n)$.


%
\subsection{Proof of Proposition \ref{rough_risk_bound_estimator_1}}
Note first that
\begin{displaymath}
\widehat b_{m}^{\prime,1}(\mathbf X) =
\widehat\Phi_m'(\widehat\Phi_{m}^{*}\widehat\Phi_m)^{-1}\widehat\Phi_{m}^{*}\mathbf Y,
\end{displaymath}
and since $Y_i = b(X_i) +\varepsilon_i$, $X_i$ is independent of $\varepsilon_i$, and $\mathbb E(\varepsilon_i) = 0$ for every $i\in\{1,\dots,n\}$,
\begin{displaymath}
\mathbb E_{\mathbf X}[\widehat b_{m}^{\prime,1}(\mathbf X)] =
\widehat\Phi_m'(\widehat\Phi_{m}^{*}\widehat\Phi_m)^{-1}\widehat \Phi_{m}^{*}b(\mathbf X)
\end{displaymath}
and
\begin{displaymath}
\mathbb E_{\mathbf X}(\langle
\widehat\Phi_m'(\widehat\Phi_m^*\widehat\Phi_m)^{-1}\widehat\Phi_{m}^{*}\varepsilon, \,  
\widehat\Phi_m'(\widehat\Phi_{m}^{*}\widehat\Phi_m)^{-1}\widehat\Phi_{m}^{*}b(\mathbf X)
- b'(\mathbf X)\rangle_n) = 0.
\end{displaymath}
Then,
\begin{eqnarray*}
 \mathbb E_{\mathbf X}(\|\widehat b_{m}^{\prime,1} - b'\|_n^2)
 & = &
 \frac{1}{n}
 \mathbb E_{\mathbf X}[
 \|\widehat b_{m}^{\prime,1}(\mathbf X) -\mathbb E_{\mathbf X}[\widehat b_{m}^{\prime,1}(\mathbf X)] +
 \mathbb E_{\mathbf X}[
 \widehat b_{m}^{\prime,1}(\mathbf X)] - b'(\mathbf X)\|_{2,n}^{2}]\\
 & = &
 \frac{1}{n}[\mathbb E_{\mathbf X}(\|\widehat\Phi_m'(\widehat\Phi_m^*\widehat\Phi_m)^{-1}
 \widehat\Phi_{m}^{*}\varepsilon\|_{2,n}^{2}) +
 \|\widehat\Phi_m'(\widehat\Phi_{m}^{*}\widehat\Phi_m)^{-1}
 \widehat\Phi_{m}^{*}b(\mathbf X) - b'(\mathbf X)\|_{2,n}^2]\\
 & =: &
 \frac{1}{n}(A + B).
\end{eqnarray*}
On the one hand,
\begin{eqnarray*} 
 A & = & 
 \mathbb E_{\mathbf X}\left[\varepsilon^*\widehat\Phi_m(\widehat\Phi_{m}^{*}\widehat\Phi_m)^{-1}
 (\widehat\Phi_m')^*\widehat\Phi_m'(\widehat\Phi_{m}^{*}\widehat\Phi_m)^{-1}\widehat\Phi_{m}^{*}\varepsilon\right]\\
 & = & \sigma^2{\rm Tr}\left[ \widehat\Phi_m(\widehat\Phi_{m}^{*}\widehat\Phi_m)^{-1}(\widehat\Phi_m')^*\widehat\Phi'_m(\widehat\Phi_m^*\widehat \Phi_m)^{-1} \widehat \Phi_m^* \right]\\
 & = & \sigma^2 {\rm Tr}\left[(\widehat\Phi_m^*\widehat \Phi_m)^{-1} (\widehat \Phi'_m)^*\widehat\Phi'_m\right].
\end{eqnarray*}
On the other hand,
\begin{displaymath}
B\leqslant
3\|\widehat\Phi_m'(\widehat\Phi_{m}^{*}\widehat\Phi_m)^{-1}\widehat\Phi_{m}^{*}b(\mathbf X) - b_m'(\mathbf X)\|_{2,n}^{2} +
3\|b_m'(\mathbf X) - (b')_m(\mathbf X)\|_{2,n}^{2} +
3\|(b')_m(\mathbf X) - b'(\mathbf X)\|_{2,n}^{2}.
\end{displaymath}
So,
\begin{displaymath}
\frac{1}{n}
\mathbb E(B)\leqslant
3\inf_{t\in\mathcal S_m}\|t - b'\|_{f}^{2} +
3\|b_m' - (b')_m\|_{f}^{2} +
3\mathbb E(C)
\end{displaymath}
with
\begin{displaymath}
C =
\frac{1}{n}
\|\widehat\Phi_m'(\widehat\Phi_{m}^{*}\widehat\Phi_m)^{-1}\widehat\Phi_{m}^{*}b(\mathbf X) - b_m'(\mathbf X)\|_{2,n}^{2}.
\end{displaymath}
In order to manage this last term, note that
%
$b_m(\mathbf X) =
\sum_{j = 1}^{m}\langle b,\varphi_j\rangle\varphi_j(\mathbf X) =
\widehat\Phi_m
(\langle b,\varphi_j\rangle)_{1\leqslant j\leqslant m}.
$
%
So,
\begin{equation}\label{rough_risk_bound_estimator_1_1}
(\langle b,\varphi_j\rangle)_{1\leqslant j\leqslant m} =
(\widehat\Phi_{m}^{*}\widehat\Phi_m)^{-1}\widehat\Phi_{m}^{*}b_m(\mathbf X)
\end{equation}
and then,
\begin{displaymath}
b_m'(\mathbf X) =
\sum_{j = 1}^{m}\langle b,\varphi_j\rangle\varphi_j'(\mathbf X) =
\widehat\Phi_m'
(\langle b,\varphi_j\rangle)_{1\leqslant j\leqslant m} =
\widehat\Phi_m'
(\widehat\Phi_{m}^{*}\widehat\Phi_m)^{-1}\widehat\Phi_{m}^{*}b_m(\mathbf X).
\end{displaymath}
Therefore,
\begin{displaymath}
C =
\frac{1}{n}
\|\widehat\Phi_m'(\widehat\Phi_{m}^{*}\widehat\Phi_m)^{-1}\widehat\Phi_{m}^{*}(b(\mathbf X) - b_m(\mathbf X))\|_{2,n}^{2}
\leqslant
\|\widehat\Phi_m'(\widehat\Phi_{m}^{*}\widehat\Phi_m)^{-1}\widehat\Phi_{m}^{*}\|_{{\rm op}}^{2}\|b - b_m\|_{n}^{2}.
\end{displaymath}
This concludes the proof. $\Box$
%


%
\subsection{Proof of Proposition \ref{increase_variance}}
The $\varphi_j$'s do not depend on $m$, so the $\mathcal S_m$'s are nested spaces, and then to establish the following equality is sufficient in order to conclude:
\begin{equation}\label{increase_variance_1}
\mathbb E_{\mathbf X}\left(\sup_{t\in\mathcal S_m :\|t\|_n = 1}
\nu_n(t)^2\right) =
\frac{\sigma^2}{n}{\rm Tr}\left[\widehat\Psi_{m}^{-1/2}\widehat\Psi_m'\widehat\Psi_{m}^{-1/2}\right]
\end{equation}
with
\begin{displaymath}
\widehat\Psi_m' :=
\frac{1}{n}(\widehat\Phi_m')^*\widehat\Phi_m'
\textrm{ and }
\nu_n(t) :=\langle\varepsilon,t'\rangle_n.
\end{displaymath}
Let us prove Equality (\ref{increase_variance_1}). Consider $t\in\mathcal S_m$ such that $\|t\|_n = 1$. Necessarily (and sufficiently),
\begin{displaymath}
t =\sum_{j = 1}^{m}a_j\varphi_j
\end{displaymath}
with $a =\widehat\Psi_{m}^{-1/2}u$ and $u\in\mathbb R^m$ such that $\|u\|_{2,m} = 1$. Then,
\begin{displaymath}
t =\sum_{k = 1}^{m}u_k\sum_{j = 1}^{m}[\widehat\Psi_{m}^{-1/2}]_{j,k}\varphi_j
\end{displaymath}
and, thanks to Cauchy-Schwarz's inequality,
\begin{displaymath}
\nu_n(t)^2 =
\langle\varepsilon,t'\rangle_{n}^{2} =
\left[\sum_{k = 1}^{m}u_k
\left\langle\varepsilon,\sum_{j = 1}^{m}[\widehat\Psi_{m}^{-1/2}]_{j,k}\varphi_j'\right\rangle_n\right]^2
\leqslant
\sum_{k = 1}^{m}
\left\langle\varepsilon,\sum_{j = 1}^{m}[\widehat\Psi_{m}^{-1/2}]_{j,k}\varphi_j'\right\rangle_{n}^{2}.
\end{displaymath}
So,
\begin{displaymath}
\sup_{t\in\mathcal S_m :\|t\|_n = 1}
\nu_n(t)^2 =
\sup_{u\in\mathbb R^m :\|u\|_{2,m} = 1}
\left[\sum_{k = 1}^{m}u_k
\left\langle\varepsilon,\sum_{j = 1}^{m}[\widehat\Psi_{m}^{-1/2}]_{j,k}\varphi_j'\right\rangle_n\right]^2 =
\sum_{k = 1}^{m}
\left\langle\varepsilon,\sum_{j = 1}^{m}[\widehat\Psi_{m}^{-1/2}]_{j,k}\varphi_j'\right\rangle_{n}^{2}.
\end{displaymath}
Therefore, since $\varepsilon_1,\dots,\varepsilon_n$ are i.i.d, centered, and respectively independent of $X_1,\dots,X_n$, and since $\widehat\Psi_{m}^{-1/2}$ and $\widehat\Psi_m'$ are symmetric matrices,
\begin{eqnarray*}
 \mathbb E_{\mathbf X}\left(
 \sup_{t\in\mathcal S_m :\|t\|_n = 1}
 \nu_n(t)^2\right) & = &
 \frac{\sigma^2}{n^2}
 \sum_{k = 1}^{m}
 \sum_{i = 1}^{n}
 \left(\sum_{j = 1}^{m}[\widehat\Psi_{m}^{-1/2}]_{j,k}\varphi_j'(X_i)\right)^2\\
 & = &
 \frac{\sigma^2}{n}
 \sum_{j,k,\ell = 1}^{m}[\widehat\Psi_{m}^{-1/2}]_{j,k}[\widehat\Psi_{m}^{-1/2}]_{\ell,k}
 \langle\varphi_j',\varphi_{\ell}'\rangle_n\\
 & = &
 \frac{\sigma^2}{n}
 \sum_{j,k,\ell = 1}^{m}[\widehat\Psi_{m}^{-1/2}]_{k,j}
 [\widehat\Psi_m']_{j,\ell}[\widehat\Psi_{m}^{-1/2}]_{\ell,k} =
 \frac{\sigma^2}{n}\textrm{Tr}\left[\widehat\Psi_{m}^{-1/2}
 \widehat\Psi_m'\widehat\Psi_{m}^{-1/2}\right].
\end{eqnarray*}
This concludes the proof. $\Box$
%


%
\subsection{Proof of Proposition \ref{rough_risk_bound_estimator_2}}
As in the proof of Proposition \ref{rough_risk_bound_estimator_1},
\begin{eqnarray*}
 \mathbb E_{\mathbf X}(\|\widehat b_{m}^{\prime,2} - b'\|_{n}^{2})
 & = &
 \frac{1}{n}
 \|-\widehat\Phi_m\Delta_{m,m + p}
 (\widehat\Phi_{m + p}^{*}\widehat\Phi_{m + p})^{-1}\widehat\Phi_{m + p}^{*}b(\mathbf X) -
 b'(\mathbf X)\|_{2,n}^{2}\\
 & &
 +\frac{\sigma^2}{n}{\rm Tr}\left[(\widehat\Phi_{m + p}^{*}\widehat\Phi_{m + p})^{-1}\Delta_{m,m + p}^{*}\widehat\Phi_{m}^{*}\widehat\Phi_m\Delta_{m,m + p}\right]\\
 & \leqslant &
 \frac{2}{n}
 [\|-\widehat\Phi_m\Delta_{m,m + p}
 (\widehat\Phi_{m + p}^{*}\widehat\Phi_{m + p})^{-1}\widehat\Phi_{m + p}^{*}b(\mathbf X) - (b')_m(\mathbf X)\|_{2,n}^{2}\\
 & &
 \hspace{7cm} +
 \|(b')_m(\mathbf X) - b'(\mathbf X)\|_{2,n}^{2}]\\
 & &
 +\frac{\sigma^2}{n}{\rm Tr}\left[(\widehat\Phi_{m + p}^{*}\widehat\Phi_{m + p})^{-1}\Delta_{m,m + p}^{*}\widehat\Phi_{m}^{*}\widehat\Phi_m\Delta_{m,m + p}\right].
\end{eqnarray*}
On the one hand, as previously,
\begin{displaymath}
\frac{2}{n}\mathbb E(\|(b')_m(\mathbf X) - b'(\mathbf X)\|_{2,n}^{2}) =
2\|(b')_m - b'\|_{f}^{2}
\leqslant 2\|f\|_{\infty}\inf_{t\in\mathcal S_m}\|t - b'\|^2.
\end{displaymath}
On the other hand, thanks to Equalities (\ref{projection_derivative}) and (\ref{rough_risk_bound_estimator_1_1}),
\begin{displaymath}
(b')_m(\mathbf X) =
-\widehat\Phi_m\Delta_{m,m + p}\left(\langle b, \varphi_j\rangle\right)_{1\leqslant j\leqslant m+p} =  -\widehat\Phi_m\Delta_{m,m+p} (\widehat\Phi^*_{m+p}\widehat\Phi_{m+p})^{-1}\widehat\Phi_{m+p}^{*}b_{m+p}({\mathbf X}).
\end{displaymath}
Then,
\begin{eqnarray*}
 & &
 \|-\widehat\Phi_m\Delta_{m,m + p}(\widehat\Phi_{m + p}^{*}\widehat\Phi_{m + p})^{-1}\widehat\Phi_{m + p}^{*}b(\mathbf X) - (b')_m(\mathbf X)\|_{2,n}^{2}\\
 & &
 \hspace{4cm}
 = \|-\widehat\Phi_m\Delta_{m,m + p} (\widehat\Phi_{m+p}^{*}\widehat\Phi_{m + p})^{-1}\widehat\Phi_{m + p}^{*}(b(\mathbf X) - b_{m + p}(\mathbf X))\|_{2,n}^{2}.
\end{eqnarray*}
This concludes the proof. $\Box$
%


%
\subsection{Proof of Proposition \ref{elaborate_risk_bound_estimator_1}}
Consider the following set
\begin{displaymath}
\Omega_m :=
\left\{|\|t\|_{n}^{2}/\|t\|_{f}^{2} - 1|\leqslant\frac{1}{2}
\textrm{ $;$ }\forall t\in\mathcal S_m\right\} =
\left\{\|\Psi_m^{-1/2}\widehat\Psi_{m}\Psi_m^{-1/2} -\mathbf I_m\|_{\normalfont{\textrm{op}}}\leqslant\frac{1}{2}\right\}.
\end{displaymath}
The proof relies on the following lemma, borrowed from Comte and Genon-Catalot \cite[Lemma 5]{CGC20}.
\begin{lemma}\label{Omega}
Under Assumption \ref{assumption_Psi}($m$), there exists a deterministic constant $\mathfrak c_{\ref{Omega}} > 0$, not depending on $m$ and $n$, such that
\begin{displaymath}
\mathbb P(\Omega_{m}^{c})
\leqslant\frac{\mathfrak c_{\ref{Omega}}}{n^8}
\quad\textrm{ and }\quad 
\mathbb P(\Lambda_{m}^{c})
\leqslant\frac{\mathfrak c_{\ref{Omega}}}{n^8}.
\end{displaymath}
\end{lemma}
\noindent
First of all,
\begin{displaymath}
\mathbb E\left[\|\widetilde b_{m}^{\prime,1} - b'\|_{n}^{2}\right] =
\mathbb E\left[\|\widehat b_{m}^{\prime,1} - b'\|_{n}^{2}\mathbf 1_{\Lambda_{m + p}}\right] 
+\mathbb E(\|b'\|_{n}^{2}\mathbf 1_{\Lambda_{m + p}^{c}}).
\end{displaymath}
Obviously, by applying Lemma \ref{Omega}, since $\mathbb E[b'(X_1)^4] <\infty$,
\begin{displaymath}
\mathbb E(\|b'\|_{n}^{2}\mathbf 1_{\Lambda_{m + p}^{c}})
\leqslant\mathbb E[b'(X_1)^4]^{1/2}\mathbb P(\Lambda_{m + p}^{c})^{1/2}
\leqslant
\mathfrak c_{\ref{Omega}}^{1/2}
\mathbb E[b'(X_1)^4]^{1/2}\frac{1}{n^4}.
\end{displaymath}
Let us dissect $\|\widehat b_{m}^{\prime,1} - b'\|_n^2{\mathbf 1}_{\Lambda_{m + p}}$ via the event $\Omega_{m + p}$:
\begin{eqnarray*}
 \mathbb E\left[\|\widehat b_{m}^{\prime,1} - b'\|_{n}^{2}\mathbf 1_{\Lambda_{m + p}}\right] & = &
 \mathbb E\left[\|\widehat b_{m}^{\prime,1} - b'\|_{n}^{2}\mathbf 1_{\Lambda_{m + p}\cap\Omega_{m + p}}\right] +
 \mathbb E\left[\|\widehat b_{m}^{\prime,1} - b'\|_{n}^{2}\mathbf 1_{\Lambda_{m + p}\cap\Omega_{m + p}^{c}}\right]\\
 & \leqslant &
 \mathbb E\left[\|\widehat b_{m}^{\prime,1} - b'\|_{n}^{2}\mathbf 1_{\Lambda_{m + p}\cap\Omega_{m + p}}\right]\\
 & &
 \hspace{1cm} +
 2\left[\mathbb E(\|\widehat b_{m}^{\prime,1}\|_{n}^{4}\mathbf 1_{\Lambda_{m + p}})^{1/2} +\mathbb E(\|b'\|_{n}^{4})^{1/2}\right]\mathbb P(\Omega_{m + p}^{c})^{1/2}
 =: S + T.
\end{eqnarray*}
On the one hand, let us find suitable bounds on the two remaining terms:
\begin{itemize}
\item For every measurable function $\psi :\mathbb R\rightarrow\mathbb R$ and $q\in [1,\infty[$ such that $\mathbb E(\psi(X_1)^{2q}) <\infty$, by Jensen's inequality,
\begin{equation}\label{elaborate_risk_bound_estimator_1_1}
 \mathbb E(\|\psi\|_{n}^{2q}) =
 \mathbb E\left[\left(\frac{1}{n}\sum_{i = 1}^{n}\psi^2(X_i)\right)^q\right]
 \leqslant
 \frac{1}{n}
 \sum_{i = 1}^{n}\mathbb E[(\psi(X_i))^{2q}] =\mathbb E(\psi^{2q}(X_1)).
\end{equation}
Then, $\mathbb E(\|b'\|_{n}^{4})\leqslant\mathbb E(b'(X_1)^4)$.
\item Recall that
\begin{displaymath}
\widehat b_{m}^{\prime,1}(\mathbf X) =
\widehat\Phi_m'(\widehat\Phi_{m}^{*}\widehat\Phi_{m})^{-1}
\widehat\Phi_{m}^{*}\mathbf Y =
\widehat\Phi_{m + p}\Delta_{m,m + p}^{*}(\widehat\Phi_{m}^{*}\widehat\Phi_{m})^{-1}
\widehat\Phi_{m}^{*}\mathbf Y.
\end{displaymath}
First,
\begin{eqnarray}
 \|\widehat\Phi_{m + p}\Delta_{m,m + p}^{*}(\widehat\Phi_{m}^{*}\widehat\Phi_{m})^{-1}
 \widehat\Phi_{m}^{*}\|_{{\rm op}}^2
 & = &
 \lambda_{\max}(\widehat\Phi_{m + p}
 \Delta_{m,m + p}^{*}(\widehat\Phi_{m}^{*}\widehat\Phi_m)^{-1}\Delta_{m,m + p}
 \widehat\Phi_{m + p}^{*})
 \nonumber\\
 & = &
 n^{-1}
 \|\widehat\Phi_{m + p}\Delta_{m,m + p}^{*}\widehat\Psi_{m}^{-1/2}\|_{{\rm op}}^{2}
 \nonumber\\
 & \leqslant &
 n^{-1}\|\widehat\Psi_{m}^{-1/2}\|_{{\rm op}}^{2}\|\widehat\Phi_{m + p}\Delta_{m,m + p}^{*}\|_{{\rm op}}^{2}
 \nonumber\\
 & = &
 \|\widehat\Psi_{m}^{-1}\|_{{\rm op}}\lambda_{\max}(\Delta_{m,m + p}
 \widehat\Psi_{m + p}\Delta_{m,m + p}^{*})
 \nonumber\\
 \label{elaborate_risk_bound_estimator_1_2}
 & \leqslant &
 \|\widehat\Psi_{m}^{-1}\|_{{\rm op}}
 \|\widehat\Psi_{m + p}\|_{{\rm op}}\|\Delta_{m,m + p}\|_{{\rm op}}^{2}.
\end{eqnarray}
Moreover, $\|\widehat\Psi_{m + p}\|_{{\rm op}}\leqslant\mathfrak L(m + p)$ and $\mathfrak L(m)\|\widehat\Psi_{m}^{-1}\|_{{\rm op}}\leqslant\mathfrak L(m + p)\|\widehat\Psi_{m + p}^{-1}\|_{{\rm op}}\leqslant\mathfrak cn/\log(n)$ on $\Lambda_{m + p}$. Then,
\begin{eqnarray*}
 \mathbb E(\|\widehat b_{m}^{\prime,1}\|_{n}^{4}\mathbf 1_{\Lambda_{m + p}})
 & \leqslant &
 \frac{1}{n^2}
 \mathbb E(\|\widehat\Psi_{m}^{-1}\|_{{\rm op}}^{2}
 \|\widehat\Psi_{m + p}\|_{{\rm op}}^{2}\mathbf 1_{\Lambda_{m + p}}
 \|\mathbf Y\|_{2,n}^{4})
 \|\Delta_{m,m + p}\|_{{\rm op}}^{4}\\
 & \leqslant &
 \frac{\mathfrak c^2n^2}{\log(n)^2}\|\Delta_{m,m + p}\|_{{\rm op}}^{4}
 \mathbb E\left[\left(\frac{1}{n}\sum_{i = 1}^{n}Y_{i}^{2}\right)^2\right]
 \leqslant
 \frac{\mathfrak c^2n^2}{\log(n)^2}\|\Delta_{m,m + p}\|_{{\rm op}}^{4}
 \mathbb E(Y_{1}^{4}).
\end{eqnarray*}
\end{itemize}
Thus, thanks to Lemma \ref{Omega},
\begin{eqnarray*}
 T & = &
 2\left[\mathbb E(\|\widehat b_{m}^{\prime,1}\|_{n}^{4}\mathbf 1_{\Lambda_{m + p}})^{1/2} +\mathbb E(\|b'\|_{n}^{4})^{1/2}\right]\mathbb P(\Omega_{m + p}^{c})^{1/2}\\
 & \leqslant &
 2\left[\frac{\mathfrak cn}{\log(n)}\|\Delta_{m,m+p}\|_{{\rm op}}^{2}
 \mathbb E(Y_{1}^{4})^{1/2} +\mathbb E(b'(X_1)^4)^{1/2}\right]\frac{\mathfrak c_{\ref{Omega}}^{1/2}}{n^4}.
\end{eqnarray*}
On the other hand, with the exact same ideas as in the proof of Proposition \ref{rough_risk_bound_estimator_1},
\begin{eqnarray*}
 & &
 S - 3\|f\|_{\infty}\inf_{t\in\mathcal S_m}\|t - b'\|^2 - 3\|b_m' - (b')_m\|_{f}^{2}\\
 & &
 \hspace{3cm}
 \leqslant
 3\mathbb E\left[\|\widehat\Phi_{m + p}\Delta_{m,m + p}^{*}(\widehat\Phi_{m}^{*}\widehat\Phi_m)^{-1}
 \widehat\Phi_{m}^{*}\|_{{\rm op}}^{2}\|b - b_m\|_{n}^{2}\mathbf 1_{\Lambda_{m + p}\cap\Omega_{m + p}}\right]\\
 & &
 \hspace{3cm}
 +\frac{\sigma^2}{n}
 \mathbb E\left[{\rm Tr}\left((\widehat\Phi_{m}^{*}\widehat\Phi_{m})^{-1}
 \Delta_{m,m + p}\widehat\Phi_{m + p}^{*}\widehat\Phi_{m + p}\Delta_{m,m + p}^{*}\right)
 \mathbf 1_{\Lambda_{m + p}\cap\Omega_{m + p}}\right]
 =: S_1 + S_2.
\end{eqnarray*}
Let us find suitable bounds on $S_1$ and $S_2$:
\begin{itemize}
 \item On $\Omega_{m + p}$, the eigenvalues of $\Psi_{m + p}^{-1/2}\widehat\Psi_{m + p}\Psi_{m + p}^{-1/2}$ belong to $[1/2,3/2]$. The same way, on $\Omega_m$, the eigenvalues of $\Psi_{m}^{-1/2}\widehat\Psi_m\Psi_m^{-1/2}$ belong to $[1/2,3/2]$ and then, those of the matrix $\Psi_{m}^{1/2}\widehat\Psi_{m}^{-1}\Psi_{m}^{1/2}$ belong to $[2/3,2]$. So, on $\Omega_{m + p}$, $\widehat S_1 :=\|\widehat\Phi_{m + p}\Delta_{m,m + p}^{*}(\widehat\Phi_{m}^{*}\widehat\Phi_m)^{-1}\widehat\Phi_{m}^{*}\|_{{\rm op}}^{2}$ satisfies
 \begin{eqnarray}
  \widehat S_1
  & = &
  \lambda_{\max}
  (\widehat\Phi_{m + p}\Delta_{m,m + p}^{*}(\widehat\Phi_{m}^{*}
  \widehat\Phi_{m})^{-1}\Delta_{m,m + p}\widehat\Phi_{m + p}^{*})
  \nonumber\\
  & = &
  \lambda_{\max}
  (\Delta_{m,m + p}^{*}
  \widehat\Psi_{m}^{-1}\Delta_{m,m + p}\widehat\Psi_{m + p})
  \nonumber\\
  & = &
  \lambda_{\max}
  (\Psi_{m + p}^{-1/2}\Delta_{m,m + p}^{f,1}\Psi_{m}^{1/2}\widehat\Psi_{m}^{-1}\Psi_{m}^{1/2}(\Delta_{m,m + p}^{f,1})^*\Psi_{m + p}^{-1/2}\widehat\Psi_{m + p})
  \nonumber\\
  & = &
  \lambda_{\max}
  ((\Psi_{m + p}^{-1/2}\widehat\Psi_{m + p}\Psi_{m + p}^{-1/2})^{1/2}\Delta_{m,m + p}^{f,1}
  \Psi_{m}^{1/2}\widehat\Psi_{m}^{-1}\Psi_{m}^{1/2}(\Delta_{m,m + p}^{f,1})^*
  (\Psi_{m + p}^{-1/2}\widehat\Psi_{m + p}\Psi_{m + p}^{-1/2})^{1/2})
  \nonumber\\
  & = &
  \|(\Psi_{m + p}^{-1/2}\widehat\Psi_{m + p}\Psi_{m + p}^{-1/2})^{1/2}\Delta_{m,m + p}^{f,1}
  \Psi_{m}^{1/2}\widehat\Psi_{m}^{-1}\Psi_{m}^{1/2}(\Delta_{m,m + p}^{f,1})^*
  (\Psi_{m + p}^{-1/2}\widehat\Psi_{m + p}\Psi_{m + p}^{-1/2})^{1/2}\|_{{\rm op}}
  \nonumber\\
  & \leqslant &
  \|\Psi_{m + p}^{-1/2}\widehat\Psi_{m + p}\Psi_{m + p}^{-1/2}\|_{{\rm op}}
  \|\Delta_{m,m + p}^{f,1}\|_{{\rm op}}^{2}
  \|\Psi_{m}^{1/2}\widehat\Psi_{m}^{-1}\Psi_{m}^{1/2}\|_{{\rm op}}
  \nonumber\\
  \label{elaborate_risk_bound_estimator_1_3}
  & \leqslant &
  3\|\Delta_{m,m + p}^{f,1}\|_{{\rm op}}^{2}.
 \end{eqnarray}
 Thus,
 \begin{displaymath}
 S_1\leqslant
 9\|\Delta_{m,m + p}^{f,1}\|_{{\rm op}}^{2}\mathbb E(\|b - b_m\|_{n}^{2}) =
 9\|\Delta_{m,m + p}^{f,1}\|_{{\rm op}}^{2}\|b - b_m\|_{f}^{2}.
 \end{displaymath}
 \item As previously, since the eigenvalues of $\Psi_{m}^{1/2}\widehat\Psi_{m}^{-1}\Psi_{m}^{1/2}\mathbf 1_{\Omega_{m + p}}$ belong to $[2/3,2]$,
 \begin{eqnarray*}
  S_2 & = &
  \frac{\sigma^2}{n}
  \mathbb E\left[{\rm Tr}\left(\Psi_{m}^{1/2}\widehat\Psi_{m}^{-1}\Psi_{m}^{1/2}(\Delta_{m,m + p}^{f,1})^{*}
  \Psi_{m + p}^{-1/2}\widehat\Psi_{m + p}\Psi_{m + p}^{-1/2}\Delta_{m,m + p}^{f,1}\right)
  \mathbf 1_{\Lambda_{m + p}\cap\Omega_{m + p}}\right]\\
  & \leqslant &
  \frac{\sigma^2}{n}
  \mathbb E\left[
  \|\Psi_{m}^{1/2}\widehat\Psi_{m}^{-1}\Psi_{m}^{1/2}\|_{{\rm op}}
  {\rm Tr}\left((\Delta_{m,m + p}^{f,1})^{*}
  \Psi_{m + p}^{-1/2}\widehat\Psi_{m + p}\Psi_{m + p}^{-1/2}\Delta_{m,m + p}^{f,1}\right)\mathbf 1_{\Omega_{m + p}}\right]\\
  & \leqslant &
  \frac{2\sigma^2}{n}
  {\rm Tr}\left[(\Delta_{m,m + p}^{f,1})^{*}
  \Psi_{m + p}^{-1/2}\mathbb E(\widehat\Psi_{m + p})\Psi_{m + p}^{-1/2}\Delta_{m,m + p}^{f,1}\right] =
  \frac{2\sigma^2}{n}\|\Delta_{m,m + p}^{f,1}\|_{F}^{2}.
 \end{eqnarray*}
\end{itemize}
The result follows by gathering all the terms. $\Box$
%


%
\subsection{Proof of Proposition \ref{elaborate_risk_bound_estimator_2}}
First of all,
\begin{displaymath}
\mathbb E\left[\|\widetilde b_{m}^{\prime,2} - b'\|_{n}^{2}\right] =
\mathbb E\left[\|\widehat b_{m}^{\prime,2} - b'\|_{n}^{2}\mathbf 1_{\Lambda_{m + p}}\right] 
+\mathbb E(\|b'\|_{n}^{2}\mathbf 1_{\Lambda_{m + p}^{c}}).
\end{displaymath}
Obviously, by applying Lemma \ref{Omega}, since $\mathbb E[b'(X_1)^4] <\infty$,
\begin{displaymath}
\mathbb E(\|b'\|_{n}^{2}\mathbf 1_{\Lambda_{m + p}^{c}})
\leqslant\mathbb E[b'(X_1)^4]^{1/2}\mathbb P(\Lambda_{m + p}^{c})^{1/2}
\leqslant
\mathfrak c_{\ref{Omega}}^{1/2}
\mathbb E[b'(X_1)^4]^{1/2}\frac{1}{n^4}.
\end{displaymath}
Let us dissect $\|\widehat b_{m}^{\prime,2} - b'\|_n^2{\mathbf 1}_{\Lambda_{m + p}}$ via the event $\Omega_{m + p}$:
\begin{eqnarray*}
 \mathbb E\left[\|\widehat b_{m}^{\prime,2} - b'\|_{n}^{2}\mathbf 1_{\Lambda_{m + p}}\right] & = &
 \mathbb E\left[\|\widehat b_{m}^{\prime,2} - b'\|_{n}^{2}\mathbf 1_{\Lambda_{m + p}\cap\Omega_{m + p}}\right] +
 \mathbb E\left[\|\widehat b_{m}^{\prime,2} - b'\|_{n}^{2}\mathbf 1_{\Lambda_{m + p}\cap\Omega_{m + p}^{c}}\right]\\
 & \leqslant &
 \mathbb E\left[\|\widehat b_{m}^{\prime,2} - b'\|_{n}^{2}\mathbf 1_{\Lambda_{m + p}\cap\Omega_{m + p}}\right]\\
 & &
 \hspace{1cm} +
 2\left[\mathbb E(\|\widehat b_{m}^{\prime,2}\|_{n}^{4}\mathbf 1_{\Lambda_{m + p}})^{1/2} +\mathbb E(\|b'\|_{n}^{4})^{1/2}\right]\mathbb P(\Omega_{m + p}^{c})^{1/2}
 =: S + T.
\end{eqnarray*}
On the one hand, let us find suitable bounds on the two remaining terms:
\begin{itemize}
\item As in the proof of Proposition \ref{elaborate_risk_bound_estimator_1}, thanks to Inequality (\ref{elaborate_risk_bound_estimator_1_1}), $\mathbb E(\|b'\|_{n}^{4})\leqslant\mathbb E(b'(X_1)^4)$.
\item Recall that
\begin{displaymath}
\widehat b_{m}^{\prime,2}(\mathbf X) =
\widehat\Phi_m\Delta_{m,m + p}(\widehat\Phi_{m + p}^{*}\widehat\Phi_{m + p})^{-1}
\widehat\Phi_{m + p}^{*}\mathbf Y.
\end{displaymath}
First,
\begin{eqnarray*}
 \|\widehat\Phi_m\Delta_{m,m + p}(\widehat\Phi_{m + p}^{*}\widehat\Phi_{m + p})^{-1}
 \widehat\Phi_{m + p}^{*}\|_{{\rm op}}^2
 & = &
 \lambda_{\max}(\widehat\Phi_m
 \Delta_{m,m + p}(\widehat\Phi_{m + p}^{*}\widehat\Phi_{m + p})^{-1}\Delta_{m,m + p}^{*}
 \widehat\Phi_{m}^{*})\\
 & = &
 n^{-1}
 \|\widehat\Phi_m\Delta_{m,m + p}\widehat\Psi_{m + p}^{-1/2}\|_{{\rm op}}^{2}\\
 & \leqslant &
 n^{-1}\|\widehat\Psi_{m + p}^{-1/2}\|_{{\rm op}}^{2}\|\widehat\Phi_m\Delta_{m,m + p}\|_{{\rm op}}^{2}\\
 & = &
 \|\widehat\Psi_{m + p}^{-1}\|_{{\rm op}}\lambda_{\max}(\Delta_{m,m + p}^{*}
 \widehat\Psi_m\Delta_{m,m + p})\\
 & \leqslant &
 \|\widehat\Psi_{m + p}^{-1}\|_{{\rm op}}
 \|\widehat\Psi_m\|_{{\rm op}}\|\Delta_{m,m + p}\|_{{\rm op}}^{2}.
\end{eqnarray*}
Moreover, $\|\widehat\Psi_m\|_{{\rm op}}\leqslant\mathfrak L(m)\leqslant\mathfrak L(m + p)$ and $\mathfrak L(m + p)\|\widehat\Psi_{m + p}^{-1}\|_{{\rm op}}\leqslant\mathfrak cn/\log(n)$ on $\Lambda_{m + p}$. Then,
\begin{eqnarray*}
 \mathbb E(\|\widehat b_{m}^{\prime,2}\|_{n}^{4}\mathbf 1_{\Lambda_{m + p}})
 & \leqslant &
 \frac{1}{n^2}
 \mathbb E(\|\widehat\Psi_{m + p}^{-1}\|_{{\rm op}}^{2}
 \|\widehat\Psi_m\|_{{\rm op}}^{2}\mathbf 1_{\Lambda_{m + p}}
 \|\mathbf Y\|_{2,n}^{4})
 \|\Delta_{m,m + p}\|_{{\rm op}}^{4}\\
 & \leqslant &
 \frac{\mathfrak c^2n^2}{\log(n)^2}\|\Delta_{m,m + p}\|_{{\rm op}}^{4}
 \mathbb E\left[\left(\frac{1}{n}\sum_{i = 1}^{n}Y_{i}^{2}\right)^2\right]
 \leqslant
 \frac{\mathfrak c^2n^2}{\log(n)^2}\|\Delta_{m,m + p}\|_{{\rm op}}^{4}
 \mathbb E(Y_{1}^{4}).
\end{eqnarray*}
\end{itemize}
Thus, thanks to Lemma \ref{Omega},
\begin{eqnarray*}
 T & = &
 2\left[\mathbb E(\|\widehat b_{m}^{\prime,2}\|_{n}^{4}\mathbf 1_{\Lambda_{m + p}})^{1/2} +\mathbb E(\|b'\|_{n}^{4})^{1/2}\right]\mathbb P(\Omega_{m + p}^{c})^{1/2}\\
 & \leqslant &
 2\left[\frac{\mathfrak cn}{\log(n)}\|\Delta_{m,m + p}\|_{{\rm op}}^{2}
 \mathbb E(Y_{1}^{4})^{1/2} +\mathbb E(b'(X_1)^4)^{1/2}\right]\frac{\mathfrak c_{\ref{Omega}}^{1/2}}{n^4}.
\end{eqnarray*}
On the other hand, with the exact same ideas than in the proof of Proposition \ref{rough_risk_bound_estimator_2},
\begin{eqnarray*}
 S - 2\|f\|_{\infty}\inf_{t\in\mathcal S_m}\|t - b'\|^2 & \leqslant &
 2\mathbb E\left[\|\widehat\Phi_m\Delta_{m,m + p}(\widehat\Phi_{m + p}^{*}\widehat\Phi_{m + p})^{-1}
 \widehat\Phi_{m + p}^{*}\|_{{\rm op}}^{2}\|b - b_{m + p}\|_{n}^{2}\mathbf 1_{\Lambda_{m + p}\cap\Omega_{m + p}}\right]\\
 & &
 +\frac{\sigma^2}{n}
 \mathbb E\left[{\rm Tr}\left((\widehat\Phi_{m + p}^{*}\widehat\Phi_{m + p})^{-1}
 \Delta_{m,m + p}^{*}\widehat\Phi_{m}^{*}\widehat\Phi_m\Delta_{m,m + p}\right)
 \mathbf 1_{\Lambda_{m + p}\cap\Omega_{m + p}}\right]\\
 & =: &
 S_1 + S_2.
\end{eqnarray*}
Let us find suitable bounds on $S_1$ and $S_2$:
\begin{itemize}
 \item On $\Omega_{m + p}$, the eigenvalues of $\Psi_{m + p}^{-1/2}\widehat\Psi_{m + p}\Psi_{m + p}^{-1/2}$ belong to $[1/2,3/2]$ and then, those of the matrix $\Psi_{m + p}^{1/2}\widehat\Psi_{m + p}^{-1}\Psi_{m + p}^{1/2}$ belong to $[2/3,2]$. The same way, on $\Omega_m$, the eigenvalues of $\Psi_{m}^{-1/2}\widehat\Psi_m\Psi_m^{-1/2}$ belong to $[1/2,3/2]$. So, on $\Omega_{m + p}$, $\widehat S_1 :=\|\widehat\Phi_m\Delta_{m,m + p}(\widehat\Phi_{m + p}^{*}\widehat\Phi_{m + p})^{-1}\widehat\Phi_{m + p}^{*}\|_{{\rm op}}^{2}$ satisfies
 \begin{eqnarray*}
  \widehat S_1
  & = &
  \lambda_{\max}
  (\widehat\Phi_m\Delta_{m,m + p}(\widehat\Phi_{m + p}^{*}
  \widehat\Phi_{m + p})^{-1}\Delta_{m,m + p}^{*}\widehat\Phi_{m}^{*})\\
  & = &
  \lambda_{\max}
  (\Delta_{m,m + p}
  \widehat\Psi_{m + p}^{-1}\Delta_{m,m + p}^{*}\widehat\Psi_m)\\
  & = &
  \lambda_{\max}
  (\Psi_{m}^{-1/2}(\Delta_{m,m + p}^{f,2})^*\Psi_{m + p}^{1/2}\widehat\Psi_{m + p}^{-1}\Psi_{m + p}^{1/2}\Delta_{m,m + p}^{f,2}\Psi_{m}^{-1/2}\widehat\Psi_m)\\
  & = &
  \lambda_{\max}
  ((\Psi_{m}^{-1/2}\widehat\Psi_m\Psi_{m}^{-1/2})^{1/2}(\Delta_{m,m + p}^{f,2})^*
  \Psi_{m + p}^{1/2}\widehat\Psi_{m + p}^{-1}\Psi_{m + p}^{1/2}\Delta_{m,m + p}^{f,2}
  (\Psi_{m}^{-1/2}\widehat\Psi_m\Psi_{m}^{-1/2})^{1/2})\\
  & = &
  \|(\Psi_{m}^{-1/2}\widehat\Psi_m\Psi_{m}^{-1/2})^{1/2}(\Delta_{m,m + p}^{f,2})^*
  \Psi_{m + p}^{1/2}\widehat\Psi_{m + p}^{-1}\Psi_{m + p}^{1/2}\Delta_{m,m + p}^{f,2}
  (\Psi_{m}^{-1/2}\widehat\Psi_m\Psi_{m}^{-1/2})^{1/2}\|_{{\rm op}}\\
  & \leqslant &
  \|\Psi_{m}^{-1/2}\widehat\Psi_m\Psi_{m}^{-1/2}\|_{{\rm op}}
  \|\Delta_{m,m + p}^{f,2}\|_{{\rm op}}^{2}
  \|\Psi_{m + p}^{1/2}\widehat\Psi_{m + p}^{-1}\Psi_{m + p}^{1/2}\|_{{\rm op}}\\
  & \leqslant &
  3\|\Delta_{m,m + p}^{f,2}\|_{{\rm op}}^{2}.
 \end{eqnarray*}
 Thus,
 \begin{displaymath}
 S_1\leqslant
 6\|\Delta_{m,m + p}^{f,2}\|_{{\rm op}}^{2}\mathbb E(\|b - b_{m + p}\|_{n}^{2}) =
 6\|\Delta_{m,m + p}^{f,2}\|_{{\rm op}}^{2}\|b - b_{m + p}\|_{f}^{2}.
 \end{displaymath}
 \item As previously, since the eigenvalues of $\Psi_{m + p}^{1/2}\widehat\Psi_{m + p}^{-1}\Psi_{m + p}^{1/2}\mathbf 1_{\Omega_{m + p}}$ belong to $[2/3,2]$,
 \begin{eqnarray*}
  S_2 & = &
  \frac{\sigma^2}{n}
  \mathbb E\left[{\rm Tr}\left(\Psi_{m + p}^{1/2}\widehat\Psi_{m + p}^{-1}\Psi_{m + p}^{1/2}\Delta_{m,m + p}^{f,2}
  \Psi_{m}^{-1/2}\widehat\Psi_m\Psi_{m}^{-1/2}(\Delta_{m,m + p}^{f,2})^*\right)
  \mathbf 1_{\Lambda_{m + p}\cap\Omega_{m + p}}\right]\\
  & \leqslant &
  \frac{\sigma^2}{n}
  \mathbb E\left[
  \|\Psi_{m + p}^{1/2}\widehat\Psi_{m + p}^{-1}\Psi_{m + p}^{1/2}\|_{{\rm op}}
  {\rm Tr}\left(\Delta_{m,m + p}^{f,2}
  \Psi_{m}^{-1/2}\widehat\Psi_m\Psi_{m}^{-1/2}(\Delta_{m,m + p}^{f,2})^*\right)\mathbf 1_{\Omega_{m + p}}\right]\\
  & \leqslant &
  \frac{2\sigma^2}{n}
  {\rm Tr}\left[\Delta_{m,m + p}^{f,2}\Psi_{m}^{-1/2}\mathbb E(\widehat\Psi_m)\Psi_{m}^{-1/2}
  (\Delta_{m,m + p}^{f,2})^{*}\right] =
  \frac{2\sigma^2}{n}\|\Delta_{m,m + p}^{f,2}\|_{F}^{2}.
 \end{eqnarray*}
\end{itemize}
The result follows by gathering all the terms. $\Box$
%


%
\subsection{Proof of Proposition \ref{bounds_f_norm}}\label{proof_bounds_f_norm}
The proof of Proposition \ref{bounds_f_norm} relies on the following general lemma.
%


%
\begin{lemma}\label{bounds_f_norm_general}
Consider $\varphi\in\mathbb L^2(I,dx)$ and let $\widehat\varphi$ be a measurable map from $\Omega\times I$ into $\mathcal S_m$ such that $\mathbb E(\|\widehat\varphi\|_{f}^{4})^{1/2}\leqslant\mathfrak mn^3$ with $\mathfrak m > 0$ not depending on $m$ and $n$. Under Assumptions \ref{assumption_f} and \ref{assumption_Psi}($m$),
\begin{displaymath}
\mathbb E(\|\widehat\varphi -\varphi\|_{f}^{2})
\leqslant 5\|f\|_{\infty}\inf_{t\in\mathcal S_m}\|t -\varphi\|^2 +
4\mathbb E(\|\widehat\varphi -\varphi\|_{n}^{2}) +
\frac{\mathfrak c_{\ref{bounds_f_norm}}(\mathfrak m,\varphi)}{n}
\end{displaymath}
with
\begin{displaymath}
\mathfrak c_{\ref{bounds_f_norm}}(\mathfrak m,\varphi) =
\sqrt{8}\mathfrak c_{\ref{Omega}}(\|\varphi\|_{f}^{2} +\mathfrak m).
\end{displaymath}
\end{lemma}
\noindent
The proof of Lemma \ref{bounds_f_norm_general} is postponed to the end of Subsection \ref{proof_bounds_f_norm}. Proposition \ref{bounds_f_norm} is obtained by applying Lemma \ref{bounds_f_norm_general} to $\varphi = b'$ and $\widehat\varphi =\widetilde b_{m}^{\prime,1}$ first, and then to $\varphi = b'$ and $\widehat\varphi =\widetilde b_{m}^{\prime,2}$. First,
\begin{eqnarray*}
 \|\widehat b_{m}^{\prime,1}\|_{f}^{2} & = &
 \int_I\left(\sum_{j = 1}^{m}[\widehat\theta_{m}^{1}]_j\varphi_j'(x)\right)^2f(x)dx
  = 
 (\widehat\theta_{m}^{1})^*\Psi_m'\widehat\theta_{m}^{1}
 \leqslant\|\Psi_m'\|_{\rm op}\|\widehat\theta_{m}^{1}\|_{2,m}^{2}
\end{eqnarray*}
with
\begin{displaymath}
\Psi_m' :=
(\langle\varphi_j',\varphi_k'\rangle_f)_{j,k} =
\frac{1}{n}\mathbb E[(\widehat{\Phi}_m')^*\widehat{\Phi}_m'] =
\Delta_{m,m + p}\Psi_{m + p}\Delta_{m,m + p}.
\end{displaymath}
Then,
\begin{displaymath}
\|\Psi_m'\|_{\rm op}\leqslant
\|\Delta_{m,m + p}\|_{\rm op}^{2}\|\Psi_{m + p}\|_{\rm op}\leqslant
\|\Delta_{m,m + p}\|_{\rm op}^{2}\mathfrak L(m + p).
\end{displaymath}
Moreover, as established in the proof of Comte and Genon-Catalot \cite{CGC20}, Proposition 5,
\begin{displaymath}
\|\widehat\theta_{m}^{1}\|_{2,m}^{2}
\leqslant\frac{1}{n}\|\widehat{\Psi}_{m}^{-1}\|_{\rm op}\|\mathbf Y\|_{2,n}^{2}
\leqslant\frac{1}{n}\|\widehat{\Psi}_{m + p}^{-1}\|_{\rm op}\|\mathbf Y\|_{2,n}^{2}
\end{displaymath}
and then, on $\Lambda_{m + p}$,
\begin{displaymath}
\|\widehat b_{m}^{\prime,1}\|_{f}^{4}
\leqslant
\|\Delta_{m,m + p}\|_{\rm op}^{4}
\frac{\mathfrak c^2}{\log(n)^2}\left(\sum_{i = 1}^{n}Y_{i}^{2}\right)^2.
\end{displaymath}
Since the $Y_i$'s are independent and $\|\Delta_{m,m + p}\|_{\rm op}^{2}\leqslant\mathfrak m_{\Delta}n^2$,
\begin{eqnarray*}
 \mathbb E(\|\widetilde b_{m}^{\prime,1}\|_{f}^{4})^{1/2}
 & \leqslant &
 \left(\|\Delta_{m,m + p}\|_{\rm op}^{4}
 \frac{\mathfrak c^2n}{\log(n)^2}\mathbb E(Y_{1}^{4})\right)^{1/2}\\
 & \leqslant &
 \mathfrak mn^3
 \quad {\rm with}\quad
 \mathfrak m =\mathfrak c\mathfrak m_{\Delta}\mathbb E(Y_{1}^{4})^{1/2}.
\end{eqnarray*}
Therefore, by Lemma \ref{bounds_f_norm_general},
\begin{displaymath}
\mathbb E(\|\widetilde b_{m}^{\prime,1} - b'\|_{f}^{2})
\leqslant 5\|f\|_{\infty}\inf_{t\in\mathcal S_m}\|t - b'\|^2 +
4\mathbb E(\|\widetilde b_{m}^{\prime,1} - b'\|_{n}^{2}) +
\frac{\mathfrak c_{\ref{bounds_f_norm}}(\mathfrak m,b')}{n}.
\end{displaymath}
The risk bound in norm $\|.\|_f$ on $\widetilde b_{m}^{\prime,2}$ is obtained via similar arguments. $\qed$
%


%
\subsubsection*{Proof of Lemma \ref{bounds_f_norm_general}}
First of all, note that
\begin{eqnarray*}
 \mathbb E(\|\widehat\varphi -\varphi\|_{f}^{2}) & = &
 \mathbb E(\|\widehat\varphi -\varphi\|_{f}^{2}\mathbf 1_{\Omega_n}) +
 \mathbb E(\|\widehat\varphi -\varphi\|_{f}^{2}\mathbf 1_{\Omega_{n}^{c}})\\
 & =: & T_1 + T_2.
\end{eqnarray*}
For any $t\in\mathbb L^2(I,f(x)dx)$, let $t^{(f)}$ be the orthogonal projection of $t$ on $\mathcal S_m$ for the theoretical norm $\|.\|_f$. On the one hand, since $\|t\|_{f}^{2}\mathbf 1_{\Omega_n}\leqslant 2\|t\|_{n}^{2}\mathbf 1_{\Omega_n}$ for every $t\in\mathcal S_m$,
\begin{eqnarray*}
 \|\widehat\varphi -\varphi\|_{f}^{2}\mathbf 1_{\Omega_n} & = &
 (\|\widehat\varphi -\varphi^{(f)}\|_{f}^{2} +
 \|\varphi^{(f)} -\varphi\|_{f}^{2})\mathbf 1_{\Omega_n}\\
 & \leqslant &
 \|\varphi^{(f)} -\varphi\|_{f}^{2} +
 2\|\widehat\varphi -\varphi^{(f)}\|_{n}^{2}\mathbf 1_{\Omega_n}\\
 & \leqslant &
 \inf_{t\in\mathcal S_m}\|t -\varphi\|_{f}^{2} +
 4\|\widehat\varphi -\varphi\|_{n}^{2} + 4\|\varphi -\varphi^{(f)}\|_{n}^{2}.
\end{eqnarray*}
Since $\mathbb E(\|\varphi -\varphi^{(f)}\|_{n}^{2}) = \|\varphi -\varphi^{(f)}\|_{f}^{2}$,
\begin{eqnarray*}
 T_1
 & \leqslant &
 5\inf_{t\in\mathcal S_m}\|t -\varphi\|_{f}^{2} +
 4\mathbb E(\|\widehat\varphi -\varphi\|_{n}^{2})\\
 & \leqslant &
 5\|f\|_{\infty}\inf_{t\in\mathcal S_m}\|t -\varphi\|^2 +
 4\mathbb E(\|\widehat\varphi -\varphi\|_{n}^{2}). 
\end{eqnarray*}
On the other hand, since $\mathbb P(\Omega_{m}^{c})\leqslant\mathfrak c_{\ref{Omega}}/n^8$ by Lemma \ref{Omega},
\begin{displaymath}
T_2\leqslant
\mathbb E(\|\widehat\varphi -\varphi\|_{f}^{4})^{1/2}\mathbb P(\Omega_{m}^{c})^{1/2}
\leqslant
\sqrt{8}[\|\varphi\|_{f}^{2} +
\mathbb E(\|\widehat\varphi\|_{f}^{4})^{1/2}]\frac{\mathfrak c_{\ref{Omega}}}{n^4}.
\end{displaymath}
Finally, the condition $\mathbb E(\|\widehat\varphi\|_{f}^{4})^{1/2}\leqslant\mathfrak mn^3$ implies that
\begin{displaymath}
\mathbb E(\|\widehat\varphi -\varphi\|_{f}^{2})
\leqslant
5\|f\|_{\infty}\inf_{t\in\mathcal S_m}\|t -\varphi\|^2 +
4\mathbb E(\|\widehat\varphi -\varphi\|_{n}^{2}) +
\frac{\sqrt{8}\mathfrak c_{\ref{Omega}}(\|\varphi\|_{f}^{2} +\mathfrak m)}{n}.
\quad\qed
\end{displaymath}
%


%
\subsection{Proof of Proposition \ref{additional_term_risk_bound_estimator_1}:}
\noindent
{\bf The Hermite case.} Consider a square integrable function $b$, and
\begin{displaymath}
b_m =\sum_{j = 0}^{m - 1}
\langle b,h_j\rangle h_j
\end{displaymath}
its projection on $\mathcal S_m = {\rm span}\{h_0,\dots,h_{m - 1}\}$. On the one hand,
\begin{displaymath}
b_m' =\sum_{j = 0}^{m - 1}
\langle b, h_j\rangle h_j'.
\end{displaymath}
Then, thanks to Equality (\ref{recursive_derivative_Hermite}),
\begin{eqnarray*}
 b_m' & = &
 \frac{1}{\sqrt 2}
 \sum_{j = 0}^{m - 1}
 \langle b,h_j\rangle(\sqrt j
 h_{j - 1} -\sqrt{j + 1}h_{j + 1})\\
 & = &
 \frac{1}{\sqrt 2}\left(
 \sum_{j = 0}^{m - 2}
 \langle b,h_{j + 1}\rangle	
 \sqrt{j + 1}h_j -\sum_{j = 1}^{m}\langle b, h_{j - 1}\rangle\sqrt jh_j\right)\\
 & = &
 \frac{1}{\sqrt 2}
 \sum_{j = 0}^{m - 2}
 \left[\sqrt{j + 1}\langle b,h_{j + 1}\rangle	
 -\sqrt j\langle b,h_{j - 1}\rangle\right]h_j
 -\left(\sqrt{\frac{m - 1}{2}}
 \langle b,h_{m - 2}\rangle h_{m - 1} +\sqrt{\frac{m}{2}}
 \langle b,h_{m - 1}\rangle h_{m}\right).
\end{eqnarray*}
On the other hand, if $b'$ is square integrable, then
\begin{displaymath}
b' =\sum_{j\geqslant 0}\langle b',h_j\rangle h_j.
\end{displaymath}
The usual integration by parts gives $\langle b',h_j\rangle = -\langle b,h_j'\rangle$ as soon as $\lim_{x\rightarrow\pm\infty}b(x)h_j(x) = 0$ (this holds because the $h_j$'s have exponential decrease and $b$ is square-integrable, thus bounded near infinity). So, the projection of $b'$ is
\begin{eqnarray*}
 (b')_m & = &
 -\sum_{j = 0}^{m - 1}\langle b,h_j'\rangle h_j\\
 & = &
 -\frac{1}{\sqrt 2}
 \sum_{j = 0}^{m - 1}
 \langle b,\sqrt jh_{j - 1} -\sqrt{j + 1}h_{j + 1}\rangle h_j\\
 & = &
 \frac{1}{\sqrt 2}
 \sum_{j = 0}^{m - 1}
 \left[\sqrt{j + 1}\langle b,h_{j + 1}\rangle -\sqrt j\langle b,h_{j - 1}\rangle\right]h_j.
\end{eqnarray*} 
All the components of $b_m'$ and $(b')_m$ are the same on $\mathcal S_{m - 2}$. So,
\begin{displaymath}
b_m' - (b')_m =
-\sqrt{\frac{m}{2}}
\langle b,h_{m - 1}\rangle h_m -\sqrt{\frac{m}{2}}\langle b,h_m\rangle h_{m - 1},
\end{displaymath}
and then,
\begin{displaymath}
\|b_m' - (b')_m\|^2 =
\frac{m}{2}\left(
\langle b,h_{m - 1}\rangle^2 +\langle b,h_m\rangle^2\right).
\end{displaymath}
If $b$ belongs to a Hermite-Sobolev space with regularity index $\alpha > 1$, then the term $\|b_m' - (b')_m\|^2$ is of order $m^{-(\alpha - 1)}$, which is also the order of $\inf_{t\in\mathcal S_m}\|t - b'\|^2$.
\\
\\
{\bf The Laguerre case.} As previously, on the one hand,
%
$b_m =
\sum_{j = 0}^{m - 1}
\langle b,\ell_j\rangle \ell_j,$
%
and thanks to (\ref{recursive_derivative_Laguerre}),
\begin{eqnarray*}
 b_m' & = &
 \sum_{j = 0}^{m - 1}
 \langle b,\ell_j\rangle
 \left(-\ell_j - 2\sum_{k = 0}^{j - 1}\ell_k\right)\\
 & = &
 -\sum_{j = 0}^{m - 1}
 \langle b,\ell_j\rangle\ell_j - 2\sum_{k = 0}^{m - 2}\left(\sum_{j = k + 1}^{m - 1}\langle b,\ell_j\rangle\right)\ell_k.
\end{eqnarray*}
On the other hand, if $b'$ is square integrable, then
%
$b' =\sum_{j\geqslant 0}\langle b',\ell_j\rangle\ell_j.$
%
Thus, since $\langle b',\ell_j\rangle = -\langle b,\ell'_j\rangle$ by Assumption \ref{assumption_bounds_I} (true when $b(0) = 0$),
\begin{displaymath}
(b')_m =\sum_{j = 0}^{m - 1}\left(
\langle b,\ell_j\rangle\ell_j - 2\sum_{k = 0}^{j - 1}\langle b,\ell_k\rangle\right)\ell_j.
\end{displaymath}
Consequently,
\begin{displaymath}
(b')_m - b_m' = 2\sum_{j = 0}^{m - 1}
\left(\sum_{k = 0}^{m - 1}\langle b,\ell_k\rangle\right)\ell_j =
2\left(\sum_{k = 0}^{m - 1}\langle b,\ell_k\rangle\right)\sum_{j = 0}^{m - 1}\ell_j,
\end{displaymath}
and then
\begin{displaymath}
\|(b')_m - b_m'\|^2 =
4m\left(\sum_{k = 0}^{m - 1}\langle b,\ell_k\rangle\right)^2.
\end{displaymath}
Moreover, by assuming that $b(0) = 0$,
\begin{displaymath}
\sum_{k\geqslant 0}
\langle b,\ell_k\rangle\ell_k(0) =
\sqrt 2\sum_{k\geqslant 0}\langle b,\ell_k\rangle = 0.
\end{displaymath}
So,
%
 $\sum_{k = 0}^{m - 1}\langle b,\ell_k\rangle =-\sum_{k\geqslant m}\langle b,\ell_k\rangle,$
%
and then
\begin{displaymath}
\|(b')_m - b_m'\|^2 =
4m\left(\sum_{k\geqslant m}
\langle b,\ell_k\rangle\right)^2.
\end{displaymath}
Finally, if $b$ belongs to a Laguerre-Sobolev space with index $\alpha > 1$, then the right-hand side in the previous equality is smaller than
\begin{displaymath}
\sum_{k\geqslant m}k^{\alpha}\langle b,\ell_k\rangle^2 = O(m^{-\alpha + 1})
\textrm{ $\Box$}
\end{displaymath}
%


%
\subsection{Proof of Lemma \ref{increase_penalty}} 
First,
\begin{displaymath}
\widehat V(m) =
\sigma^2\frac{m}{n}\|\widehat\Psi_{m}^{-1}
(\widehat\Phi_m')^*\widehat\Phi_m'\|_{\rm op}
=\sigma^2\frac{m}{n}\|\widehat\Psi_m^{-1/2}
(\widehat\Phi_m')^*\widehat\Phi_m'\widehat\Psi_m^{-1/2}\|_{\rm op}
\end{displaymath}
where $\Psi_{m}^{-1/2}$ is a symmetric square root of $\Psi_{m}^{-1}$. Now, as the matrix is symmetric,
\begin{displaymath}
\|\widehat\Psi_{m}^{-1/2}(\widehat\Phi_m')^*
\widehat\Phi_m'\widehat\Psi_{m}^{-1/2}\|_{\rm op} =
\sup_{\mathbf x\in\mathbb R^m}\mathbf x\widehat\Psi_{m}^{-1/2}
(\widehat\Phi_m')^*\widehat\Phi_m'\widehat\Psi_m^{-1/2}{\mathbf x} =
n\sup_{t\in\mathcal S_m :\|t\|_n = 1}\|t'\|_{n}^{2}.
\end{displaymath}
So, clearly, $m\mapsto\widehat V(m) =\sigma^2m/n\sup_{t\in\mathcal S_m :\|t\|_n = 1}\|t'\|_{n}^{2}$ is increasing. $\Box$
%


%
\subsection{Theorem \ref{bound_GL_estimator} and its proof}\label{sec-thm}

\subsubsection{Statement of Theorem \ref{bound_GL_estimator}}

%
\begin{theorem}\label{bound_GL_estimator}
Let Assumption \ref{assumption_f} be fulfilled. Let also Assumption \ref{assumption_Psi}($m + p$) be fulfilled for every $m\in\mathcal M_n$. Moreover, assume that there exists $\kappa > 0$ such that $\mathbb E(\exp(\kappa\varepsilon_{1}^{2})) <\infty$, that
\begin{equation}\label{bound_GL_estimator_1}
\sup_{n\in\mathbb N\backslash\{0\}}
\left\{
\frac{1}{\log(n)}
\sum_{m\leqslant n}
\frac{\mathfrak L'(m)}{\mathfrak L(m)}[
\exp(-{\tt a}_1 m) +
\exp(-{\tt a}_2\sqrt{\mathfrak L(m)})]\right\}
<\infty\textrm{ $;$ $\forall{\tt a}_1,\forall{\tt a}_2 > 0$}
\end{equation}
with
\begin{displaymath}
\mathfrak L'(m) :=\sup_{x\in I}\sum_{j = 0}^{m - 1}\varphi_j'(x)^2,
\end{displaymath}
and that there exists $q\in\mathbb N\backslash\{0\}$ such that
\begin{equation}\label{bound_GL_estimator_2}
\sup_{n\in\mathbb N\backslash\{0\}}\left\{
\frac{1}{n^{q/2}\log(n)}
\sum_{m\leqslant n}
\frac{\mathfrak L'(m)}{\mathfrak L (m)}\right\}
<\infty.
\end{equation}
Then, there exists a constant $\mathfrak c_{\ref{bound_GL_estimator}} > 0$, not depending on $n$, such that
\begin{eqnarray*}
 \mathbb E(\|\widehat b' - b'\|_{n}^{2})
 & \leqslant &
 \mathfrak c_{\ref{bound_GL_estimator}}
 \inf_{m\in\mathcal M_n}
 \left\{\mathbb E(\|\widehat b_{m}^{\prime,1} - b'\|_{n}^{2}) +\kappa_1 V(m)+
 \|\Delta_{m,m + p}^{f,1}\|_{{\rm op}}^{2}\|b - b_m\|_{f}^{2}\right.\\
 & & \hspace{2cm}
 \left. +
 \sup_{m'\in\mathcal M_{n}^{+} : m' > m}
 \left\{\|\Delta_{m',m' + p}^{f,1}\|_{\rm op}^{2}(\|b_{M_{n}^{+}} - b_{m'}\|_{f}^{2} +
 \|b - b_{M_{n}^{+}}\|_{\infty}^{2})\right\}\right.\\
 & & \hspace{2cm}
 \left. +\frac{9}{2}\|f\|_{\infty}
 \sup_{m'\in {\mathcal M}_{n}^{+} : m' > m}
 \|b_{m'}' - b_m'\|^2\right\}
 +\frac{\mathfrak c_{\ref{bound_GL_estimator}}}{n}.
\end{eqnarray*}
\end{theorem}
\noindent
Conditions (\ref{bound_GL_estimator_1}) and (\ref{bound_GL_estimator_2}) are fulfilled by all the bases we mentioned (trigonometric, Laguerre, Hermite, Legendre) because $\mathfrak L(m)$ and $\mathfrak L'(m)$ have the order of  powers of $m$. The condition on $\varepsilon_1$ is fulfilled by Gaussian random variables for any $\kappa < 1/(2\sigma^2)$, and by random variables with a compactly supported distribution. The quantity
\begin{displaymath}
\inf_{m\in\mathcal M_n}
\left\{\mathbb E(\|\widehat b_{m}^{\prime, 1} - b'\|_{n}^{2}) +\kappa_1V(m)  + 
\|\Delta_{m,m + p}^{f,1}\|_{{\rm op}}^{2}\|b - b_m\|_{f}^{2}\right\}
\end{displaymath}
has the order of the minimum risk over the estimators of the collection in this problem. The three additional terms are due to the bound on the {\it bias term}
\begin{displaymath}
\mathbb E\left(\sup_{m'\in\widehat{\mathcal M}_n}
\|\mathbb E_{\mathbf X}(\widehat b_{m\wedge m'}^{\prime,1}) -
\mathbb E_{\mathbf X}(\widehat b_{m'}^{\prime,1})\|_n^2\right). 
\end{displaymath}
Concretely, Theorem \ref{bound_GL_estimator}  can be applied some of our specific bases. 

\subsubsection{Proof of Theorem \ref{bound_GL_estimator}}

Throughout this subsection, for the sake of readability, we omit the superscript $1$ and write $\widehat b'_m$ instead of $\widehat{b}^{\prime,1}_m$.
\\
\\
Following the lines of the proof of Theorem 2 in Comte and Genon-Catalot \cite{CGC20}, we consider the sets
\begin{displaymath}
\Xi_n =\{\omega :\mathcal M_n
\subset\widehat{\mathcal M}_n(\omega)\subset {\mathcal M}_{n}^{+}\}
\quad {\rm and}\quad
\Omega_n =\bigcap_{m\in\mathcal M_{n}^{+}}\Omega_m,
\end{displaymath}
where 
\begin{displaymath}
\mathcal M_{n}^{+} :=
\left\{m\in\{1,\dots,n\} :
\mathfrak{L}(m + p)(\|\Psi_{m + p}^{-1}\|_{\rm op}^{2}\vee 1)\leqslant
4\mathfrak c\cdot\frac{n}{\log(n)}\right\}.
\end{displaymath}
First,
\begin{displaymath}
\mathbb E\left[\|\widehat b' - b'\|_{n}^{2}\mathbf 1_{(\Omega_n\cap \Xi_n)^c}\right]
\leqslant\frac{\mathfrak c_1}{n}
\quad\textrm{with}\quad\mathfrak c_1 > 0.
\end{displaymath}
This follows from the proof of Proposition \ref{elaborate_risk_bound_estimator_1}, using that $\mathbb P(\Xi_{n}^{c})\leqslant\mathfrak c_3/n^8$  and $\mathbb P(\Omega_{n}^{c})\leqslant\mathfrak c_4/n^8$. For these last probabilities, we refer to Comte and Genon-Catalot \cite{CGC20}, Lemmas 7 and 9, where the choice of $\mathfrak d = 1/[\mathfrak f (\|f\|_{\infty}\vee 1 + 3^{-1})]$ with $\mathfrak f = 192$ is explained. Here, the constant $\mathfrak f$ has to be increased to obtain the power $n^{-8}$ instead of $n^{-2}$.
\\
\\
Now, we control the loss of $\widehat b'_{\widehat m}$ on $\Omega_n\cap \Xi_n$. For any $m\in\mathcal M_n$, using that on $\Xi_n$ it also holds that $m\in \widehat{\mathcal M}_n$, we have
\begin{eqnarray}
 \nonumber
 \|\widehat b'_{\widehat m} - b'\|_{n}^{2}
 & \leqslant &
 3(\|\widehat b_{\widehat m}' -\widehat b_{\widehat m\wedge m}'\|_{n}^{2}
 +\|\widehat b_{\widehat m\wedge m}' -\widehat b_m'\|_{n}^{2}
 +\|\widehat b'_m - b'\|_{n}^{2})\\
 \nonumber
 & \leqslant &
 3(A(m) +\kappa_0\widehat V(\widehat m) +
 A(\widehat m) +\kappa_0\widehat V(m) +
 \|\widehat b_m' - b'\|_{n}^{2})\\
 \label{bound_GL_estimator_3}
 & \leqslant &
 6(A(m) +\kappa_1\widehat V(m)) + 3\|\widehat b_m' - b'\|_{n}^{2}
 \mbox{ as }
 \kappa_0\leqslant\kappa_1.
\end{eqnarray}
Moreover,
\begin{eqnarray*}
 A(m)
 & \leqslant &
 3\sup_{m\in\widehat{\mathcal M}_n}
 \left\{\|\widehat b_m' -\mathbb E_{\mathbf X}(\widehat b_m')\|_{n}^{2} -\frac{\kappa_0}{6}\widehat V(m)\right\}_+\\
 & &
 + 3\sup_{m'\in\widehat{\mathcal M}_n}
 \left\{\|\widehat b_{m\wedge m'}' -\mathbb E_{\mathbf X}(\widehat b_{m\wedge m'}')\|_{n}^{2} -\frac{\kappa_0}{6} \widehat V(m')\right\}_+\\
 & &
 + 3\sup_{m'\in\widehat{\mathcal M}_n}
 \|\mathbb E_{\mathbf X}(\widehat b_{m\wedge m'}') -\mathbb E_{\mathbf X}(\widehat b'_{m'})\|_{n}^{2},
\end{eqnarray*}
and since
\begin{displaymath}
\sup_{m'\in\widehat{\mathcal M}_n}\{\quad\cdots\quad\} =
\max\left(\sup_{m'\in\widehat{\mathcal M}_n : m'\leqslant m}\{\quad\cdots\quad\}
\textrm{ $;$ }
\sup_{m'\in\widehat{\mathcal M}_n : m'\geqslant m}\{\quad\cdots\quad\}\right),
\end{displaymath}
by Lemma \ref{increase_penalty} ($m\mapsto \widehat V(m)$ is increasing),
\begin{eqnarray*}
 & &
 \sup_{m'\in\widehat{\mathcal M}_n}
 \left\{\|\widehat b'_{m\wedge m'} -\mathbb E_{\mathbf X}(\widehat b'_{m\wedge m'})\|_{n}^{2} -\frac{\kappa_0}{6} \widehat V(m')\right\}_+\\
 & &
 \hspace{1cm}
 \leqslant
 \max\left(
 \sup_{m'\in\widehat{\mathcal M}_n}
 \left\{\|\widehat b_{m'}' -\mathbb E_{\mathbf X}(\widehat b_{m'}')\|_{n}^{2} -\frac{\kappa_0}6\widehat V(m')\right\}_+
 \textrm{ $;$ }
 \left\{\|\widehat b_m' -\mathbb E_{\mathbf X}(\widehat b_m')\|_{n}^{2} -\frac{\kappa_0}{6}\widehat V(m)\right\}_+\right)\\
 & &
 \hspace{1cm}
 \leqslant
 \sup_{m\in \widehat{\mathcal M}_n}
 \left\{\|\widehat b_m' -\mathbb E_{\mathbf X}(\widehat b_m')\|_{n}^{2} -\frac{\kappa_0}{6}\widehat V(m)\right\}_+.
\end{eqnarray*}
Thus,
\begin{equation}\label{bound_GL_estimator_4}
A(m)\leqslant
6\sup_{m\in\widehat{\mathcal M}_n}\left\{\|\widehat b_m' -\mathbb E_{\mathbf X}(\widehat b_m')\|_{n}^{2}
-\frac{\kappa_0}{6}\widehat V(m)\right\}_+ +
3 \sup_{m'\in \widehat{\mathcal M}_n}
\|\mathbb E_{\mathbf X}(\widehat b_{m\wedge m'}') - \mathbb E_{\mathbf X}(\widehat b_{m'}')\|_{n}^{2}.
\end{equation}
The following lemma provides a suitable bound on the first term in the right-hand side of Inequality (\ref{bound_GL_estimator_4}) obtained via the conditional Talagrand inequality.
%


%
\begin{lemma}\label{conditional_Talagrand}
Let Assumption \ref{assumption_f} be fulfilled. Let also Assumption \ref{assumption_Psi}($m + p$) be fulfilled for every $m\in\mathcal M_n$. Moreover, assume that there exists $\kappa > 0$ such that $\mathbb E(\exp(\kappa\varepsilon_{1}^{2})) <\infty$, and that Conditions (\ref{bound_GL_estimator_1}) and (\ref{bound_GL_estimator_2}) hold. Then,
\begin{displaymath}
\mathbb E\left[
\sup_{m\in\widehat{\mathcal M}_n}\left\{
\|\widehat b_m' -\mathbb E_{\mathbf X}(\widehat b_m')\|_{n}^{2}
-\frac{\kappa_0}{6}\widehat V(m)\right\}_+\right]
\leqslant\frac{\mathfrak c_{\ref{conditional_Talagrand}}}{n},
\end{displaymath}
where $\mathfrak c_{\ref{conditional_Talagrand}} > 0$ is a deterministic constant not depending on $n$.
\end{lemma}
\noindent
By Inequalities (\ref{bound_GL_estimator_3}) and (\ref{bound_GL_estimator_4}), and then by Lemma \ref{conditional_Talagrand}, for any $m\in {\mathcal M}_n$, 
\begin{eqnarray*}
 \mathbb E(\|\widehat b' - b'\|_{n}^{2}\mathbf 1_{\Omega_n\cap\Xi_n})
 & \leqslant &
 36\mathbb E\left[\sup_{m\in\widehat{\mathcal M}_n}\left\{\|\widehat b_m' -\mathbb E_{\mathbf X}(\widehat b_m')\|_{n}^{2}
 -\frac{\kappa_0}{6}\widehat V(m)\right\}_+\right]\\
 & & +
 18\mathbb E\left[\mathbf 1_{\Xi_n\cap \Omega_n}\sup_{m'\in \widehat{\mathcal M}_n}
 \|\mathbb E_{\mathbf X}(\widehat b_{m\wedge m'}') - \mathbb E_{\mathbf X}(\widehat b_{m'}')\|_{n}^{2}\right]\\
 & &
 \hspace{4cm}
 + 6\kappa_1\mathbb E(\widehat V(m)\mathbf 1_{\Omega_n})
 + 3\mathbb E(\|\widehat b_m' - b'\|_{n}^{2})\\
 & \leqslant &
 3\mathbb E(\|\widehat b_m' - b'\|_{n}^{2})
 + 6\kappa_1\mathfrak c_2V(m) +\frac{\mathfrak c_{\ref{conditional_Talagrand}}}{n}\\
 & &
 + 18\mathbb E\left[\mathbf 1_{\Xi_n\cap\Omega_n}\sup_{m'\in \widehat{\mathcal M}_n}
 \|\mathbb E_{\mathbf X}(\widehat b_{m\wedge m'}') - \mathbb E_{\mathbf X}(\widehat b_{m'}')\|_{n}^{2}\right]
\end{eqnarray*}
with $\mathfrak c_2 > 0$. The inequality $\mathbb E(\widehat V(m)\mathbf 1_{\Omega_n})\leqslant\mathfrak c_2V(m)$ is obtained via the same method than in the proof of Proposition \ref{elaborate_risk_bound_estimator_2}. 
\\
\\
Let us now control
\begin{displaymath}
B_{m,n} :=
\frac{1}{n}
\mathbb E\left[\mathbf 1_{\Xi_n\cap \Omega_n}
\sup_{m'\in\widehat{\mathcal M}_n : m' > m}
\|\mathbb E_{\mathbf X}(\widehat b_{m}^{\prime,1}(\mathbf X)) -\mathbb E_{\mathbf X}(\widehat b_{m'}^{\prime,1}(\mathbf X))\|_{2,n}^{2}\right].
\end{displaymath}
Thanks to Equality (\ref{rough_risk_bound_estimator_1_1}), 
\begin{eqnarray*}
 \mathbb E_{\mathbf X}(\widehat b_{m'}^{\prime,1}(\mathbf X)) -
 \mathbb E_{\mathbf X}(\widehat b_{m}^{\prime,1}(\mathbf X)) & = &
 \mathbb E_{\mathbf X}(\widehat b_{m'}^{\prime,1}(\mathbf X)) - b_{m'}'(\mathbf X)
 - (\mathbb E_{\mathbf X}(\widehat b_{m}^{\prime,1}(\mathbf X)) - b_m'(\mathbf X))
 + b_{m'}'(\mathbf X) - b_m'(\mathbf X)\\
 & = &
 \widehat P_{m'}(b(\mathbf X) - b_{m'}(\mathbf X)) +
 \widehat P_{m}(b(\mathbf X) - b_{m}(\mathbf X)) +
 b_{m'}'(\mathbf X) - b_m'(\mathbf X).
\end{eqnarray*}
Since we are on $\Omega_n$, and since $m,m'\in\mathcal M_{n}^{+}$ on $\Xi_n$ when $m\in\mathcal M_n$ and $m'\in\widehat{\mathcal M}_n$,
\begin{displaymath}
\textrm{Sp}\left[\Psi_{m' + p}^{-1/2}\widehat\Psi_{m' + p}\Psi_{m' + p}^{-1/2}\right]
\subset [1/2,3/2]
\textrm{ and }
\textrm{Sp}\left[\Psi_{m'}^{1/2}\widehat\Psi_{m'}^{-1}\Psi_{m'}^{1/2}\right]
\subset [2/3,2],
\end{displaymath}
and the same for $m$ instead of $m'$. So, thanks to Inequality (\ref{elaborate_risk_bound_estimator_1_3}),
\begin{displaymath}
\|\widehat P_{m'}\|_{\rm op}^{2}
\leqslant
3\|\Delta_{m',m' + p}^{f,1}\|_{\rm op}^{2}.
\end{displaymath}
In the same way,
\begin{displaymath}
\|\widehat P_m\|_{\rm op}^{2}
\leqslant
3\|\Delta_{m,m + p}^{f,1}\|_{\rm op}^{2}.
\end{displaymath}
Thus, on $\Omega_n$,
\begin{eqnarray*}
 \frac{1}{n}
 \|\widehat P_{m'}(b(\mathbf X) - b_{m'}(\mathbf X))\|_{2,n}^{2}
 & \leqslant &
 3\|\Delta_{m',m' + p}^{f,1}\|_{{\rm op}}^{2}
 \|b - b_{m'}\|_{n}^{2}\\
 & \leqslant &
 6\|\Delta_{m',m' + p}^{f,1}\|_{{\rm op}}^{2}
 (\|b - b_{M_n}\|_{n}^{2} +\|b_{M_n} - b_{m'}\|_{n}^{2})\\
 & \leqslant &
 6\|\Delta_{m',m' + p}^{f,1}\|_{{\rm op}}^{2}
 (\|b - b_{M_{n}^{+}}\|_{\infty}^{2} + 3/2\|b_{M_{n}^{+}} - b_{m'}\|_{f}^{2})
\end{eqnarray*}
where $M_{n}^{+}$ is the maximal element of $\mathcal M_{n}^{+}$, 
and 
\begin{displaymath}
\frac{1}{n}
\|\widehat P_m(b(\mathbf X) - b_m(\mathbf X))\|_{2,n}^{2}
\leqslant
3\|\Delta_{m,m + p}^{f,1}\|_{\rm op}^{2}
\|b - b_m\|_{n}^{2}.
\end{displaymath}
For the last term, on $\Omega_n$,
\begin{displaymath}
\|b_{m'}'(\mathbf X) - b_m'(\mathbf X)\|_{n}^{2}
\leqslant\frac{3}{2}\|b_{m'}' - b_m'\|_{f}^{2}
\leqslant\frac{3}{2}\|f\|_{\infty}\|b_{m'}' - b_m'\|^2.
\end{displaymath}
Therefore,
\begin{eqnarray*}
 B_{m,n} & \leqslant &
 9\|\Delta_{m,m + p}^{f,1}\|_{\rm op}^{2}\|b - b_m\|_{f}^{2} +
 18\sup_{m'\in\mathcal M_{n}^{+} : m' > m}
 \left\{\|\Delta_{m',m' + p}^{f,1}\|_{\rm op}^{2}(\|b_{M_{n}^{+}} - b_{m'}\|_{f}^{2} +
 \|b - b_{M_{n}^{+}}\|_{\infty}^{2})\right\}\\
 & &
 +\frac{9}{2}\|f\|_{\infty}
 \sup_{m'\in\mathcal M_{n}^{+} : m' > m}
 \|b_{m'}' - b_m'\|^2.
\end{eqnarray*}
This concludes the proof. $\Box$
%


%
\subsection{Proof of Lemma \ref{conditional_Talagrand}}
We emphasize that the lemma would be true for $\widehat{\mathcal M}_n$ replaced by the weaker (and more natural)
\begin{displaymath}
\left\{m\in\{1,\dots,n\} :
\mathfrak{L}(m + p)(\|\widehat\Psi_{m + p}^{-1}\|_{\rm op}\vee 1)\leqslant
\mathfrak c\frac{n}{\log(n)}\right\}
\end{displaymath}
with $\mathfrak c$ defined in Assumption \ref{assumption_Psi}($m + p$). We only use this constraint in the following.
\\
\\
First of all, for any $m\in\widehat{\mathcal M}_n$, since $\|\psi\|_{n}^{2} =\sup_{t\in\mathcal S_m :\|t\|_n = 1}\langle t,\psi\rangle_{n}^{2}$ for every $\psi\in\mathcal S_m$,
\begin{displaymath}
\|\widehat b_m' -\mathbb E_{\mathbf X}(\widehat b_m')\|_{n}^{2} =
\sup_{t\in\mathcal S_m :\|t\|_n = 1}\nu_n(t)^2
\end{displaymath}
with, for any $\vec b = (b_1,\dots,b_m)\in\mathbb R^m$ and $t =\sum_{j = 1}^{m}b_j\varphi_j$,
\begin{displaymath}
\nu_n(t) =
\frac{1}{n}\langle t,\widehat\Phi_m'\widehat\Psi_{m}^{-1}\widehat\Phi_{m}^{*}\varepsilon\rangle_n =
\frac{1}{n^2}\langle [\widehat\Phi_m'\widehat\Psi_{m}^{-1}\widehat\Phi_{m}^{*}]^*t(\mathbf X),\varepsilon\rangle_{2,n} =
\frac{1}{n}\langle\Theta_t(\mathbf X),\varepsilon\rangle_{2,n},
\end{displaymath}
where
\begin{displaymath}
\Theta_t(\mathbf X) =
\frac{1}{n}\widehat\Phi_m\widehat\Psi_{m}^{-1}(\widehat\Phi_m')^*t(\mathbf X) =
\frac{1}{n}\widehat\Phi_m\widehat\Psi_{m}^{-1}(\widehat\Phi_m')^*\widehat\Phi_m\vec b.
\end{displaymath}
Note that $\nu_n(t) =\nu_{n}^{(1)}(t) +\nu_{n}^{(2)}(t)$, where $\nu_{n}^{(1)}(t) = n^{-1}\langle\Theta_t(\mathbf X),\varepsilon^{(1)}\rangle_{2,n}$ with $\varepsilon^{(1)} = (\varepsilon_i\mathbf 1_{|\varepsilon_i|\leqslant\mathfrak m_n} - {\mathbb E}(\varepsilon_i\mathbf 1_{|\varepsilon_i|\leqslant\mathfrak m_n}))_i$ and $\mathfrak m_n =(q\kappa^{-1}\log(n))^{1/2}$, and $\nu_{n}^{(2)}(t) = n^{-1}\langle\Theta_t(\mathbf X),\varepsilon^{(2)}\rangle_{2,n}$ with $\varepsilon^{(2)} = (\varepsilon_i\mathbf 1_{|\varepsilon_i| >\mathfrak m_n} -{\mathbb E}(\varepsilon_i\mathbf 1_{|\varepsilon_i| >\mathfrak m_n}) )_i$. On the one hand, in order to apply Talagrand's inequality to $\sup_{t\in\mathcal S_m :\|t\|_n = 1}\nu_{n}^{1}(t)^2$ conditionally to $(X_1,\dots,X_n)$, consider
\begin{eqnarray*}
 A_{m,n}(\mathbf X) & := &
 \mathbb E_{\mathbf X}\left(
 \sup_{t\in\mathcal S_m :\|t\|_n = 1}\nu_{n}^{(1)}(t)^2\right),\\
 B_{m,n} & := &
 \sup_{t\in\mathcal S_m :\|t\|_n = 1}
 \left\{\sup_{(e,x)\in [-2\mathfrak m_n,2\mathfrak m_n]\times I}|e\Theta_t(x)|
 \right\}
 \textrm{ and}\\
 C_{m,n}(\mathbf X) & := &
 \sup_{t\in\mathcal S_m :\|t\|_n = 1}
 \left\{
 \frac{1}{n}\sum_{i = 1}^{n}\textrm{var}_{\mathbf X}\left[
 \varepsilon_{i}^{(1)}\Theta_t(X_i)\right]
 \right\},
\end{eqnarray*}
and let us find suitable bounds on each of these random quantities.
\begin{itemize}
 \item\textbf{Bound on $A_{m,n}(\mathbf X)$.} Note that
 \begin{displaymath}
 \textrm{var}(\varepsilon_{1}^{(1)})
 \leqslant
 \mathbb E(\varepsilon_{1}^{2}\mathbf 1_{|\varepsilon_1|\leqslant\mathfrak m_n})
 \leqslant
 \mathbb E(\varepsilon_{1}^{2}) =
 \sigma^2.
 \end{displaymath}
 Then,
 \begin{eqnarray*}
  A_{m,n}(\mathbf X)
  & \leqslant &
  \frac{1}{n^3}\mathbb E_{\mathbf X}(\|\widehat\Phi_m'\widehat\Psi_{m}^{-1}\widehat\Phi_{m}^{*}\varepsilon^{(1)}\|_{2,n}^{2})
  =\frac{\textrm{var}(\varepsilon_{1}^{(1)})}{n}
  {\rm Tr}\left[\widehat\Phi_m'\widehat\Psi_{m}^{-1}(\widehat\Phi_m')^*\right]\\
  & \leqslant &
  \frac{\sigma^2}{n^2}
  {\rm Tr}\left[\widehat\Phi_m'\widehat\Psi_{m}^{-1}(\widehat\Phi_m')^*\right]
  \leqslant
  \frac{\sigma^2m}{n^2}
  \|\widehat\Psi_m^{-1}(\widehat \Phi_m')^*\widehat\Phi_m'\|_{{\rm op}} =: H^2.
 \end{eqnarray*}
 \item\textbf{Bound on $B_{m,n}$.} Since $m\in\widehat{\mathcal M}_n$, $m\|\widehat\Psi_{m}^{-1}\|_{{\rm op}}\leqslant (m + p)\|\widehat\Psi_{m + p}^{-1}\|_{{\rm op}}\leqslant\mathfrak cn/\log(n)$, and then
 \begin{eqnarray*}
  B_{m,n} & \leqslant &
  \frac{2\mathfrak m_n}{n}\sup_{\vec b :\|\widehat\Phi_m\vec b\|_{2,n} =\sqrt n}
  \left|\sum_{j = 1}^{m}[\widehat\Psi_{m}^{-1}(\widehat\Phi_m')^*\widehat\Phi_m\vec b]_j\varphi_j(x)\right|\\
  & \leqslant &
  \frac{2\mathfrak m_n}{n}\mathfrak L(m)^{1/2}
  \sup_{\vec b :\|\widehat\Phi_m\vec b\|_{2,n} =\sqrt n}
  \|\widehat\Psi_{m}^{-1}(\widehat\Phi_m')^*\widehat\Phi_m\vec b\|_{m,2}
  \leqslant
  \frac{2\mathfrak m_n}{\sqrt n}
  \sqrt{\mathfrak L(m)
  \|\widehat\Psi_{m}^{-1}(\widehat\Phi_m')^*\|_{{\rm op}}^{2}}\\
  & \leqslant &
  \frac{2\mathfrak m_n}{\sqrt n}
  \sqrt{\mathfrak L(m)
  \|\widehat\Psi_{m}^{-1}\|_{{\rm op}}
  \|\widehat\Psi_{m}^{-1}(\widehat\Phi_m')^*\widehat\Phi_m'\|_{{\rm op}}}
  \leqslant
  2\sqrt{\mathfrak cq\kappa^{-1}\|\widehat\Psi_{m}^{-1}(\widehat\Phi_m')^*\widehat\Phi_m'\|_{{\rm op}}} =: M.
 \end{eqnarray*}
 Then,
 \begin{displaymath}
 \frac{nH}{M} =\frac{\sigma}{2\sqrt{\mathfrak cq\kappa^{-1}}}\cdot\sqrt{m}.
 \end{displaymath}
 \item\textbf{Bound on $C_{m,n}(\mathbf X)$:}
 \begin{eqnarray*}
  C_{m,n}(\mathbf X) & \leqslant &
  \frac{\mathbb E(|\varepsilon_{1}^{(1)}|^2)}{n}
  \sup_{t\in\mathcal S_m :\|t\|_n = 1}
  \sum_{i = 1}^{n}\Theta_t(X_i)^2
  \leqslant\frac{\sigma^2}{n^3}
  \sup_{\vec b :\|\widehat\Phi_m\vec b\|_{2,n} =\sqrt n}
  \|\widehat\Phi_m\widehat\Psi_{m}^{-1}(\widehat\Phi_m')^*\widehat\Phi_m\vec b\|_{2,n}^{2}\\
  & = &
  \frac{\sigma^2}{n^3}
  \sup_{\vec b :\|\widehat\Phi_m\vec b\|_{2,n} =\sqrt n}
  \vec b^*\widehat\Phi_{m}^{*}\widehat\Phi_m'\widehat\Psi_{m}^{-1}
  \underbrace{\widehat\Phi_{m}^{*}\widehat\Phi_m\widehat\Psi_{m}^{-1}}_{= n\mathbf I_m}
  (\widehat\Phi_m')^*\widehat\Phi_m\vec b\\
  & = &
  \frac{\sigma^2}{n^2}
  \sup_{\vec b :\|\widehat\Phi_m\vec b\|_{2,n} =\sqrt n}
  \|\widehat\Psi_{m}^{-1/2}(\widehat\Phi_m')^*\widehat\Phi_m\vec b\|_{2,n}^{2}
  \leqslant
  \frac{\sigma^2}{n}\|\widehat\Psi_{m}^{-1/2}(\widehat\Phi_m')^*\|_{{\rm op}}^{2}
  =\frac{\sigma^2}{n}\|\widehat\Phi_m'\widehat\Psi_{m}^{-1}(\widehat\Phi_m')^*\|_{{\rm op}} =: v.
 \end{eqnarray*}
 Then,
 \begin{displaymath}
 \frac{nH^2}{v} = m.
 \end{displaymath}
\end{itemize}
So, by Talagrand's inequality,
\begin{eqnarray*}
 & &
 \mathbb E_{\mathbf X}\left[
 \left(\sup_{t\in\mathcal S_m :\|t\|_n = 1}\nu_{n}^{(1)}(t)^2 - 4H^2\right)_+\right]\\
 & &
 \hspace{2cm}
 \leqslant
 \mathfrak c_1\left[\frac{v}{n}\exp\left(-\mathfrak c_2\frac{nH^2}{v}\right)
 +\frac{M^2}{n^2}\exp\left(-\mathfrak c_3\frac{nH}{M}\right)\right]\\
 & &
 \hspace{2cm}
 \leqslant
 \frac{\overline{\mathfrak c}_1}{n^2}
 \|\widehat\Psi_{m}^{-1}\|_{{\rm op}}\mathfrak L'(m)
 \left[e^{-\mathfrak c_2m}
 + q\exp\left(-\frac{\overline{\mathfrak c}_3}{\sqrt q}\cdot \sqrt{\mathfrak L(m)}\right)\right]
\end{eqnarray*}
where $\mathfrak c_1,\mathfrak c_2,\mathfrak c_3,\overline{\mathfrak c}_1,\overline{\mathfrak c}_3 > 0$ are universal constants, and thus
\begin{eqnarray*}
 S & := &
 \mathbb E_{\mathbf X}\left[
 \sup_{m\in\widehat{\mathcal M}_n}\left\{\sup_{t\in\mathcal S_m :\|t\|_n = 1}\nu_{n}^{(1)}(t)^2
 -\frac{\kappa_0}{6}\widehat V(m)\right\}_+\right]\\
 & \leqslant &
 \frac{\overline{\mathfrak c}_1}{n^2}
 \sum_{m\in\widehat{\mathcal M}_n}\left[
 \|\widehat\Psi_{m}^{-1}\|_{{\rm op}}\mathfrak L'(m)
 \left[e^{-\mathfrak c_2m}
 + q\exp\left(-\frac{\overline{\mathfrak c}_3}{\sqrt q}\cdot\sqrt{\mathfrak L(m)}\right)\right]\right]\\
 & \leqslant &
 \frac{\mathfrak c\overline{\mathfrak c}_1}{n\log(n)}
 \sum_{m\leqslant n}\left[
 \frac{\mathfrak L'(m)}{\mathfrak L(m)}
 \left[e^{-\mathfrak c_2m}
 + q\exp\left(-\frac{\overline{\mathfrak c}_3}{\sqrt q}\cdot\sqrt{\mathfrak L(m)}\right)\right]\right]
\end{eqnarray*}
thanks to the definition of $\widehat{\mathcal M}_n$. Thanks to Condition (\ref{bound_GL_estimator_1}), this term is of order $1/n$. On the other hand, since $\mathfrak L(m)\|\widehat\Psi_{m}^{-1}\|_{{\rm op}}\leqslant\mathfrak cn/\log(n)$ for every $m\in\widehat{\mathcal M}_n$, and by Markov's inequality,
\begin{eqnarray*}
 T & := &
 \mathbb E_{\mathbf X}\left(
 \sup_{m\in\widehat{\mathcal M}_n}
 \sup_{t\in\mathcal S_m :\|t\|_n = 1}\nu_{n}^{(2)}(t)^2
 \right)
 \leqslant
 \frac{1}{n^2}
 \mathbb E(\|\varepsilon^{(2)}\|_{2,n}^{2})
 \sup_{m\in\widehat{\mathcal M}_n}
 \sup_{t\in\mathcal S_m :\|t\|_n = 1}
 \|\Theta_t(\mathbf X)\|_{2,n}^{2}\\
 & \leqslant &
 \frac{1}{n}\mathbb E(\varepsilon_{1}^{4})^{1/2}\mathbb P(|\varepsilon_1| >\mathfrak m_n)^{1/2}
 \sup_{m\in\widehat{\mathcal M}_n}
 \left\{\frac{1}{n^2}\sup_{\vec b :\|\widehat\Phi_m\vec b\|_{2,n} =\sqrt n}
 \|\widehat\Phi_m\widehat\Psi_{m}^{-1}(\widehat\Phi_m')^*\widehat\Phi_m\vec b\|_{2,n}^{2}\right\}\\
 & \leqslant &
 \frac{1}{n}\mathbb E(\varepsilon_{1}^{4})^{1/2}
 \mathbb P(\exp(\kappa\varepsilon_{1}^{2}) > n^q)^{1/2}
 \sup_{m\in\widehat{\mathcal M}_n}
 \|\widehat\Phi_m'\widehat\Psi_{m}^{-1}(\widehat\Phi_m')^*\|_{{\rm op}}\\
 & \leqslant &
 \frac{1}{n}\mathbb E(\varepsilon_{1}^{4})^{1/2}
 \frac{1}{n^{q/2}}\mathbb E(\exp(\kappa\varepsilon_{1}^{2}))^{1/2}
 \frac{\mathfrak cn}{\log(n)}
 \sum_{m\leqslant n}\frac{\mathfrak L'(m)}{\mathfrak L (m)}
 \leqslant
 \frac{\mathfrak c_4}{n^{q/2}\log(n)}
 \sum_{m\leqslant n}\frac{\mathfrak L'(m)}{\mathfrak L (m)}
\end{eqnarray*}
with $\mathfrak c_4 =\mathfrak c\mathbb E(\varepsilon_{1}^{4})^{1/2}\mathbb E(\exp(\kappa\varepsilon_{1}^{2}))^{1/2}$. Thanks to Condition (\ref{bound_GL_estimator_2}), this term is of order $1/n$. In conclusion,
\begin{displaymath}
\mathbb E\left[
\sup_{m\in\widehat{\mathcal M}_n}\left\{
\|\widehat b_{m}^{\prime,1} -\mathbb E_{\mathbf X}(\widehat b_{m}^{\prime,1})\|_{n}^{2}
-\frac{\kappa_0}{3}\widehat V(m)\right\}_+\right]
\leqslant
2\mathbb E(S) + 2\mathbb E(T)\leqslant\frac{\mathfrak c_5}{n}.
\quad\Box
\end{displaymath}
%


%
\subsection{Proof of Proposition \ref{rate_trigonometric_GL}}\label{subsection_rate_trigonometric_GL}
Here, $(\varphi_j)_{j\in\mathbb N^*}$ is the trigonometric basis. Thus, with $\mathfrak L(m)$ of order $m$ and $\mathfrak L'(m)$ of order $m^3$, Conditions (\ref{bound_GL_estimator_1}) and (\ref{bound_GL_estimator_2}) are obviously fulfilled. Moreover, under $f(x)\geqslant f_0 > 0$, we know that $\|\Psi_{m}^{-1}\|_{\rm op}\leqslant 1/f_0$, and then $M_{n}^{+}$ has order $n/\log(n)$. We take $M_{n}^{+} = n/\log(n)$ for simplicity. The first terms of the bound in Theorem \ref{bound_GL_estimator} have been already evaluated in the proof of Corollary \ref{rate_trigonometric}, so we have to study the additional ones:
\begin{displaymath}
\sup_{m < m'\leqslant n/\log(n)}
\left\{\|\Delta_{m',m' + p}^{f,1}\|_{\rm op}^{2}
(\|b_{M_{n}^{+}} - b_{m'}\|_{f}^{2} +\|b - b_{M_{n}^{+}}\|_{\infty}^{2})\right\}
\quad {\rm and}\quad
\sup_{m < m'\leqslant n/\log(n)}
\|b_{m'}' - b_m'\|^2.
\end{displaymath}
We assume that $\sum_j\langle b,\varphi_j\rangle^2 j^{2\beta}\leqslant L$ with $\beta\in\mathbb N\cap (1,\infty)$. First,
\begin{eqnarray*}
 \sup_{m < m'\leqslant n/\log(n)}
 \|\Delta_{m',m' + p}^{f,1}\|_{\rm op}^{2}
 \|b_{M_{n}^{+}} - b_{m'}\|_{f}^{2}
 & \leqslant &
 \frac{\|f\|_\infty }{f_0}
 \sup_{m < m'\leqslant n/\log(n)}(m')^2
 \sum_{j\geqslant m'}\langle b,\varphi_j\rangle^2
  \leqslant 
 \frac{\|f\|_{\infty}}{f_0}m^{-2(\beta-1)}.
\end{eqnarray*}
So, this term is of same order than the bias term. Next,
\begin{displaymath}
\sup_{m < m'\leqslant n/\log(n)}
\|b_{m'}' - b_m'\|^2\leqslant
\sum_{j\geqslant m}[(2\pi j)\langle b,\varphi_j\rangle]^2\leqslant
Cm^{-2(\beta -1)}.
\end{displaymath}
Lastly, for any $m$ and $x\in I$,
\begin{eqnarray*}
 |(b - b_m)(x)|
 & \leqslant &
 \sqrt 2\sum_{j\geqslant m}|\langle b,\varphi_j\rangle|
 \leqslant
 \sqrt 2\left(\sum_{j\geqslant m}j^{2\beta}\langle b,\varphi_j\rangle^2
 \sum_{j\geqslant m}j^{-2\beta}\right)^{1/2}\\
 &\leqslant &
 \sqrt{\frac{2L}{2\beta - 1}}m^{-\beta + 1/2}
 \leqslant\mathfrak c(\beta,L)m^{-\beta + 1/2},
\end{eqnarray*}
which gives $\|b - b_m\|_{\infty}\leqslant\mathfrak c(\beta,L)m^{-\beta + 1/2}$ and
\begin{displaymath}
\sup_{m < m'\leqslant n/\log(n)}
\|\Delta_{m',m' + p}^{f,1}\|_{\rm op}^{2}
\|b - b_{M_{n}^{+}}\|_{\infty}^{2}\leqslant
\frac{\mathfrak c(\beta, L)}{f_0}\left(\frac{n}{\log(n)}\right)^{-2\beta + 3}.
\end{displaymath}
We have $n^{-2\beta + 3}\leqslant n^{-2(\beta - 1)/(2\beta + 1)}$ as soon as $\beta\geqslant (3 +\sqrt{13})/4\simeq 1.65$, which holds true when $\beta\in\mathbb N\cap (1,\infty)$. In conclusion, this together with Theorem \ref{bound_GL_estimator} and the orders given in Section \ref{trigo-order} gives the announced result. $\Box$
%


%
\subsection{Proof of Proposition \ref{rate_Hermite_GL}}
Here, $(\varphi_j)_{j\in\mathbb N^*}$ is the Hermite basis, and for $\|\Psi_{m}^{-1}\|_{\rm op} = m^{\gamma}$, the constraint on the collection of models implies $M_{n}^{+} = n^{1/(2\gamma + 1/2)}$. This is compatible with the choice of $m_{\rm opt} = 1/n^{s + 1/2}$ as $s > 2\gamma + 9/4 > 2\gamma$. Again, we have to study the orders of
 the additional terms of the bound in Theorem \ref{bound_GL_estimator}:
\begin{displaymath}
\sup_{m < m'\leqslant n/\log(n)}
\left\{\|\Delta_{m',m' + p}^{f,1}\|_{\rm op}^{2}
(\|b_{M_{n}^{+}} - b_{m'}\|_{f}^{2} +\|b - b_{M_{n}^{+}}\|_{\infty}^{2})\right\}
\quad {\rm and}\quad
\sup_{m < m'\leqslant n/\log(n)}
\|b_{m'}' - b_m'\|^2
\end{displaymath}
under the regularity condition $b\in W_{s}^{H}(L)$ with $s > 2\gamma + 9/4 > 1$. First,
\begin{eqnarray*}
 \sup_{m < m'\leqslant M_{n}^{+}}
 \|\Delta_{m',m' + p}^{f,1}\|_{\rm op}^{2}\|b_{M_n^+} - b_{m'}\|_{f}^{2}
 & \leqslant &
 \|f\|_{\infty}\sup_{m < m'\leqslant M_{n}^{+}}(m')^{\gamma + 1}
 \sum_{j\geqslant m'}\langle b,\varphi_j\rangle^2\\
 & \leqslant &
 \|f\|_{\infty}m^{-s +\gamma + 1}.
\end{eqnarray*}
So, this term is of same order than the bias term. Next, using Formula (\ref{recursive_derivative_Hermite}),
\begin{displaymath}
\sup_{m < m'\leqslant M_{n}^{+}}
\|b_{m'}' - b_m'\|^2\lesssim
\sum_{j\geqslant m}
j\langle b,\varphi_j\rangle^2\leqslant Cm^{-s + 1}.
\end{displaymath}
This term is also of same order than the first bias term and is negligible with respect to the previous one. Lastly, $\|b - b_m\|_{\infty}^{2}\leqslant C(s,L)\pi^{-1/2}m^{-s + 1}$, and thus
\begin{displaymath}
\sup_{m < m'\leqslant M_{n}^{+}}
\left\{\|\Delta_{m',m' + p}^{f,1}\|_{\rm op}^{2}
\|b - b_{M_{n}^{+}}\|_{\infty}^{2}\right\}\lesssim n^{-(s -\gamma - 2)/(2\gamma + 1/2)}
\end{displaymath}
by using the value of $M_{n}^{+}$. We have
\begin{displaymath}
-\frac{s -\gamma - 2}{2\gamma + 1/2}
\leqslant -\frac{s - 1 -\gamma}{s + 1/2}
\quad {\rm if}\quad
(s -\gamma - 2)\left(s +\frac{1}{2}\right) -
(s -\gamma - 1)\left(2\gamma +\frac{1}{2}\right) > 0.
\end{displaymath}
Since
\begin{displaymath}
(s -\gamma - 2)\left(s +\frac{1}{2}\right) -
(s -\gamma - 1)\left(2\gamma +\frac{1}{2}\right) =
(s -\gamma)(s - 2\gamma - 2) -\frac{1}{2},
\end{displaymath}
$s -\gamma > 2$ and $s - 2\gamma - 2 > 1/4$, the constraint is fulfilled and this last term is negligible with respect to the rate. Considering the orders obtained in section \ref{Explicit_Hermite}, we get the result. $\Box$

\end{document}